\pdfoutput=1 
\documentclass[10pt,a4paper,reqno]{amsart} 
\usepackage[marginratio=1:1,height=725pt,width=520pt,tmargin=70pt]{geometry} 

\usepackage{etoolbox}

\usepackage[
	backend=biber,
	citestyle=alphabetic,
	bibstyle=alphabetic,
	sorting=anyvt,
	backref=true,
	backrefstyle=none,
	]{biblatex}
\renewbibmacro{in:}{\ifentrytype{article}{}{\printtext{\bibstring{in}\intitlepunct}}} 
\usepackage{url}

\newrobustcmd{\mklocalfilter}[1]{%
	\defbibfilter{#1}{%
		segment=0
		or
		segment=\therefsegment
}}
\usepackage[utf8]{inputenc}

\usepackage{xcolor}

\usepackage[shortlabels]{enumitem}

\usepackage{amssymb}
\usepackage{amsmath}
\usepackage{amsthm} 
\usepackage{latexsym}
\usepackage{mathtools} 
\usepackage{mathrsfs} 
\usepackage{graphicx}

\usepackage{tikz}
\usetikzlibrary{calc}
\usepackage{tikz-cd} 
\usetikzlibrary{arrows,decorations.markings,topaths}
\tikzcdset{arrow style=tikz,>=angle 60}
\tikzset{double line with arrow/.style args={#1,#2}{decorate,decoration={markings,%
			mark=at position 0 with {\coordinate (ta-base-1) at (0,1pt);
				\coordinate (ta-base-2) at (0,-1pt);},
			mark=at position 1 with {\draw[#1] (ta-base-1) -- (0,1pt);
				\draw[#2] (ta-base-2) -- (0,-1pt);
}}}}
\tikzset{Equal/.style={-,double line with arrow={-,-}}}

\usepackage{mathcomp}

\makeatletter 
\let\ams@starttoc\@starttoc
\makeatother
\usepackage[parfill]{parskip}
\makeatletter
\let\@starttoc\ams@starttoc
\patchcmd{\@starttoc}{\makeatletter}{\makeatletter\parskip\z@}{}{}
\makeatother

\makeatletter
\g@addto@macro \normalsize {%
	\setlength\abovedisplayskip{5pt}%
	\setlength\belowdisplayskip{5pt}%
}
\makeatother

\definecolor{coolblack}{rgb}{0.0, 0.18, 0.39}
\definecolor{royalbluetraditional}{rgb}{0.0, 0.14, 0.4}
\definecolor{sapphire}{rgb}{0.03, 0.15, 0.4} 

\definecolor{mediumtealblue}{rgb}{0.0, 0.33, 0.71}
\definecolor{yaleblue}{rgb}{0.06, 0.3, 0.57}
\definecolor{smaltdarkpowderblue}{rgb}{0.0, 0.2, 0.6}
\definecolor{internationalkleinblue}{rgb}{0.0, 0.18, 0.65}
\definecolor{uablue}{rgb}{0.0, 0.2, 0.67}
\definecolor{persianblue}{rgb}{0.11, 0.22, 0.73}
\definecolor{royalazure}{rgb}{0.0, 0.22, 0.66}

\definecolor{zaffre}{rgb}{0.0, 0.08, 0.66}
\definecolor{phthaloblue}{rgb}{0.0, 0.06, 0.54}
\definecolor{navyblue}{rgb}{0.0, 0.0, 0.5}
\definecolor{ultramarine}{rgb}{0.07, 0.04, 0.56}
\definecolor{dukeblue}{rgb}{0.0, 0.0, 0.61} 
\definecolor{darkblue}{rgb}{0.0, 0.0, 0.55}
\definecolor{midnightblue}{rgb}{0.1, 0.1, 0.44}
\definecolor{stpatricksblue}{rgb}{0.14, 0.16, 0.48}

\definecolor{indigoweb}{rgb}{0.29, 0.0, 0.51}
\definecolor{darkmagenta}{rgb}{0.55, 0.0, 0.55}
\definecolor{electricindigo}{rgb}{0.44, 0.0, 1.0}

\definecolor{ao(english)}{rgb}{0.0, 0.5, 0.0}
\definecolor{armygreen}{rgb}{0.29, 0.33, 0.13}
\definecolor{britishracinggreen}{rgb}{0.0, 0.26, 0.15}

\definecolor{calpolypomonagreen}{rgb}{0.12, 0.3, 0.17}
\definecolor{darkgreen}{rgb}{0.0, 0.2, 0.13}
\definecolor{msugreen}{rgb}{0.09, 0.27, 0.23}
\definecolor{deepjunglegreen}{rgb}{0.0, 0.29, 0.29}

\definecolor{huntergreen}{rgb}{0.21, 0.37, 0.23}

\definecolor{upforestgreen}{rgb}{0.0, 0.27, 0.13}

\definecolor{sacramentostategreen}{rgb}{0.0, 0.34, 0.25}

\definecolor{forestgreentraditional}{rgb}{0.0, 0.27, 0.13} 

\definecolor{midnightgreeneaglegreen}{rgb}{0.0, 0.29, 0.33}
\definecolor{warmblack}{rgb}{0.0, 0.26, 0.26}

\definecolor{myrtle}{rgb}{0.13, 0.26, 0.12}

\newcommand{\mycolor}{dukeblue} 

\usepackage[bookmarksopen=false,bookmarksnumbered,colorlinks=true,urlcolor=\mycolor,linkcolor=\mycolor,citecolor=\mycolor,frenchlinks=false,hyperindex,breaklinks]{hyperref}
\usepackage[capitalize,nameinlink,noabbrev]{cleveref} 
\usepackage{hypcap}
\usepackage{bookmark} 

\newtheoremstyle{italic} 
	{10pt plus 4pt minus 4pt} 
	{10pt plus 4pt minus 4pt} 
	{\itshape} 
	{} 
	{\bfseries} 
	{:} 
	{0.5em} 
	{} 

\newtheoremstyle{normal} 
	{10pt plus 4pt minus 4pt}
	{10pt plus 4pt minus 4pt}
	{}
	{}
	{\bfseries}
	{:}
	{0.5em}
	{}

\theoremstyle{italic}
\newtheorem{theorem}{Theorem}[subsection]
\newtheorem{proposition}[theorem]{Proposition} 
\newtheorem{lemma}[theorem]{Lemma}

\newtheorem{corollary}[theorem]{Corollary}
\newtheorem{definition}[theorem]{Definition}
\newtheorem{theorem-definition}[theorem]{Theorem-Definition}
\newtheorem{proposition-definition}[theorem]{Proposition-Definition}
\newtheorem{lemma-definition}[theorem]{Lemma-Definition}

\theoremstyle{normal}
\newtheorem*{remark}{Remark} 
\newtheorem*{remarks}{Remarks} 

\newtheorem*{example}{Example}
\newtheorem*{examples}{Examples} 

\numberwithin{equation}{section} 

\makeatletter
\renewenvironment{proof}[1][\proofname]{\par
  \pushQED{\qed}%
  \normalfont \topsep0\p@\@plus0\p@\relax
  \trivlist
  \item[\hskip\labelsep
        \itshape
    #1\@addpunct{:}]\ignorespaces
}{%
  \popQED\endtrivlist\@endpefalse
}
\makeatother

\DeclareSymbolFont{largesymbolsA}{U}{esint}{m}{n}
\DeclareMathSymbol{\fintop}{\mathop}{largesymbolsA}{'037}

\tikzset{stairwayStyleLineWidth/.style={line width=0.04em}}
\tikzset{stairwayStyleSharp/.style={stairwayStyleLineWidth}}
\newcommand{\stairup}{\mathbin{
		\tikz[baseline=(stairwayanchor.base)]{
			\node (stairwayanchor) {\quad}; 
			\draw[stairwayStyleSharp] 
			($(stairwayanchor.south west) + (0.0em,0.4em)$)
			-- ++(0.20em,0) -- ++(0,0.20em)
			-- ++(0.20em,0) -- ++(0,0.20em)
			-- ++(0.20em,0) 
			;
	}}
}

\newcommand{\defeq}{\stackrel{\text{\tiny\rm def}}{=}}

\makeatletter
\def\moverlay{\mathpalette\mov@rlay}
\def\mov@rlay#1#2{\leavevmode\vtop{%
		\baselineskip\z@skip \lineskiplimit-\maxdimen
		\ialign{\hfil$\m@th#1##$\hfil\cr#2\crcr}}}
\newcommand{\charfusion}[3][\mathord]{
	#1{\ifx#1\mathop\vphantom{#2}\fi
		\mathpalette\mov@rlay{#2\cr#3}
	}
	\ifx#1\mathop\expandafter\displaylimits\fi}
\makeatother
\newcommand{\bigcupdot}{\charfusion[\mathop]{\bigcup}{\cdot}}

\DeclareFontFamily{U}{mathx}{\hyphenchar\font45}
\DeclareFontShape{U}{mathx}{m}{n}{
	<5> <6> <7> <8> <9> <10>
	<10.95> <12> <14.4> <17.28> <20.74> <24.88>
	mathx10
}{}
\DeclareSymbolFont{mathx}{U}{mathx}{m}{n}
\DeclareFontSubstitution{U}{mathx}{m}{n}
\DeclareMathAccent{\widecheck}{0}{mathx}{"71}
\DeclareMathAccent{\wideparen}{0}{mathx}{"75}

\DeclareSymbolFont{YHlargesymbols}{OMX}{yhex}{m}{n}
\DeclareMathAccent{\wparen}{\mathord}{YHlargesymbols}{"F3}

\makeatletter
\newcommand*\rel@kern[1]{\kern#1\dimexpr\macc@kerna}
\newcommand*\widebar[1]{%
  \begingroup
  \def\mathaccent##1##2{%
    \rel@kern{0.8}%
    \overline{\rel@kern{-0.8}\macc@nucleus\rel@kern{0.2}}%
    \rel@kern{-0.2}%
  }%
  \macc@depth\@ne
  \let\math@bgroup\@empty \let\math@egroup\macc@set@skewchar
  \mathsurround\z@ \frozen@everymath{\mathgroup\macc@group\relax}%
  \macc@set@skewchar\relax
  \let\mathaccentV\macc@nested@a
  \macc@nested@a\relax111{#1}%
  \endgroup
}
\makeatother

\makeatletter
\def\wbreve{\mathpalette\wide@breve}
\def\wide@breve#1#2{\sbox\z@{$#1#2$}%
	\mathop{\vbox{\m@th\ialign{##\crcr
				\kern0.08em\brevefill#1{0.8\wd\z@}\crcr\noalign{\nointerlineskip}%
				$\hss#1#2\hss$\crcr}}}\limits}
\def\brevefill#1#2{$\m@th\sbox\tw@{$#1($}%
	\hss\resizebox{#2}{\wd\tw@}{\rotatebox[origin=c]{90}{\upshape(}}\hss$}
\makeatletter

\newcommand{\wbar}[1]{\widebar{#1}}
\newcommand{\what}[1]{\widehat{#1}}
\newcommand{\wtilde}[1]{\widetilde{#1}}

\usepackage{stackengine}
\usepackage{scalerel}
\usepackage{xcolor}
\newcommand\openbigstar[1][0.7]{%
	\scalerel*{%
		\stackinset{c}{-.125pt}{c}{}{\scalebox{#1}{\color{white}{$\bigstar$}}}{%
			$\bigstar$}%
	}{\bigstar}
}
\newcommand{\mystar}{\openbigstar[.8]}

\newcommand{\eps}{\epsilon}
\newcommand{\vareps}{\varepsilon}

\newcommand{\pd}{\partial}
\newcommand{\pdbar}{\bar{\partial}}

\renewcommand{\d}{\mathrm{d}}

\newcommand{\pdv}[2]{\frac{\pd#1}{\pd#2}}

\newcommand{\norm}[1]{\left\lVert#1\right\rVert}
\newcommand{\abs}[1]{\left\lvert#1\right\rvert}

\makeatletter 
\newcommand\notni{\mathrel{\m@th\mathpalette\canc@l\owns}}
\newcommand\canc@l[2]{{\ooalign{$\hfil#1/\mkern1mu\hfil$\crcr$#1#2$}}}
\makeatother

\newcommand{\rF}{\mathrm{F}}

\newcommand{\rM}{\mathrm{M}}

\newcommand{\rh}{\mathrm{h}}

\renewcommand{\AA}{\mathbb{A}}
\newcommand{\BB}{\mathbb{B}}
\newcommand{\CC}{\mathbb{C}}
\newcommand{\DD}{\mathbb{D}}

\newcommand{\GG}{\mathbb{G}}
\newcommand{\HH}{\mathbb{H}}

\newcommand{\KK}{\mathbb{K}}

\newcommand{\NN}{\mathbb{N}}
\newcommand{\OO}{\mathbb{O}}
\newcommand{\PP}{\mathbb{P}} 

\newcommand{\RR}{\mathbb{R}}
\renewcommand{\SS}{\mathbb{S}}
\newcommand{\TT}{\mathbb{T}}

\newcommand{\ZZ}{\mathbb{Z}}

\renewcommand{\a}{\mathfrak{a}}

\newcommand{\m}{\mathfrak{m}}
\newcommand{\n}{\mathfrak{n}}

\newcommand{\fa}{\mathfrak{a}}

\newcommand{\fm}{\mathfrak{m}}
\newcommand{\fn}{\mathfrak{n}}

\newcommand{\fA}{\mathfrak{A}}
\newcommand{\fB}{\mathfrak{B}}
\newcommand{\fC}{\mathfrak{C}}

\newcommand{\fF}{\mathfrak{F}}

\newcommand{\fI}{\mathfrak{I}}

\newcommand{\fM}{\mathfrak{M}}
\newcommand{\fN}{\mathfrak{N}}

\newcommand{\A}{\mathcal{A}}
\newcommand{\B}{\mathcal{B}}

\renewcommand{\L}{\mathcal{L}}

\renewcommand{\O}{\mathcal{O}}

\newcommand{\Z}{\mathcal{Z}}

\newcommand{\cA}{\mathcal{A}}
\newcommand{\cB}{\mathcal{B}}
\newcommand{\cC}{\mathcal{C}}

\newcommand{\cO}{\mathcal{O}}

\newcommand{\cZ}{\mathcal{Z}}

\newcommand{\Cc}{\mathscr{C}}

\newcommand{\tr}{\operatorname{tr}}
\renewcommand{\mod}{\operatorname{mod}}

\newcommand{\nil}{\operatorname{nil}}
\newcommand{\id}{\operatorname{id}}

\newcommand{\pr}{\operatorname{pr}}
\newcommand{\pt}{\mathrm{pt.}}

\newcommand{\im}{\operatorname{im}}
\newcommand{\rk}{\operatorname{rk}}

\newcommand{\ch}{\operatorname{char}}
\newcommand{\Spec}{\operatorname{Spec}}
\newcommand{\Spm}{\operatorname{Spm}}

\newcommand{\Ind}{\operatorname{Ind}}
\newcommand{\ind}{\operatorname{ind}}

\def\span{\mathrm{span}}
\newcommand{\GL}{\operatorname{GL}}
\newcommand{\End}{\operatorname{End}}
\newcommand{\Aut}{\operatorname{Aut}}
\newcommand{\Hom}{\mathrm{Hom}}
\newcommand{\Mor}{\operatorname{Mor}}
\newcommand{\Mat}{\operatorname{Mat}}

\newcommand{\diam}{\operatorname{diam}}
\newcommand{\const}{\mathrm{const}}

\newcommand{\Adm}{\operatorname{Adm}}

\newcommand{\Int}{\operatorname{Int}}

\newcommand{\pw}{\mathrm{pw}}

\newcommand{\sing}{\mathrm{sing}}

\newcommand{\op}{\mathrm{op}}
\renewcommand{\sp}{\mathrm{sp}}
\newcommand{\loc}{\mathrm{loc}}

\newcommand{\Vect}{\mathsf{Vect}}

\newcommand{\Alg}{\mathsf{Alg}}
\newcommand{\fdAlg}{\mathsf{fdAlg}}
\newcommand{\fdCAlg}{\mathsf{fdCAlg}}
\newcommand{\CRing}{\mathsf{CRing}}

\newcommand{\BanAlg}{\mathsf{BanAlg}}
\newcommand{\CBanAlg}{\mathsf{CBanAlg}}

\newcommand{\fdCFkth}{\mathsf{fdCFkth}}

\begin{document}
	
	\title[Functions Holomorphic over Algebras]{Functions Holomorphic over Finite-Dimensional Commutative Associative Algebras 1: One-Variable Local Theory I}
	
	\author{Marin Genov}
	\address{Faculty of Mathematics and Informatics, Sofia University, Bulgaria}
	\email{mvgenov@uni-sofia.bg, marin.genov@gmail.com}
	

	
	
	
	\date{\today}
	
		
	\begin{abstract}
	We study in detail the one-variable local theory of functions holomorphic over a finite-dimensional commutative associative unital $\CC$-algebra $\cA$, showing that it shares a multitude of features with the classical one-variable Complex Analysis, including the validity of the Jacobian conjecture for $\cA$-holomorphic regular maps and a generalized Homological Cauchy's Integral Formula.
	In fact, in doing so we replace $\cA$ by a morphism $\varphi: \cA \to \cB$ in the category of finite-dimensional commutative associative unital $\CC$-algebras in a natural manner, paving a way to establishing an appropriate \textit{category of Funktionentheorien}.
	We also treat the very instructive case of non-unital finite-dimensional commutative associative $\RR$-algebras as far as it serves above agenda.
	\end{abstract}
	
	\maketitle
	
	\tableofcontents
	
	\section*{Notations \& Conventions}
	

By abuse of notation, we shall denote by the same letter both a curve $\gamma:[0,1]\to X$ (as a continuous map) and its image (support as a cycle), e.g. $\int_\gamma\omega$. 
Moreover, loops $\gamma:\SS^1\to X$ will always be parametrized for notational simplicity as $\gamma(t)$ with $t\in[0,1]$.

In the presence of tensors, we will try to always denote free indices by $i,j,k,\ell$ and summation indices by $r,s,t$, unless we forget to do so.



$\mathbb{N}$: the natural numbers $\{1,2,3,\dots\}$.

$\mathbb{N}_0$: the extended natural numbers, i.e. $\mathbb{N}\cup\{0\}$.

$R$: a ring, often times commutative, with unit, and of characteristic $\ch R=0$.


$K$: arbitrary field, usually $\ch K=0$ or perfect, unless explicitly stated otherwise.

$\kappa$: arbitrary algebraically closed field, usually $\ch\kappa=0$, unless explicitly stated otherwise.

$\KK$: the field $\RR$ or $\CC$.

$\CC^{\times}$: the multiplicative group of the complex numbers.

$\CC^{+}$: the additive group of the complex numbers.

$\GG_a$: the additive (affine) group (scheme).

$\GG_m$: the multiplicative (affine) group (scheme).

$\SS^n$: the $n$-sphere.

$\TT^n\coloneqq (\SS^1)^n$: the standard $n$-torus.

%
%
%

$\Vect_K$: the category of $K$-vector spaces.

$\CRing$: the category of commutative rings with identity element.


$\fdAlg_{K}$: the category of finite-dimensional associative unital algebras over $K$.

$\fdCAlg_{K}$: the category of finite-dimensional commutative associative unital algebras over $K$.

$\BanAlg_{\KK}$: the category of Banach algebras over $\KK$.

$\CBanAlg_{\KK}$: its subcategory of commutative Banach algebras over $\KK$.

$\A$, $\B$, $\mathcal{C}$: finite-dimensional, commutative, associative, unital algebras over $R$, $K$ or $\KK$.

$\fA$, $\fB$, $\fC$: finite-dimensional commutative associative, \textit{but not necessarily unital} algebras over $R$, $K$ or $\KK$.

$\fM$, $\fN$: \textit{connected} finite-dimensional commutative associative \textit{non-unital} $\KK$-algebras, maximal ideals of some $\A$.

$\A^\times = U(\A)$: the group of units (=multiplicative group) of $\A$.

$J(\A)$: the Jacobson radical of $\A$.

$J_\KK F$: the ($\KK$-)Jacobian of a map $F$.

$\Mat_{n\times m}(R)$: $n\times m$-matrices over some ring $R$.

$\rM_n(R)\coloneqq\Mat_n(R)$: the ring of $n\times n$-matrices over a ring $R$ and an associative $R$-algebra if $R\in\CRing$.

$E_n$: the identity matrix of size $n$.

$I_k$: certain idempotent element.

$A^T$: transpose of a matrix $A$.

$\Int(\gamma)$: the interior of a closed plane curve $\gamma\subseteq\CC$.

$\DD_r(z_0)$: the open disc with radius $r$ around $z_0\in\CC$.

$\DD^{\ast}_r(z_0)$: the punctured open disc with radius $r$ around $z_0\in\CC$.

$\Delta_r(z)$: open polydisc of radius $r$ around $z\in\CC^n$.

$\widebar{\Delta}_r(z)$: closed polydisc of radius $r$ around $z\in\CC^n$.

$\BB_r(z)$: open ball of radius $r$ around $z\in\KK^n$.

$\widebar{\BB}_r(z)$: closed ball of radius $r$ around $z\in\KK^n$.

$\Z(f)$: the zero set of a function $f$ whenever that makes sense.

$C^k(U,V)$: space of $k$-times continuously differentiable functions on $U$ with values in $V$, $k\in\NN_0\cup\{\infty,\omega\}$.

$\norm{f}_K\coloneqq \sup_{x\in K}\norm{f(x)}$: the uniform norm over (usually a compactum) $K$, e.g. $\norm{f}_{\fB,\gamma}$.

$\Cc^1_\pw(\SS^1,U)$: space of piece-wise continuously differentiable curves.



$\{\pt\}$: the one-point space.

$\triangle$: a proverbial solid triangle with boundary $\pd\triangle$.

$\square$: a proverbial solid rectangle with boundary $\pd\square$.

$\mystar$: a star-shaped open set.

\textleaf: a leaf from a tree or a bush.

$\stairup$: piece-wise smooth path.

	\section{Introduction and Overview}
	
Starting with the definition of complex differentiability of a function $f:U\to\CC$ for an open $U\subseteq\CC$,
\begin{equation}
f'(z_0) \coloneqq \lim_{h\to 0}\frac{f(z_0+h)-f(z_0)}{h},
\end{equation}
given the ubiquity of Complex Analysis throughout mathematics, it is a natural question if it can be generalized to $\KK$-algebras $\A$ other than $\CC$ itself and, if yes, to what extent.
For example, if $U\subseteq\A$ is open and $f:U\to\A$ some function, one may want to define
\begin{equation}
f'(Z_0) \coloneqq \lim_{\substack{H\to 0 \\ \ H\in\A^\times}}\frac{f(Z_0+H)-f(Z_0)}{H},	
\end{equation}
whenever the limit makes sense, or
\begin{equation}\label{tmp01}
f(Z_0+H) = f(Z_0) + f'(Z_0)H + o(H) \text{ as } H\to 0,\ H\in\cA.
\end{equation}
A moment's thought shows that such a generalization and its usefulness would depend in an essential way on the basic properties of $\A$ such as associativity, commutativity, or unitality.
If $\A$ is both commutative and associative, then a function $f$ satisfying \ref{tmp01} is often called $\A$-differentiable
at $Z_0$.
To reasonably formalize such generalization attempts, borrowing from the German language, we shall henceforth call any \textit{one-variable} theory of functions that exhibits the following broadly understood properties a (generalized) \textit{Funktionentheorie}\footnote{(ger.) function theory;}:
\begin{enumerate}[(i)]
	\item presence of generalized Cauchy-Riemann equations (CR);
	\item presence of generalized Cauchy and Morera Integral Theorems (CIT \& MIT);
	\item presence of generalized Cauchy Integral Formula(s) (CIF);
	\item presence of Analyticity (AN).
\end{enumerate}

Furthermore, since we are in practice working within an entire category of algebraic objects, namely (finite-dimensional) \textit{commutative associative} (Banach) $\KK$-algebras, it makes sense to adopt a relative point of view and consider a morphism $\varphi: \fA \to \fB$ of such algebras instead of restricting ourselves to a single fixed object $\fA$.
This should serve as a first hint at the underlying categorical flavour of the theory.
Picking a morphism $\varphi$ turns $\fB$ into an $\fA$-algebra in the usual way via $\forall A \in \fA\ \forall B \in \fB: AB\coloneqq \varphi(A)B$, and the appropriate generalization of $\fA$-differentiability reads as follows: for $U\subseteq\fA$ open we define a function $f:U\to\fB$ to be $\varphi$-differentiable at $Z_0\in U$ iff
\begin{equation}
f(Z_0+H) = f(Z_0) + f'(Z_0)\varphi(H) + o(\varphi(H)) \defeq f(Z_0) + f'(Z_0) H + o(\norm{\varphi(H)}) \text{ as } H\to 0 \text{ in } \fA
\end{equation}
for some element $f'(Z_0) \in \fB$, which is then the ($\varphi$)-derivative of $f$.
For example, if $\varphi:\CC\to\CC$, $z\mapsto \bar{z}$, is the complex conjugation, then the $\varphi$-holomorphic functions are precisely the anti-holomorphic functions.
If $\cA$ is complex and unital and $\iota: \CC \hookrightarrow \cA$ is the canonical inclusion of the scalars, then $\iota$-holomorphic functions are the same as $\cA$-holomorphic functions when viewed as functions of Several Complex Variables (satisfying additional conditions).
More generally, $\varphi$-differentiable functions for some $\varphi: \fA \to \fB$ can be viewed as certain $\fB$-valued functions of Several (usually, however, non-free) $\fA$-Variables.
Finally, note that we recover $\fA$-differentiability by simply setting $\varphi \coloneqq \id_\fA$.

In the present series of papers including \cites{MG02,MG03} we study in detail and optimal generality the local theory of $\varphi$-differentiable functions in the finite-dimensional setting, having at our disposal the basic structure theory of finite-dimensional commutative associative algebras and the apparatus of (finite-dimensional) differential geometry and algebraic topology.
The two key facts we use are that these algebras, when unital, are Artinian and hence decompose into a finite direct sum of \textit{local} Artinian $\KK$-algebras with nilpotent corresponding maximal ideals and that their group of units $\cA^\times$ is path-connected iff the algebra carries a compatible complex structure, in which case they are also triangulable since $\CC$ is algebraically closed.

While the notion of $\cA$-differentiability for a (finite-dimensional) unital commutative associative (Banach) $\KK$-algebra is rather old\footnote{the first article on the subject was published already in 1893 by \textcite{Scheff}. 
While a survey on $\cA$-differentiability would be too ambitious for the scope of this introduction, we give an overview of relevant references on the subject at the end of the article.}, 
its generalization to $\varphi$-differentiability for a morphism $\varphi: \fA \to \fB$ of (not necessarily unital) finite-dimensional associative commutative $\KK$-algebras and the study of the thus resulting function theory together with the possibility of a category of \textit{Funktionentheorien} are to our best knowledge very much new.
To unpack the definition, in \cref{sec:phihol} we give a discussion of various aspects, pitfalls, and consequences of $\varphi$-differentiability, using the aforementioned algebraic structure theory.
In particular, we demonstrate that for many purposes of a local function theory as a consequence it suffices to consider morphisms of connected algebras, or even only inclusions of connected algebras.
A byproduct of this discussion is the fact that a $\varphi$-differentiable function $f:U \to \fB$ automatically extends to a $\varphi$-differentiable function $\hat{f}: \what{U} \to \fB$, where the precise meaning of $\what{U}$ is given in \cref{def:projclosure}.

Furthermore, in \cref{jacobian} we establish the relationship between the abstract $\fA$-derivative of an $\fA$-differentiable function $f$ and its Jacobian $J_\KK f$ over $\KK$ when seen as a function of several $\KK$-variables.
As a consequence of this and the structure theory of finite-dimensional commutative associative $\KK$-algebras we show in \cref{jaconj} that the Jacobian Conjecture easily follows for $\cA$-holomorphic regular maps, where $\cA$ is unital.
The validity of the Jacobian Conjecture in the case of $\cA$-holomorphic regular maps as well as the easiness of its proof serve in our opinion as a philosophical affirmation of the view that $\cA$-holomorphic functions are much closer in spirit to the function theory of one complex variable than of Several Complex Variables, both locally and globally, and hence provide a middleground between both realms that presents a suitable opportunity to test statements known in one complex dimension, but unknown or difficult in higher complex dimensions.
We finish \cref{sec:phihol} by verifying that the collection of $\varphi$-differentiable functions actually yields a sheaf $\O_\varphi$ of $\fA$-algebras with a distinguished derivation.
In \cite{MG02} we study among other things further properties of $\O_\varphi$ for $\varphi: (\cA,\fm) \to (\cB,\fn)$ a morphism of local $\CC$-algebras and show in particular that $(U,\cO_\cA|_U)$ for $U\subseteq\cA$ open is a \textit{locally ringed space}.
Later on, in \cite{MG07} we shall use $\cO_\cA$ to model global 1-dimensional analytic spaces over $\cA$, i.e. $\cA$-analytic curves with singularities.

In \cref{sec:CR} we move on to state various versions of the corresponding generalized Cauchy-Riemann Equations for $\varphi$-differentiable functions and give a discussion of the various angles of the local coordinate formulations, including answering some inverse questions such as the extent to which an abstractly given collection of $\fA$-differentiable functions determines $\fA$ uniquely.
We finish \cref{sec:CR} with the introduction of certain differential operators $\d_{ij} \coloneqq a_i \pdv{}{z^j} - a_j \pdv{}{z^i}$, where $\{a_1,\dots,a_n \}$ is a $\KK$-basis of $\fA$, that actually play the analogous role of the Wirtinger derivatives, and state the analogous critaria for $\varphi$-differentiability.
Later on we will show in \cref{closed} that $\d_{ij} f$ are in fact the coordinates of $\d (f(Z)\d Z)$ in the case $\KK=\RR$ or $\pd (f(Z)\d Z)$ in the case of $\KK=\CC$.

In \cref{sec:nonunitalfunc} we routinely but carefully verify for completeness sake that certain standard facts about integrals over closed contours from the Complex Analysis of One Variable transfer almost verbatim to the case of a morphism of \textit{not necessarily unital} commutative associative $\KK$-algebras, including the Cauchy-Goursat Integral Theorem.
These concern $\varphi$-primitives and $\varphi$-integrability of a function $f:U \to \fB$, which are the obvious analogues of the corresponding notions in the Complex Analysis of One Variable.
Most of the proofs turn out to be just like in the case of one complex variable and can therefore be skipped. 
However, some of the consequences are topologically more interesting due to the fact that we now find ourselves in three or more real dimensions.
At the heart of the matter lies the fact that for a $\varphi$-differentiable function $f$ the $\fB$-valued 1-form $\omega \coloneqq f(Z)\d Z$ is $\d$-closed, which is also central to the later theory by ensuring the homotopy-invariance of (the integrals of) $\omega$.

This is pretty much as far as one can model the classical theory of one complex variable in the case of non-unital commutative associative $\KK$-algebras.
To take the theory further, we need to consider contour integrals of $\d Z/Z$ and related 1-forms, and thus it becomes abundantly clear that one cannot dispense with the path-connectedness of $\cA^\times$, which is shown in \cref{subs:UA} to be the case if and only if $\cA$ is in fact a $\CC$-algebra. 
Triviality of $\pi_0(\cA^\times)$ turns out to be the only obstruction to obtaining all the features of a \textit{Funktionentheorie} as \cref{sec:unitalfunc} demonstrates.
Thus, from here on out we will suppose $\varphi: \cA \to \cB$ to be a morphism of \textit{complex unital} finite-dimensional commutative associative algebras.

In \cref{subs:adm} we introduce some conventions and objects of purely technical nature that will reoccur throughout and that allow us to state the integration theorems in their full generality.
A key feature governing the Funktionentheorie over finite-dimensional commutative associative unital $\CC$-algebras $\cA$ is the algebraic fact that if $\Spm \cA = \{\fM_1,\dots,\fM_M\}$ is the (maximal) (ring-theoretic) spectrum of $\cA$, then the quotient (spectral) projections $\sigma_k: \cA \twoheadrightarrow \cA/\fM_k \cong \CC$, $1\leq k\leq M$, are not only algebra epimorphisms, but also projections onto 1-complex-dimensional \textit{subspaces} of $\cA$ and moreover $\sigma(Z) \coloneqq \{\sigma_1(Z),\dots,\sigma_M(Z)\}$ is precisely the (eigenvalue) spectrum of $Z\in\cA$ as suggested by the choice of notation.

In \cref{subs:ind} we introduce $\Ind_\varphi(\Gamma,Z_0)$ for $Z_0\in U$ and $\Gamma \in Z_1(U,\ZZ)$, which is the appropriate analogue of the index of a 1-cycle around a point in the complex plane suitable for the purposes of the theory of $\varphi$-holomorphic functions.
While some ad hoc lower-dimensional cases of $\Ind_\varphi$ in disguise have already been computed explicitly in the literature, we give a general treatment and computation in full, showing that it reduces in an appropriate way to the usual $\CC$-indices of the spectral projections of the said 1-cycle in the respective complex planes.
While the generalized index depends in principle very much on a choice of an algebra structure on the underlying vector space $\cA$, it turns out to be invariant under algebra endomorphisms.
Another interesting feature of the generalized index is the fact that it is $\cB$-valued in a natural way, hinting that it is also useful to consider 1-cycles $\Gamma\in Z_1(U,R)$ with values in a commutative ring other than $\ZZ$ like $R=\cA$ or $R=\cB$.
This point of view will be pursued in detail later in \cite{MG04}.

In \cref{subs:CIF} we prove the analogue of Cauchy's Integral Formula \textit{with index} for $\varphi$-holomorphic functions.
A consequence of this, most easily stated for a morphism $\varphi: (\cA,\fm) \to (\cB,\fn)$ of local $\CC$-algebras, is that, if $f:U \to \cB$ is a $\varphi$-holomorphic function on some open $U\subseteq \cA$, then $f$ extends $\varphi$-holomorphically to cylinders $\DD_r(\sigma_\cA(Z)) \times \fm$, where $\sigma_\cA: \cA \twoheadrightarrow \cA/\fm \cong \CC$ is the canonical projection, $Z\in U$ is a point, and $r>0$ is a small enough radius such that $\DD_r(\sigma_\cA(Z)) \subseteq \sigma_\cA(U)$.
In fact, we show a little later that $f$ extends to the whole $\wtilde{U} \coloneqq \sigma_\cA(U) \times \fm$.

\cref{subs:anal} is entirely dedicated to questions of analyticity.
Let $\cA = (\cA,\fm)$ be a local $\CC$-algebra.
If $f(Z) = \sum_{n=0}^\infty A_n Z^n \in \cA\{Z\}$ is a convergent power series, we define as usual 
\begin{equation}
R\coloneqq R_f\coloneqq \frac{1}{\limsup\limits_{n\to\infty}\norm{A_n}^{1/n}}\in[0,\infty]
\end{equation}
to be its radius of convergence.
If $\rho_\cA$ denotes the \textit{spectral} radius of elements in $\cA$ (in fact, we have in this case $\rho_\cA = \abs{\cdot} \circ \sigma_\cA$), we show that the precise domain of convergence of $f$ is the spectral ball of radius $R$, that is, $f(Z)$ is convergent for all $\rho_\cA(Z)<R$ and divergent for all $\rho_\cA(Z)>R$.
While the first part is no surprise at all, the second part is specific to finite-dimensional commutative (local) $\CC$-algebras. 
If we did not suppose commutativity or finite dimensionality, divergence of $f(Z)$ outside of the spectral ball would be in general false since in that case we could have for example zero divisors with positive spectral radius.
The convergence behaviour of $f(Z)$ for a general finite-dimensional commutative $\CC$-algebra $\cA$ then follows from the Artin property of $\cA$ and the decomposition of $\cA$ into a direct sum of local Artin $\CC$-algebras.
We remark that the spectral ball of a local finite-dimensional $\CC$-algebra around $Z_0\in \cA$ is actually a spectral cylinder as it has the form $\DD_R(\sigma_\cA(Z_0)) \times \fm$.

We then go on to prove a version of the Cauchy-Transform for $\cB$-valued measures and a certain class of $\cA$-valued measurable functions satisfying ``more-width-than-depth''-type of a condition.
When applied to the generalized Cauchy's Integral Formula, this shows that $\varphi$-holomorphic functions are locally analytic of the class $\cB\{\varphi(Z-Z_0)\}$ and gives a generalized Cauchy's Integral Formula \textit{with index} for the ($\varphi$)-derivatives of $f$.
This in turn leads to generalized Cauchy's Inequalities and a generalized Liouville's Theorem specific to $\varphi$-holomorphic functions.
We remark that we study the properties of the algebras $\cA\{Z\}$ and other related objects much more closely in \cite{MG02}, while in \cite{MG05} we take a purely algebraic approach to generalized holomorphy that is applicable to a much wider class of commutative associative $R$-algebras for some $R\in\CRing$ of $\ch R=0$ by considering subalgebras $\cA[[a_1 X_1 + \cdots + a_n X_n]] \subseteq \cA[[X_1,\dots,X_n]]$ and related objects and establish properties of $\varphi$-holomorphic functions valid already on the level of formal power series.
In particular, this treatment includes \textit{arithmetic holomorphy} which concerns itself with the case of number fields and rings of integers.
For example, in this context the classical holomorphy is the one associated to the field extension $\CC/\RR$.

In the other direction, the local form of analyticity in the shape of $\cB\{\varphi(Z)\}$ has several important implications, some of which we list here.
For starters, it is immediate that, if for example $\varphi: \cA \hookrightarrow \cB$ is an inclusion ($\CC$-linear ring extension) of $\CC$-algebras, then a $\varphi$-holomorphic function $f:U\to\cB$, $U\subseteq \cA$ open, is locally a restriction of a unique $\cB$-holomorphic function, regardless of the dimension difference between $\cA$ and $\cB$.
Another important consequence is that, if $\varphi: (\cA,\fm) \to (\cB,\fn)$ is a morphism of local $\CC$-algebras, then $f$ has the canonical form
\begin{equation}
f(Z) \defeq f(z\oplus X) = \sum_{k=0}^{\nu-1}\frac{f^{(k)}(z)}{k!}\varphi(X)^k,
\end{equation}
where $z\in\CC$, $X\in\m$, $f^{(k)}|\sigma_\cA(U)\to\B$ simply are $\CC$-holomorphic functions with values in the finite-dimensional Banach space $\B$, and $\nu \coloneqq \rh(\varphi) \coloneqq \min\{\rh(\cA),\rh(\cB)\}$ is called the height of the morphism, i.e. the minimum of the degrees of nilpotency of the maximal ideals $\fm$ and $\fn$.
This in turn leads to the promised automatic extension of $f$ to $\wtilde{U}$ and goes to show by separating the scalar variable from the nilpotent variable that the theory of $\varphi$-holomorphic functions is, locally at least, quite literally a combination of the complex analysis of one variable and the multiplicative structure of finite-dimensional commutative connected $\CC$-algebras.
Conversely, above formula prescribes a way to create $\varphi$-holomorphic functions, and so we obtain a full characterization of $\varphi$-holomorphic functions for $\varphi: \cA \to \cB$ being a morphism of unital $\CC$-algebras.
It is not difficult to take this a little further and show that an $\cA$-biholomorphism $f:U \xrightarrow{\sim} V$ automatically extends to an $\cA$-biholomorphism $\tilde{f}: \wtilde{U} \xrightarrow{\sim} \wtilde{V}$.

Finally, a key consequence of this form of analyticity is the ability to \textit{derive at nilpotent elements}\footnote{for a lack of a better term;}: if $(\cA,\fm) \xrightarrow{\varphi} (\cB,\fn)$ is a morphism of local $\CC$-algebras and $Z_0\in U$ and $X\in\m=\nil\A$, then the limit
\begin{equation}
	f'_{(X)}(Z_0) \coloneqq \lim_{\substack{H\to X \\ H\in\A^\times}} \frac{f(Z_0+H)-f(Z_0)}{\varphi(H)}
\end{equation} 	
exists and gives rise to a function $\check{f}:\m\times U\to\B$, $(X,Z)\mapsto f'_{(X)}(Z)$, polynomial in $X$ and $\varphi$-analytic in $Z$.
This has in turn the important implication that the function $g:U\times U\to\B$, given by
\begin{equation}
g(Z,W) \coloneqq
\begin{cases}
\frac{f(W)-f(Z)}{\varphi(W-Z)}, &\text{if } W-Z\in\A^\times\cap U\\
f'_{(W-Z)}(Z), &\text{if } W-Z\in\m\cap U
\end{cases}
\end{equation}
is $\varphi$-holomorphic both in $Z$ and $W$ (see \cref{nilanal}).
This plays a key role in the proof of the \hyperref[homolCIT]{generalized Homological Cauchy Integral Theorem (\ref*{homolCIT})}, which we state here only for the case of local $\CC$-algebras for notational simplicity: if $\varphi: (\cA,\fm) \to (\cB,\fn)$ is a morphism of local $\CC$-algebras, $\sigma_\cA: \cA \twoheadrightarrow \cA/\fm \cong \CC$ the canonical quotient projection, $U\subseteq \cA$ path-connected and open, $f\in\cO_\varphi(U)$, and $\Gamma\in Z_1(U,\ZZ)$ a 1-cycle such that $(\sigma_\cA)_\# \Gamma \in B_1(\sigma(U),\ZZ)$, then 
\begin{equation}
f(Z)\Ind_\varphi(\Gamma,Z)=\frac{1}{2\pi i}\int_\Gamma\frac{f(W)}{\varphi(W-Z)}\d W,
\end{equation}
whenever the integral is well-defined.
A similar result has been shown by Giovanni Battista Rizza (1952) \cite{Rizza} for 1-cycles $\Gamma$ in a local finite-dimensional $\CC$-algebra $\A=(\A,\m)$ under the condition that $[\Gamma]=0$ in $H_1(U,\ZZ)$. 
Clearly, our result makes use of a much weaker topological assumption.

Thus, in summary, to every morphism of finite-dimensional complex commutative associative unital algebras one associates a \textit{Funktionentheorie}.
We mention that already in complex dimension 7 there exist infinitely many distinct isomorphism classes of local commutative associative unital $\CC$-algebras \cites[see][]{Poo,Sup,SupTu,SupTuEng}, which attests to the richness of the theory.
In \cite{MG03} we shall continue this line of thought in conjunction with the philosophy of \cite{Sch} and \cite{SchZm} towards constructing an ``analytic'' category $\fdCFkth$ of \textit{finite-dimensional commutative Funktionentheorien}.
In \cite{MG06} we will study the $\cA$-analogue of Riemann Surfaces, i.e. smooth $\cA$-curves, of which there exist many examples, towards constructing a ``smooth geometric'' version of $\fdCFkth$, while in \cite{MG07} we will also allow for singularities.

	\section{Preliminaries}
	

\subsection{The Structure Constants of a $K$-Algebra}
Let $K$ be an arbitrary field and let $\fA$ be a finite-dimensional $K$-algebra with a choice of basis $\{a_1,\dots,a_n\}$.
We shall denote by $\lambda: \fA \to \End_{\Vect}(\fA) \cong \rM_n(K)$, $a \mapsto (\lambda_a: a' \mapsto a a')$, the left regular representation of $\fA$.
Notice that, unless $\fA = \cA$ is unital, $\lambda$ is not necessarily faithful, e.g. take $\fA_0 \coloneqq \CC \vareps$ with $\vareps^2 = 0$, but one can obtain a faithful representation of $\fA$ by adjoining a unit and then removing the scalars from the regular faithful representation of the unital algebra.

\begin{definition}
	The structure tensor (structure constants) $(\alpha^i_{jk})_{1\leq i,j,k\leq n}$ of $\fA$ with respect to the choice of basis $\{a_1,\dots,a_n\}$ is given by
	\begin{equation}
	\forall 1\leq j,k\leq n: a_j a_k = \sum_{r=1}^n \alpha^r_{jk} a_r .
	\end{equation}
\end{definition}

In other words, we have
\begin{equation}
\forall 1\leq j\leq n: a_j (a_1,\dots,a_n) = (a_1,\dots,a_n)(\alpha^i_{jk})_{1\leq i,k\leq n}.
\end{equation}
Thus $\lambda(a_j) = (\alpha^i_{jk})_{1\leq i,k\leq n} \eqqcolon A_j \in\rM_n(K)$, $1\leq j\leq n$, are precisely the representation matrices of the basis vectors.
 
Furthermore, if $\fA = \cA$ is unital with unit $1_\cA$, we define
\begin{equation}\label{unitcoord}
1_\A \eqqcolon \sum_{r=1}^n \eps^r a_r 
\end{equation}
to be the coordinates of the unit.

\begin{lemma}[Basic Relations for the Structure Tensor of $\fA$]\ 
	\begin{enumerate}[topsep=-\parskip]
		\item $\fA$ is unital with unit $1_{\fA}$ if and only if $\exists (\eps^r)_{1\leq r\leq n} \in K^n \ \forall 1\leq i,k\leq n:$
		\begin{equation}\label{unit}
			\sum_{r=1}^n \eps^r \alpha^i_{rk} = \sum_{r=1}^n \eps^r \alpha^i_{kr} = \delta^i_k.
		\end{equation}
		Moreover, in the special case $a_1=1_{\cA}$, we have $\forall 1\leq i,k\leq n:$
		\begin{equation}
			\alpha^i_{1k} = \alpha^i_{k1} = \delta^i_k.
		\end{equation}
		
		 \item $\fA$ is commutative if and only if $\forall 1\leq i,j,k\leq n:$
		 \begin{equation}\label{commut}
		  	\alpha^i_{jk}=\alpha^i_{kj}.
		 \end{equation}
		 
		 \item $\fA$ is associative if and only if $\forall 1\leq i,j,k,\ell \leq n:$
		 \begin{equation}\label{ass}
		 \sum_{r=1}^n \alpha^r_{jk} \alpha^i_{r\ell} = \sum_{r=1}^n \alpha^r_{k\ell} \alpha^i_{jr}.
		 \end{equation}
	\end{enumerate}
	
\end{lemma}

\begin{proof}
To (1): this is a coordinate restatement of $\forall 1\leq k\leq n: 1_\A a_k = a_k 1_\A = a_k$.

To (2): this is a coordinate restatement of $\forall 1\leq j,k\leq n: a_j a_k = a_k a_j$.

To (3): this is a coordinate restatement of $\forall 1\leq i,j,j\leq n: (a_i a_j) a_k = a_i (a_j a_k)$.
\end{proof}

\subsection{The Structure Constants of a Morphism of $K$-Algebras}
Let $\fA \xrightarrow{\varphi} \fB$ be a morphism of $K$-algebras with bases $\{a_1,\dots,a_n\}$ and $\{b_1,\dots,b_m\}$ and corresponding structure tensors $(\alpha^i_{jk})_{1\leq i,j,k\leq n}$ and $(\beta^i_{jk})_{1\leq i,j,k\leq m}$, respectively.
Then $\fB$ is a left $\fA$-module via $\forall a\in\fA\ b\in\fB: ab \coloneqq \varphi(a)b$.

\begin{definition}
The (left) structure constants $(\gamma^i_{jk})_{1\leq i,k\leq m, 1\leq j\leq n}$ of the morphism $\varphi$ with respect to the bases $\{a_1,\dots,a_n\}$ and $\{b_1,\dots,b_m\}$ are given by
	\begin{equation}
		\forall 1\leq j\leq n\ \forall 1\leq k\leq m: a_j b_k = \sum_{r=1}^m \gamma^r_{jk} b_r.
	\end{equation}
\end{definition}

There are many relations one could state, but for the moment we shall need only the following two:

\begin{lemma}
Let $\fB$ be associative.
Then:
	\begin{enumerate}
		\item We have $\forall 1\leq j \leq n\ \forall 1\leq i,k,\ell \leq m:$
		\begin{equation}\label{assB1}
		\sum_{r=1}^m \beta^i_{r\ell} \gamma^r_{jk} = \sum_{r=1}^m \gamma^i_{jr} \beta^r_{k\ell}.
		\end{equation}
		
		\item Multiplicativity of $\varphi$ implies $\forall 1\leq j,k\leq n\ \forall 1\leq i,\ell\leq m:$
		\begin{equation}\label{assA1}
		\sum_{r=1}^n \gamma^i_{r\ell} \alpha^r_{jk} = \sum_{s=1}^m \gamma^i_{js} \gamma^s_{k\ell}.
		\end{equation}
	\end{enumerate}
\end{lemma}

\begin{proof}
To (1): $\forall 1\leq j\leq n,\ \forall 1\leq k,\ell\leq m: (a_j b_k)b_\ell = a_j (b_k b_\ell)$.

To (2): $\forall 1\leq j,k\leq n,\ \forall 1\leq\ell\leq m: (a_j a_k)b_\ell = a_j(a_k b_\ell)$.
\end{proof}

\subsection{Some Commutative Algebra}

\emph{In the following subsection all rings (incl. algebras) are assumed commutative, associative, and with unit, unless explicitly stated otherwise.
Moreover, in this chapter we will be using $U(R)$ to denote the group of units of any $R\in\CRing$ instead of $R^\times$, e.g. $U(R[[T_1,\dots,T_m]])$ instead of $R[[T_1,\dots,T_m]]^\times$.}

Since commutative algebras are central in what follows, it is prudent to discuss some general facts from Commutative Algebra. 
If $\A$ is an $n$-dimensional $K$-algebra, then, by virtue of finite dimensionality, $\A$ is Artinian, that is, every descending chain of ideals stabilizes (becomes stationary). 
We summarize the most important facts about Artinian rings in the following

\begin{lemma}[Artin Rings]
	Let $A$ be an Artinian Ring. Then:
	\begin{enumerate}
		\item $A$ possesses a composition series of ideals $A\eqqcolon\a_\ell\supsetneqq\a_{\ell-1}\supsetneqq\dots\supsetneqq\a_1\supsetneqq\a_0\coloneqq 0$ of length $\ell_A(A)\coloneqq\ell<\infty$, i.e. $A$ is of finite length as an $A$-module.
		\item $A$ is Noetherian.
		\item $A$ has Krull $\dim A=0$, i.e. $\Spec A=\Spm A$ as sets, and moreover they are finite.
		\item $J(A)=\nil A$, and moreover it is a nilpotent ideal.
		\item $A$ can be written uniquely up to index permutation as a finite product of Artin local rings $A \cong \bigoplus_{j=1}^M (A_j,\fm_j)$, where we have a bijection
		\[
		\{\fm_1,\dots,\fm_M\} \leftrightarrow \Spm A \eqqcolon \{\fM_1,\dots,\fM_M\},
		\]
		given by $\fM_j = A_1 \times \cdots \times A_{j-1} \times \fm_j \times A_{j+1} \times \cdots A_M$, $1\leq j\leq M$.
	\end{enumerate}
\end{lemma}

\begin{proof}
	We refer the reader to \cite{AM}.
\end{proof}

Thus let us additionally suppose that $(\A,\m,\kappa)$ is local. 
In particular, $\m=J(\A)=\nil(\A)$ is nilpotent of order $\leq n$. 
By finite dimensionality, $\kappa/K$ is a finite field extension and $\A\cong\kappa\oplus\m$ as vector $K$-spaces, hence we have $\dim_K\A=[\kappa:K]+\dim_K\m$.

\begin{lemma}
	Let $(A,\m,\kappa)$ be a local ring and $M$ a simple $A$-module. Then $M\cong\kappa$.
\end{lemma}

\begin{proof}
	Let $0\neq x\in M$ and define a homomorphism of $A$-modules $f:A\to M$, $a\mapsto ax$. 
	Then $f\neq 0$ since $f(1)=x\neq 0$. 
	Hence $\im f\subseteq M$ is a non-trivial submodule. 
	By simplicity of $M$, it follows that $f$ is surjective, thus $M\cong A/\ker f$, which is then also simple and therefore $\ker f=\m$ by the correspondence for ideals of quotient rings. 
	In other words, $M\cong A/\m=\kappa$.
\end{proof}

\begin{corollary}
	Let $(\A,\m,\kappa)$ be a local $n$-dimensional $\kappa$-algebra. 
	Then $\ell_\A(\A)=\dim_\kappa\A \defeq n$.
\end{corollary}

\begin{proof}
	Since $\A$ is Artinian, take a composition series of ideals $\A\eqqcolon\a_\ell\supsetneqq\m\eqqcolon\a_{\ell-1}\supsetneqq\dots\supsetneqq\a_1\supsetneqq\a_0\coloneqq 0$. 
	Then by definition $\a_i/\a_{i+1}$ is a simple $\A$-module and therefore a $1$-dimensional $\kappa$-vector space. 
	In particular, so is $\a_1$.
\end{proof}

\begin{lemma}\label{structext}
	Let $K$ be a perfect field and let $(\A,\m,\kappa)$ be a finite-dimensional local $K$-algebra. 
	Then $\A$ is automatically a $\kappa$-algebra, extending the $K$-algebra structure.
\end{lemma}

\begin{proof}
	Since $\A=\kappa\oplus\m$ as $K$-vector spaces, it suffices to show that $\m$ is a vector space over $\kappa$, extending the $K$-linear structure. 
	Since $K$ is a perfect field, the finite extension $\kappa/K$ is separable.
	Thus, by the Primitive Element Theorem we have $\kappa=K[\bar{\theta}]$ for some $0\neq\bar{\theta}\in\A/\m=\kappa$.
	Now, pick a representative $\theta\in\bar{\theta}$, $\theta\in U(\A)$, and define a $K$-linear action $\kappa\times\m\to\m$ via $\bar{\theta}\cdot x\coloneqq \theta x$.
	This defines a $\kappa$-linear structure on $\m$, depending on $\theta$ and clearly extending the $K$-linear one.
\end{proof}



For the special case of $\kappa=\CC$, we include here a second proof\footnote{due to Jeremy Rickard, see \cite{PathUnits};} that does not require finite dimensionality of $\A$:

\begin{lemma}
	Let $(\A,\m,\CC)$ be an Artin local, not necessarily finite-dimensional, $\RR$-algebra. 
	Then $\A$ is automatically a $\CC$-algebra, extending the $\RR$-algebra structure.
\end{lemma}

\begin{proof}
	Since $\A$ is Artin, we have $\nil(\A)=\m$. 
	We are going to plug nilpotent elements into formal power series to construct a root of $T^2+1$ in $\A$. 
	As $\kappa$ is algebraically closed, there exists $a\in\A$ such that $\bar{a}^2=-1$ in $\kappa$, i.e. $a^2+1\eqqcolon x\in\m$ is nilpotent. 
	Therefore, $-a^2 = 1-x$ has an inverse $1+y = 1+x+x^2+\dots$ for $y\in\m$, which in turn has a square root $1+z=1+y/2-y^2/8+y^3/16-\dots$ for some $z\in\m$. 
	Now set $a'\coloneqq a(1+z)$, and we get $a'^2 = a^2(1+z)^2 = a^2(1+y) = a^2/(1-x) = -1$ as desired.
\end{proof}


So far, for an $n$-dimensional local Artinian $K$-algebra $(\cA,\fm,\kappa)$ we have established the following important quantities:
\begin{enumerate}[(i)]
	\item the dimension of $\A$ as a $K$-vector space: $\dim_K\A=n$ (by definition);
	\item the Krull dimension of $\A$ (the supremum of the lengths of all strictly ascending chains of prime ideals of $\A$, i.e. the ``geometric'' dimension of $\A$): $\dim\A=0$;
	\item the length of $\A$ as an $\A$-module: $\ell_\A(A)=\dim_\kappa\A$.
\end{enumerate}

But there are several other important quantities one can associate to an $n$-dimensional local Artinian $K$-algebra $(\A,\m,\kappa)$ that will play a role.

\begin{definition}\ 
	\begin{enumerate}[topsep=-\parskip]
		\item Height\footnote{should not be confused with the \textit{height of a prime ideal}, which coincides with the Krull dimension of its localization;} of $\A$: $\rh(\A)$ is the smallest $\nu\in\NN$ such that $\m^\nu=0$, $\m^{\nu-1}\neq 0$, i.e. the ``order of nilpotency'' of $\m$;
		
		\item Width(s) of $\A$: $d_i\coloneqq\dim_K\m^i/\m^{i+1}$, $1\leq i\leq h-1$, in particular, $d\coloneqq d_1$ is the dimension of the Zariski cotangent space $\m/\m^2$, also called width of the local algebra $\A$;
		
		\item Define $n_r\coloneqq 2+\sum_{i=1}^{r-1} d_i$, $1\leq r\leq h$ and notice that $n_1=2$, $n_{r+1}-n_r = d_r\geq 1$, and $n_h-1 = n$.  
	\end{enumerate}
\end{definition}

\begin{definition}[Height of a morphism]
	Let $(\cA,\fm) \xrightarrow{\varphi} (\cB,\fn)$ be a morphism of local Artinian $K$-algebras.
	Then $\rh(\varphi) \coloneqq \min\{\rh(\cA),\rh(\cB)\}$ is called height of the morphism $\varphi$.
\end{definition}

\begin{remark}
	We have $\forall X\in\m: \varphi(X)^{\rh(\varphi)} = \varphi(X^{\rh(\varphi)}) = 0$ in any case.
\end{remark}

Since each $\m^i$ is a finitely generated $\A$-module, each choice of $\kappa$-basis for $\m^i/\m^{i+1}$ lifts to a generating set (over $\A$) of $\m^i$. In the case of $\m/\m^2$ such a generating set for $\m$ is also called pseudo-basis of $\m$.

\begin{lemma}\label{nicebasis}
	Let $(\A,\m,\kappa)$ be an $n$-dimensional $\kappa$-algebra. 
	Then $\A$ has a basis $\{e_1=1_\A,e_2,\dots,e_n\}$ with structure tensor $(\alpha^i_{jk})$ such that:
	\begin{enumerate}[(i)]
		\item $\forall 1\leq i,k\leq n: \alpha^i_{1k}=\alpha^i_{k1}=\delta^i_k$ and
		\item $\forall j,k\geq 2\ \forall i\leq\max\{j,k\}: \alpha^i_{jk}=\alpha^i_{kj}=0$.
	\end{enumerate}
\end{lemma}

\begin{proof}
	Since $\A$ is an algebra over its residue field, there exists a composition series of ideals of length $n$, $\A \eqqcolon \a_1 \supsetneqq \m \eqqcolon \a_2 \supsetneqq \dots \supsetneqq \a_n \supsetneqq \a_{n+1} \coloneqq 0$. 
	Pick $e_1 \coloneqq 1_\A\in\a_1\setminus\a_2$, $e_2\in\a_2\setminus\a_3$, $e_i\in\a_i\setminus\a_{i+1}$, $1\leq i\leq n$, and notice that the choice of $e_1$ gives (i).
	
	Now without loss of generality suppose that $2\leq j\leq k$. 
	Then $e_j e_k\in\a_k$, hence $\forall i<k: \alpha^i_{jk}=0$. 
	We are left to check the case $i=k$. 
	We have
	\[
	(e_j-\alpha^k_{jk})e_k = \sum_{i=k+1}^n \alpha^i_{jk}e_i\in\a_{k+1}.
	\]
	If $\alpha^k_{jk}\neq 0$, then $e_j-\alpha^k_{jk}\in U(\A)$, therefore $e_k\in\a_{k+1}$, a contradiction. This proves (ii).
\end{proof}

\begin{remark}
	In particular, if $n\geq 2$, then $e_n^2 = 0$, and if $n\geq 3$, then $e_{n-1}e_n = 0$.
\end{remark}

One criterion for niceness of a basis choice is the number of zero structure constants. 
In this regard, the choice in \cref{nicebasis} is not necessarily the most optimal one. 
To improve on this number one might use the following

\begin{lemma}\label{nicebasis2}
	Let $(\A,\m,\kappa)$ be a local $n$-dimensional $\kappa$-algebra. 
	Then $\A$ has a $\kappa$-basis $\{e_1\coloneqq 1_\A,e_2,\dots,e_n\}$ with structure tensor $(\alpha^i_{jk})$ such that:
	\begin{enumerate}[(i)]
		\item $\forall 1\leq i,k\leq n: \alpha^i_{1k}=\alpha^i_{k1}=\delta^i_k$ and
		\item $\forall 1\leq r, s\leq h-1\ \forall j\geq n_r\ \forall k\geq n_s\ \forall 1\leq i\leq n_{\min\{r+s,h\}}-1: \alpha^i_{jk}=0$.
	\end{enumerate}
\end{lemma}

\begin{proof}
	Consider the descending chain of ideals $\A \supsetneqq \m \supsetneqq \m^2 \supsetneqq \dots \supsetneqq \m^{h-1} \supsetneqq \m^h = 0$. 
	This is a decreasing chain of vector $\kappa$-subspaces of $\A$, therefore we can pick a basis $e_1\coloneqq 1_\A\in\A\setminus\m$, $e_{n_1},\dots,e_{n_2-1}\in\m\setminus\m^2$, $e_{n_2},\dots,e_{n_3-1}\in\m^2\setminus\m^3$ etc. 
	Let $r,s\geq 1$ and let $e_j\in\m^r$ and $e_k\in\m^s$, that is, $j\geq n_r$ and $k\geq n_s$. 
	Then $e_j e_k\in\m^{r+s}$, therefore $\forall 1\leq i\leq n_{\min\{r+s,h\}}-1: \alpha^i_{jk}=0$, where the $\min$ condition comes from the fact that $e_j e_k = 0$ whenever $r+s\geq h$.
\end{proof}


\begin{remarks}\ 
	\begin{enumerate}[topsep=-\parskip]
		\item We still have $e_n^2 = 0$ if $n\geq 2$, and $e_{n-1}e_n = 0$ if $n\geq 3$. In fact $\forall j\geq n_r\ \forall k\geq n_{h-r}: e_j e_k = 0$, since $e_j\in\m^r$ and $e_k\in\m^{h-r}$.
		
		\item It is not always possible to improve on the number of zero structure constants by means of \cref{nicebasis2}, e.g. if $d_1=d_2=\dots=d_{h-1} = 1$ like in the case of $\A_{n-1}\coloneqq K[t]/(t^{n-1})$.
		
	\end{enumerate}
\end{remarks}

\begin{lemma}[Triangular Form]\label{triangular}
Let $\kappa$ be algebraically closed and $\A=(\A,\m,\kappa)$ local $\kappa$-algebra of $\dim_\kappa\A=n$. 
Then $\A$ is triangulable, i.e. isomorphic to a (commutative) $\kappa$-algebra of matrices of size $n$ of the form
\[
\begin{pmatrix}
x & 	   & \ast	\\
& \ddots &  		\\
0 &        & x 		\\
\end{pmatrix}.
\]
\end{lemma}

\begin{proof}
Identify $\A$ with one of its faithful representations into $\rM_n(\kappa)$. 
Since $\kappa$ is algebraically closed, all elements of $\A$ are triangulable as their minimal polynomials split, and in fact they are simultaneously triangulable since they are mutually commuting \cite[see][p.199-204]{HoffKun}, hence $\A$ is triangulable. 
Next notice that any local ring is connected: if $R_1$ and $R_2$ are commutative rings together with $\m_1\in\Spm R_1$ and $\m_2\in\Spm R_2$, then $\m_1\times R_2$ and $R_1\times\m_2$ are two distinct maximal ideals of $R_1\times R_2$. 
Put together, this forces $\A$ to have the above form.
\end{proof}

\begin{remark}
	It is easily seen that the regular representation of a commutative upper-triangular algebra is the transposed, lower-triangular one (see also \cref{nicebasis} (ii)), which we will occasionally also call the stable lower-triangular representation due to the fact that it coincides with its own regular representation. 
	We will often use the lower-triangular representation as the more canonical one.
\end{remark}

\begin{corollary}
Let $\kappa$ be algebraically closed, let $\A=\bigoplus_{k=1}^M (\A_k,\m_k,\kappa)$ be a finite-dimensional $\kappa$-algebra decomposed into Artinian local $\kappa$-algebras and let $n_k\coloneqq\dim_\kappa\A_k$, $1\leq k\leq M$. 	
Identify $\A$ with its stable lower-triangular representation and let $Z=\bigoplus_{k=1}^M (z_k\oplus X_k)\in\A$, where $z_k\in \kappa$ and $X_k\in\m_k$, $1\leq k\leq M$.
Then:
	\begin{enumerate}
		\item We have $\tr(Z)=\sum_{k=1}^M n_k z_k$ and $\det(Z)=\prod_{k=1}^M z_k^{n_k}$.
		
		\item The spectrum of $Z$ is given by $\sigma(Z)\coloneqq\{z_1,\dots,z_M\}$, where each eigenvalue $z_k$ has algebraic multiplicity $n_k$.
	\end{enumerate}
\qed
\end{corollary}

Finally, let us mention a result of Schur and Jacobson: 

\begin{lemma}[Schur-Jacobson]\label{maxrepdim}
	If $K$ is an arbitrary field, then the maximal dimension of a commutative $K$-algebra of matrices of size $n$ is $\lfloor n^2/4\rfloor+1$.
\end{lemma}

\begin{proof}
	We refer the reader to \cite{Jac} for the classical source or to \cite{Mrz} for a more modern and proof-efficient treatment.
\end{proof}

Let us now determine the structure of morphisms $\A\xrightarrow{\varphi}\B$ of finite-dimensional commutative $K$-algebras.
Let $\B=\bigoplus_{\ell=1}^N \B_\ell$ be the decomposition of $\B$ into Artin local $K$-algebras.
Then $\varphi$ is the same as an $N$-tuple of morphisms $\varphi_\ell: \A\to\B_\ell$, where $\varphi_\ell = \pr^\cB_\ell\circ\varphi$ and $\pr^\cB_\ell:\B\twoheadrightarrow\B_\ell$ is the canonical projection, $1\leq\ell\leq N$, i.e. $\varphi(a)=(\varphi_1(a),\dots,\varphi_N(a))$.
Therefore it suffices to assume that $\B$ is local, in particular connected.
We have:

\begin{lemma}[Ring morphisms and decomposition]\label{idempotents} Let $R\xrightarrow{\varphi}S$ be a non-zero morphism in $\CRing$. Then:
	\begin{enumerate}
		\item $R=\bigoplus_{k=1}^M R_k$ decomposes as a finite direct product (sum) of non-trivial rings ($M\geq 2$) if and only if $R$ contains a non-trivial complete system of $M$ distinct mutually orthogonal idempotents $\{e_1,\dots,e_M\}$, i.e. $\sum_{k=1}^M e_k = 1_R$ and $\forall 1\leq k,\ell\leq M: e_k\neq 0,1$, $e_k e_\ell = \delta_{k \ell} e_k$. 
		In this case we have $\forall 1\leq k\leq M: R_k = R e_k$.
		
		\item If $R=\bigoplus_{k=1}^M R_k$ and $S$ is connected (for example if $S$ is local), then there exists exactly one $1\leq k\leq M$ such that  $\varphi$ factors as $\varphi = \bar{\varphi}\circ\pr_k$ for some ring homomorphism $\bar{\varphi}:R_k\to S$, where $\pr_k$ denotes the $k$-th canonical projection.		
		That is, we have a commutative diagram
		\[
		\begin{tikzcd} 
		R \arrow[d,two heads,swap,"\pr_k"] \arrow[r,"\varphi"] & S \\
		R_k  \arrow[ur,dashed,swap,"\bar{\varphi}"] &
		\end{tikzcd}
		\]		
	\end{enumerate}
\end{lemma}

\begin{proof}
	To (1): This is well-known. 
	For example, it follows by induction from \cite[][Prop. 1.10]{AM}, since any sum of distinct non-trivial mutually orthogonal idempotents is again idempotent.
	
	To (2): First of all, notice that $\varphi\neq 0$ implies $S\neq 0$. 
	Put $f_k\coloneqq \varphi(e_k)$, $1\leq k\leq M$. 
	Then $\forall 1\leq k,\ell\leq M: f_k f_\ell = \varphi(e_k e_\ell) = \delta_{k\ell} f_k$ and $\sum_{k=1}^M f_k = \varphi(\sum_{k=1}^M e_k) = 1_S$. 
	Thus we only need check the conditions $f_k\in\{0,1_S\}$.
	Suppose that $\forall 1\leq k\leq M: f_k=0$, then $0=\sum_{k=1}^M f_k = 1_S$, i.e. $S=0$, a contradiction. 
	Hence not all $f_k$ are 0.
	Therefore, if no $f_k$ is $1_S$ either, we obtain a non-trivial complete subsystem of distinct mutually orthogonal idempotents, which is a contradiction to $S$ being connected.
	Thus at least one $f_k$ is $1_S$. 
	Now suppose that there are two of them, i.e. that there exist indices $k\neq\ell$ such that $f_k=f_\ell=1_S$. 
	Then $1_S = f_k f_\ell = 0$, again a contradiction.
	Therefore there is exactly one $1\leq k\leq M$ such that $f_k = 1_S$. 
	Without loss of generality let that be $f_1 = 1_S$.
	We have $\forall 2\leq k\leq M: 0 = f_1 f_k = f_k=\varphi(e_k)$, hence $e_2,\dots,e_M \in \ker\varphi$.
	Thus, if we put $I\coloneqq \sum_{k=2}^M R e_k \subseteq \ker\varphi$, then $\varphi$ factors through $ R/I \cong (\bigoplus_{k=1}^M R e_k)/\sum_{k=2}^M R e_k \cong R e_1 = R_1$ as desired.
\end{proof}

\begin{corollary}[Canonical Factorization of Homomorphisms of Artinian Rings]\label{MorFact}
	Let $A \xrightarrow{\varphi} B$ be a homomorphism of Artinian rings and let $A \cong \bigoplus_{k=1}^M (A_k,\fm_k)$ and $B \cong \bigoplus_{\ell=1}^N (B_\ell,\fn_\ell)$ be their respective decompositions into Artin local rings.
	Then up to index permutation there exists a uniquely determined mapping $\tau\coloneqq\tau_\varphi:\{1,\dots,N\} \to \{1,\dots,M\}$, $\ell\mapsto \tau(\ell)$, such that we have a commutative diagram
	\begin{equation}\label{morfact}
	\begin{tikzcd}[column sep = scriptsize]
	\bigoplus_{k=1}^M A_k \ar[r,"\varphi"] \ar[two heads,d,"\Pi",swap] & \bigoplus_{\ell=1}^N B_\ell \\
	\bigoplus_{\ell=1}^N A_{\tau(\ell)} \ar[ur,"\exists_1 \bar{\varphi}",swap,dashed]
	\end{tikzcd}
	\end{equation}
	where $\Pi\coloneqq\Pi_\tau\coloneqq (\pr^A_{\tau(1)},\dots,\pr^A_{\tau(N)}): (a_1,\dots,a_M) \mapsto (a_{\tau(1)},\dots,a_{\tau(N)})$ is the obvious ring epimorphism and $\bar{\varphi} =  \oplus_{\ell=1}^N \bar{\varphi}_\ell$ for some local homomorphisms $\bar{\varphi}_\ell: (A_{\tau(\ell)},\fm_{\tau(\ell)}) \to (B_\ell,\fn_\ell)$, $1\leq \ell\leq N$, of local Artinian rings, that is, $\varphi = (\bar{\varphi}_1 \circ \pr^A_{\tau(1)},\dots,\bar{\varphi}_N \circ \pr^A_{\tau(N)})$.
	Moreover, if $\varphi\in\Mor(\fdCAlg_K)$, then also $\bar{\varphi}\in\Mor(\fdCAlg_K)$.
\end{corollary}

\begin{proof}
	This follows from the previous discussion, so we only need to show that $\bar{\varphi}_\ell$ are local, $1\leq\ell\leq N$.
	Indeed, more generally, if $\psi: (R,\fm) \to (S,\fn)$ is a homomorphism of local rings with maximal ideals of nilpotent elements, then $\psi$ is necessarily itself local: one clearly has $\psi(\m) \subseteq \n$ since $\forall x\in\m: \psi(x)$ is nilpotent.
	$K$-linearity of $\bar{\varphi}$ follows directly from the proof of \cref{idempotents} if we suppose $K$-linearity of $\varphi$ there.
\end{proof}

\begin{remarks}\ 
	\begin{enumerate}[topsep=-\parskip]
		\item Conversely, any mapping of indices $\tau:\{1,\dots,N\} \to \{1,\dots,M\}$ together with a family of ring homomorphisms $\varphi_\ell: A_{\tau(\ell)} \to B_\ell$, $1\leq\ell\leq N$, gives again a morphism $\varphi\coloneqq (\varphi_1 \circ \pr_{\tau(1)},\dots,\varphi_N \circ \pr_{\tau(N)}): A \coloneqq \bigoplus_{k=1}^M A_k \to B \coloneqq \bigoplus_{\ell=1}^N B_\ell$ of commutative rings.
		
		\item If $\fA$ is a non-unital finite-dimensional commutative associative $K$-algebra, then we recover all of the above structure theory by adjoning a unit.
		In particular, $\fA$ decomposes as $\fA=\bigoplus_{k=1}^N \fA_k$ into connected  $K$-algebras, some of which may be unital, and $\fA$ has stable lower-triangular representation. 
		Moreover, if $\fA$ is a connected non-unital $K$-algebra, then by adjoining a unit one obtains a local $\A=(\A,\m)$ with $\m=\fA$.
		
		\item If $\psi: (R,\fm,\kappa) \to (S,\fn,\lambda)$ is a morphism of local rings with maximal ideals of nilpotent elements, then $\psi$ induces a well-defined extension $\bar{\psi}:\kappa \hookrightarrow \lambda$ of residue fields, and we have a commutative diagram
		\[
		\begin{tikzcd}
		(R,\fm) \ar[r,"\psi"] \ar[d,two heads,"\pi_R",swap] & (S,\fn) \ar[d,two heads,"\pi_S"] \\
		\kappa \ar[r,hook,"\bar{\psi}"] & \lambda
		\end{tikzcd}
		\]
		where $\pi_R$ and $\pi_S$ denote the canonical quotient projections.
		In particular, if $\varphi: (\cA,\fm,K) \to (\cB,\fn,K)$ is a morphism of local $K$-algebras with maximal ideals of nilpotent elements, then we have a commutative diagram
		\begin{equation}\label{scalarcom}
		\begin{tikzcd}[column sep = tiny, row sep = scriptsize]
		(\cA,\fm,K) \ar[two heads,dr,"\sigma_\cA",swap] \ar[rr,"\varphi"] & & (\cB,\fn,K) \ar[two heads,dl,"\sigma_\cB"] \\
		& K &
		\end{tikzcd}
		\end{equation}
		where $\sigma_\cA$ and $\sigma_\cB$ denote the canonical quotient projections.
	\end{enumerate}
\end{remarks}

\subsection{The Structure of $U(\A)$}\label{subs:UA}

For any finite-dimensional associative unital $K$-algebra, $U(\A)$ is an affine algebraic group (and a rational variety). 
If $K=\KK$, then $U(\A)$ is also a Lie group, open and dense in $\A$. 
Since we are only interested in the commutative kind, we can take an elementary, ad hoc approach in determining $U(\A)$. 
If $\A$ is a finite-dimensional unital commutative associative $K$-algebra, then $\A$ decomposes as $\A \cong \A_1 \times \dots \times \A_M$ of commutative Artin local $K$-algebras, hence $U(\A)\cong U(\A_1)\times\dots\times U(\A_M)$. 
Therefore it suffices to assume that $(\A,\m,\kappa)$ is a finite-dimensional (Artin) local $K$-algebra.
Note, however, that in general the $\cA_k$-s can have different residue fields.

\begin{proposition}\label{Ustruct}
	Let $(\A,\m,\kappa)$ be a local finite-dimensional commutative associative unital $K$-algebra. 
	Then
	\[
	U(\A)\xrightarrow{\cong}\kappa^{\times}\times(\m,+),\ u+x\mapsto\left(u,\log\left(1+\frac{x}{u}\right)\right),
	\] 
	which is a rational map. \qed
\end{proposition}

\begin{proof}
	By \cref{structext}, the $K$-algebra structure of $\A$ extends to a $\kappa$-algebra structure and $\A\cong\kappa\oplus\m$ as vector $\kappa$-spaces. 
	In particular, every unit in $\A$ can be uniquely written as $u+x=u\left(1+\frac{x}{u}\right)$ for some  $u \in \kappa \setminus \{0\} = \kappa^{\times}$ and $x\in\m$ (nilpotent). 
	Moreover, the natural epimorphism $\sigma_\cA:\A\twoheadrightarrow\A/\m = \kappa$ induces a short exact sequence of abelian groups:
	\[
	1\rightarrow 1+\m \rightarrow U(\A) \rightarrow \kappa^{\times} \rightarrow 1,
	\]
	which splits non-canonically: we have an isomorphism of (abelian) groups $U(\A)\cong\kappa^{\times}\oplus(1+\m,\times)$, $u+x\mapsto\left(u,1+\frac{x}{u}\right)$, induced by the algebra structure. 
	Next we show that $(1+\m,\times)\cong (\m,+)$ as groups. 
	We use a little trick involving formal power series and nilpotents. 
	Recall the formal power series
	\[
	\exp(T)\defeq\sum_{k=0}^\infty \frac{T^k}{k!}\in K[[T]] \text{ and }
	\log(1+T)\defeq\sum_{k=1}^\infty \frac{(-1)^{k+1}}{k}T^k\in K[[T]],
	\]
	which are formal inverses to each other.
	Then clearly $\forall x\in\m: \log(1+x)\in\m$ and $\exp(x)\in 1+\m$ are well-defined as $x$ is nilpotent, and we get homomorphisms $\exp$ and $\log$ of abelian groups such that $\exp(\log(1+x))=1+x$.
\end{proof}

\begin{corollary}\label{UstructC}
	Let $\kappa$ be an algebraically closed field and let $\A$ be a finite-dimensional commutative associative unital $\kappa$-algebra. 
	Then $U(\A)\cong(\kappa^{\times})^M \times (\kappa^{+})^N$ for some $M\in\mathbb{N}$ and $N\in\mathbb{N}_0$, where $M=\#\Spec\A=\#\Spm\A$. \qed
\end{corollary}

\begin{definition}
	$U(\A)$ is called for short of type $(M,N)$ if $U(\A)\cong\GG_m^M\times\GG_a^N$.
\end{definition}

Conversely, any $(M,N)$-type occurs as the group of units of some $\A\in\fdCAlg_{\kappa}$, for instance\footnote{example due to Julian Rosen, see \cite{ConnectedUnits}} 
\[
\A\coloneqq \kappa^M \times \kappa[X_1,\dots,X_N]/(X_1,\dots,X_N)^2
\]
In particular, there exist finite-dimensional local $\kappa$-algebras $(\A,\m,\kappa)$ with an arbitrary (finite) $\dim_\kappa\m$.

\begin{proposition}[Path-connectedness of $U(\A)$]\label{complex}
	Let $\A$ be a finite-dimensional commutative associative unital $\RR$-algebra.
	Then $\pi_0(U(\A))$ is trivial if and only if $\A$ has a complex structure, extending the real one.
\end{proposition}

\begin{proof}
	“$\Leftarrow$”: Clear by \cref{UstructC}.
	
	“$\Rightarrow$”: By Artin decomposition and \cref{Ustruct} we have
	\[
	U(\A)\cong\prod_{i=1}^M U(\A_i) \cong \prod_{i=1}^M (\kappa^{\times}_i\times\m_i),
	\]
	where $\kappa_i$ is either $\RR$ or $\CC$ and the isomorphism is of topological groups (it is even given by rational functions). 
	Since $U(\A)$ is assumed path-connected, no copies of $\RR^{\times}$ can occur in above representation.
\end{proof}

%

\subsection{Matrix Analysis: Norms and Spectral Radius}

Our main reference for this chapter is the book \cite{HornJohn}. 
By a vector norm on $\rM_n(\KK)$ or any finite-dimensional associative $\KK$-algebra $\fA$ we will mean a norm $\norm{\cdot}$ that turns the underlying \emph{vector space} structure of $\fA$ into a normed space, hence a Banach space over $\KK$ by finite dimensionality.
By a matrix or an algebra norm on $\fA$ we will mean a vector norm $\norm{\cdot}$ on $\fA$ that is also submultiplicative, that is, $\forall A,B\in\fA: \norm{AB}\leq\norm{A}\norm{B}$, i.e. it turns $\fA$ into a Banach $\KK$-algebra.
%

A vector norm $\norm{\cdot}$ on $\A\in\fdAlg_\KK$ is called normalized or unital iff $\norm{1_\A}=1$.
Notice that for an arbitrary submultiplicative norm on $\A$ we only have $\norm{1_\A}\geq 1$.
Finite-dimensional associative $\KK$-algebras $\fA$ usually enjoy various (automatically topologically equivalent) submultiplicative norms, for example inherited from $\rM_n(\KK)$ along a faithful representation $\rho:\fA\hookrightarrow\rM_n(\KK)$.
Despite being topologically equivalent, some algebra norms are more equal than others, depending on the application. We list some.
One particularly close to geometric intuition is given by:

\begin{lemma-definition}[$\ell_2$-norm, Frobenius / Schur / Hilbert-Schmidt norm]
	For $A\in\rM_n(\KK)$ the expression
	\[
	\norm{A}_\mathrm{F}\coloneqq
	\bigg(\sum_{i,j=1}^{n}\abs{a_{ij}}^2\bigg)^{\frac{1}{2}}
	\]
	defines a matrix norm on $\rM_n(\KK)$, called the Frobenius norm. \qed
\end{lemma-definition}

Indeed this coincides with the Euclidean norm on the vector space $\rM_n(\KK)$ (with the usual basis). 
However, in general there is no reason to expect that it restricts to the Euclidean norm on the vector space $\fA\cong\KK^n$, e.g.
\[
\norm{
	\begin{pmatrix}
	z & w\\
	0 & z
	\end{pmatrix}
}^2_\mathrm{F} = 2\abs{z}^2+\abs{w}^2 \neq \abs{z}^2+\abs{w}^2.
\]

%

%

A sometimes technical disadvantage of ``geometric'' (read $\ell_p$-related) algebra norms is that, somewhat ironically, they are usually not normalized: for example if $E_n$ is the identity matrix of $\rM_n(\KK)$, then $\norm{E_n}_\rF=\sqrt{n}$.

\begin{lemma-definition}[Induced Operator Norm]
	Let $(V,\norm{\cdot})$ be a finite-dimensional normed $\KK$-vector space.
	Then the induced by $\norm{\cdot}$ operator norm on $\End_\KK(V)\cong\rM_n(\KK)$ given by
	\[
	\norm{A}_\op\coloneqq\sup_{v\in V}\frac{\norm{Av}}{\norm{v}}
	\]
	is a normalized submultiplicative norm. \qed
\end{lemma-definition} 

For the moment we are only concerned with the existence of normalized norms, and the above lemma grants us that.

Perhaps the most essential drawback of algebra norms is the fact that they can be only submultiplicative, which breaks many of the standard proofs in complex analysis as we shall see.

\begin{proposition}[Gelfand-Mazur]\label{gelfandmazur}
	Let $\A$ be a normed division $\CC$-algebra. 
	Then $\A\cong\CC$ isometrically. \qed
\end{proposition}

\begin{corollary}[Most $\CC$-algebra norms are only submultiplicative]
	Let $(\A,\norm{\cdot}_\A)\in\BanAlg_\CC$ such that $\norm{\cdot}_\A$ is multiplicative. 
	Then $\A\cong\CC$ isometrically.
\end{corollary}

\begin{proof}
	This is well-known, so we only sketch the proof. 
	One shows that, if we have $\forall x\in \A^\times: \norm{x^{-1}}\leq\norm{x}^{-1}$, then $\A^\times$ is closed in $\A\setminus\{0\}$.
	But clearly it is also open in $\A\setminus\{0\}$. 
	Since $\A$ is complex, $\A\setminus\{0\}$ is connected, hence $\A^\times=\A\setminus\{0\}$, i.e. $\A$ is a division algebra and one can apply \cref{gelfandmazur}.
\end{proof}

There is also a real version of Gelfand-Mazur:

\begin{proposition}[A ``Frobenius Theorem'' for normed algebras]
	Let $(\A,\norm{\cdot}_\A)$ be a normed division $\RR$-algebra.
	Then $\A\cong\RR, \CC$, or $\HH$ isometrically. \qed
\end{proposition}

However, for a $(\A,\norm{\cdot}_\A)\in\BanAlg_\RR$ with a multiplicative norm the proof of the last corollary does not apply since $\A\setminus\{0\}$ need not be connected, and it no longer follows that $\A\in\{\RR,\CC,\HH\}$.
Indeed, there are infinite-dimensional real Banach algebras with multiplicative norms  (also termed ``absolute values'') \cite[see][]{UrbWri}.
On the other hand, in the finite-dimensional setting the full list of all real unital, but not necessarily associative algebras with multiplicative norms is rather short, being given by $\{\RR,\CC,\HH,\OO\}$ as shown in \cite{Alb}. 
But in that setting at least all norms are equivalent.

Often times any sufficiently abstract statement involving some algebra norm is automatically valid for all algebra norms (finite dimensionality assumed). 
With the help of the next proposition this suggests that the correct notion to consider in many instances instead is the spectral radius, even though it itself is not a norm.

\begin{definition}[Spectral Radius]
	Let $A\in\rM_n(\KK)$ and let $\sigma(A)$ denote its spectrum. 
	Then
	\[
	\rho(A)\coloneqq\max\{\abs{z}:z\in\sigma(A)\}
	\]
	is called the spectral radius of $A$.
\end{definition}

\begin{proposition}\label{normexist}
	Let $\norm{\cdot}$ denote an algebra norm on $\rM_n(\CC)$. Then: 
	\begin{enumerate}
		\item $\forall A\in\rM_n(\CC): \norm{A}\geq\rho(A)$;
		
		\item $\forall A\in\GL_n(\CC):\norm{A^{-1}}\geq\rho(A)^{-1}$;
		
		\item $\forall A\in\rM_n(\CC)\ \forall \vareps>0\ \exists\norm{\cdot}: \rho(A)\leq\norm{A}\leq\rho(A)+\vareps$; 
		
		\item $\forall A\in\rM_n(\CC): \rho(A)=\inf\{\norm{A}: \norm{\cdot} \text{matrix norm}\} =\inf\{\norm{A}_\op: \norm{\cdot}_\op \text{induced\ operator\ norm}\}$.
	\end{enumerate}
\end{proposition}

\begin{proof}
	These are Theorem 5.6.9 and Lemma 5.6.10 in \cite{HornJohn}.
\end{proof}

Notice that $\norm{\cdot}$ in (3) depends not only on $\vareps$ but also on $A \in \rM_n(\CC)$. 
The following is a trivial consequence of (3) that we record here because it will often be used in one form or another:

\begin{corollary}[Principle of spectral comparison]
	Let $A,B\in\rM_n(\CC)$. 
	\begin{enumerate}
		\item If for all matrix norms $\norm{A}\leq\norm{B}$, then $\rho(A)\leq\rho(B)$.
		
		\item Conversely, if $\rho(A)<\rho(B)$, then there exists a matrix norm such that $\norm{A}<\norm{B}$.
	\end{enumerate}
\end{corollary}

\begin{proof}
	To (1): take $\inf_{\norm{\cdot}}$ on both sides.
	
	To (2): Let $\vareps>0$ such that $\rho(A)+\vareps<\rho(B)$. 
	Then by \cref{normexist} there exists an algebra norm $\norm{\cdot}$, depending on both $A$ and $\vareps$, such that $\norm{A}\leq \rho(A)+\vareps<\rho(B)\leq\norm{B}$.
\end{proof}

%
%

\begin{corollary}[Spectral radius is algebra-norm-like]
	Let $A,B\in\rM_n(\CC)$ such that $[A,B]=0$, then:
	\begin{enumerate}
		\item $\rho(E_n)=1$;
		\item $\forall z\in\CC: \rho(z A) = \abs{z}\rho(A)$;
		\item $\rho(A+B)\leq \rho(A)+\rho(B)$;
		\item $\rho(AB)\leq \rho(A)\rho(B)$; \qed
	\end{enumerate}
\end{corollary}

The only norm property that $\rho$ lacks here is due to the fact that $\rho(X)=0$ does not imply $X=0$.

\begin{lemma}[Spectral radius in $\fdCAlg_\CC$]
Let $\A\in\fdCAlg_{\CC}$ and identify $\A$ with its stable lower-triangular matrix representation. 
Let $\norm{\cdot}$ be an arbitrary algebra norm on $\A$ and denote $\rho_\A\coloneqq\rho|_\A$ the restriction of the spectral radius function.
We have:
	\begin{enumerate}
		\item If $\A=(\A,\m)$ is local and $\sigma_\cA:\A\twoheadrightarrow\A/\m\cong\CC$ is the canonical projection, then $\rho_\A = \abs{\cdot}\circ\sigma_\cA$. 
		
		\item If $\A=(\A,\m)$ is local, then $\m=\{Z\in\A:\rho_\cA(Z)=0\}$ and $U(\A) = \{Z\in\A:\rho_\cA(A)>0\}$. 
		In other words, $\rho_\A$ is ``pseudo-valuative''.
		
		\item Let $\A=\bigoplus_{k=1}^N\A_k$ be a decomposition of $\A$ into Artin local $\CC$-algebras and let $Z=\bigoplus_{k=1}^N (z_k\oplus X_k)\in\A$, where $z_k\in\CC$ and $X_k\in\m_k$, $1\leq k\leq N$. 
		Then $\rho_\cA(Z) = \max\limits_{1\leq k\leq N} \abs{z_k}$.
		In particular $\norm{Z}\geq\max\limits_{1\leq k\leq N}\abs{z_k}$ and $\norm{Z^{-1}}\geq\frac{1}{\min\limits_{1\leq k\leq N} \abs{z_k}}$.
	\end{enumerate}
In particular, $\rho_\A$ is multiplicative if and only if $\A$ is local.
\end{lemma}

\begin{proof}
	All of that follows directly from the lower-triangular representation of $\A$ and the previous discussion.
\end{proof}

Because of submultiplicativity of algebra norms, it is generally difficult to estimate norms of inverses from above.
The following standard fact helps:

\begin{lemma}[Upper Bound of the Norm of an Inverse]\label{norminverse}
	Let $\norm{\cdot}$ be a unital matrix norm on $\rM_n(\CC)$ and let $A\in\rM_n(\CC)$ with $\norm{A}<1$. Then:
	\begin{equation}\label{norminverseformula}
	\frac{1}{1+\norm{A}}\leq\norm{\frac{1}{1-A}}\leq\frac{1}{1-\norm{A}}.
	\end{equation}
\qed
\end{lemma}


In fact, when $\A=(\A,\m)$ is local, we can do better:

\begin{lemma}[Upper Bound of the Norm of Inverses in local $\A$]\label{norminverselocal}
	Let $\A=(\A,\m)$ be a local $\CC$-algebra equipped with a unital submultiplicative norm $\norm{\cdot}$. 
	Write $Z\coloneqq z\oplus X\in\A^\times$ with $z\in\CC^\times$ and $X\in\m$ and let $\nu\in\NN$ be the smallest integer such that $X^\nu = 0$.
	Then: 
	\begin{equation}
	\norm{Z^{-1}} \leq \frac{1}{\abs{z}} \bigg(\norm{1_\A}+\sum_{j=1}^{\nu-1}\bigg(\frac{\norm{X}}{\abs{z}}\bigg)^j\bigg)
	\end{equation}
	In particular, if $\norm{\cdot}$ is unital, we have:
	\begin{equation}\label{norminverselocalformula}
	\begin{aligned}
	\norm{Z^{-1}} &\leq
	\frac{1}{\abs{z}}\sum_{j=0}^{\nu-1}
	\left(\frac{\norm{X}}{\abs{z}}\right)^j=
	\begin{cases}
	\frac{1-\norm{X}^\nu/\abs{z}^\nu}{\abs{z}-\norm{X}}, &\text{if } \abs{z}\neq\norm{X}\\
	\frac{\nu}{\abs{z}}, &\text{if } \abs{z}=\norm{X}
	\end{cases}
	\defeq
	\begin{cases}
	\frac{1-\norm{X}^\nu/\rho_\cA(Z)^\nu}{\rho_\cA(Z)-\norm{X}}, &\text{if } \rho_\cA(Z)\neq\norm{X}\\
	\frac{\nu}{\rho_\cA(Z)}, &\text{if } \rho_\cA(Z)=\norm{X}.
	\end{cases}
	\end{aligned}
	\end{equation}
\end{lemma}

\begin{proof}
	Since $X\in\m=\nil\A$ with $X^\nu=0$, we have:
	\[
	Z^{-1} = z^{-1}\left(1-\left(\frac{-X}{z}\right)\right)^{-1} = \frac{1}{z}\sum_{j=0}^{\nu-1}\left(\frac{-X}{z}\right)^j,
	\]
	from which both claims follow.
\end{proof}

\begin{remarks}\ 
	\begin{enumerate}[topsep=-\parskip]
		\item Note that $\nu\in\NN$ can instead be chosen uniformly for the whole $\m$, since $\m$ itself is nilpotent, though this is of course less optimal than choosing for an individual $X \in \fm$.
		
		\item Indeed the estimate in \cref{norminverselocalformula} is not only valid for $\norm{X}\geq\abs{z}$ unlike \cref{norminverseformula}, but it is also tighter than \cref{norminverseformula} for the case $\norm{X}<\abs{z}$: since $\norm{X}<\abs{z}$, we have $1-(\norm{X}/\abs{z})^\nu<1$, hence
		\[
		\frac{1-\norm{X}^\nu/\abs{z}^\nu}{\abs{z}-\norm{X}} < \frac{1}{\abs{z}-\norm{X}},
		\]
		the latter being the estimate of $\norm{Z^{-1}}$ obtained by means of \cref{norminverseformula}.
	\end{enumerate}
\end{remarks}

When the ``depth'' of $Z\in (\cA,\fm)$ is less than its ``width'', we even have:

\begin{lemma}\label{norminverselocaluniform}
	Let $\A=(\A,\m)$ be a local $\CC$-algebra and let $Z=z\oplus X\in\A^\times$.
	If $\nu\in\NN$ is the smallest integer with $X^\nu=0$ and $\norm{X}\leq\abs{z}$, then
	\begin{equation}
	\norm{Z^{-1}} \leq \frac{\nu}{\abs{z}} \defeq \frac{\nu}{\rho_\cA(Z)}.
	\end{equation}
\end{lemma}

\begin{proof}
	By \cref{norminverselocal} it suffices to show that 
	\[
	\frac{1-\norm{X}^\nu/\abs{z}^\nu}{\abs{z}-\norm{X}} \leq \frac{\nu}{\abs{z}}
	\]
	whenever $\norm{X}\leq\abs{z}$.
	Indeed, putting $t\coloneqq \frac{\norm{X}}{\abs{z}}$, this is equivalent to
	\[
	p(t) \coloneqq t^\nu - t + \nu -1 \geq 0
	\]
	for $t\in [0,1]$.
	We have $p(0)=p(1)=\nu-1\geq 0$ as $\nu\geq 1$ and $p'(t) = \nu t^{\nu-1} -1$.
	If $\nu$ is even, then $t_0 \coloneqq (1/\nu)^{1/(\nu-1)}$ is the only critical point of $p(t)$, and $t_0\in (0,1)$ since $\nu\geq 2$.
	In this case $p''(t) = \nu (\nu-1) t^{\nu-2}$, hence $p''(t_0)>0$, and thus $t_0$ is a local minimum of $p(t)$.
	Moreover, we have $1/\nu \in (0,1) \Rightarrow (1/\nu)^{1/(\nu-1)} \in (0,1) \Rightarrow (1/\nu)^{1/(\nu-1)} > (1/\nu)^{\nu/(\nu-1)} \in (0,1) \Rightarrow 
	(1/\nu)^{1/(\nu-1)} - (1/\nu)^{\nu/(\nu-1)} \in (0,1) \Rightarrow p(t_0) = \nu-1 - ((1/\nu)^{1/(\nu-1)} - (1/\nu)^{\nu/(\nu-1)}) > 0$ since $\nu-1\geq 1$.
	It follows that $\forall t\in [0,1]: p(t)\geq 0$.
	
	If $\nu$ is odd, then without loss of generality $\nu\neq 1$ since for $\nu=1$ we have $p(t)\equiv\const=0$.
	For $\nu\geq 3$ odd $p'(t)$ has two roots $t_0>0$ as above and $(-t_0)$.
	But $t\in [0,1]$, hence the same discussion as above is applicable as $\nu>1$. 
\end{proof}

\begin{lemma-definition}[Direct Sum Norm]
	Let $\A\in\fdCAlg_\CC$ and let $\A=\bigoplus_{k=1}^N\A_k$ be the decomposition of $\A$ into Artin local $\CC$-algebras $(\A_k,\m_k)$, $1\leq k\leq N$.
	Let each $\A_k$ be equipped with a normalized submultiplicative norm $\norm{\cdot}_{\A_k}$, $1\leq k\leq N$ (we already know this to be possible).
	Then 
	\[
	\norm{\bigoplus_{k=1}^N Z_k}_\oplus\coloneqq
	\max\limits_{1\leq k\leq N}\norm{Z_k}_{\A_k}
	\]
	defines a normalized submultiplicative norm on the direct sum $\A$. \qed
\end{lemma-definition}

\begin{definition}[Spectral Balls and Spectral Annuli]
	Let $\A\in\BanAlg_\KK$, $Z_0\in\A$, and $R>r\geq 0$. Then:
	\begin{enumerate}
		\item Spectral ball: $\BB^\sp_\A(Z_0,R) \coloneqq \{Z\in\A: \rho_\A(Z-Z_0)<R\}$ and $\wbar{\BB}^\sp_\A(Z_0,R) \coloneqq \{Z\in\A: \rho_\A(Z-Z_0)\leq R\}$ are called the open and the closed spectral $R$-balls of $\A$ respectively.
		
		\item Spectral annulus: $\AA^\sp_\A(Z_0,r,R) \coloneqq \{Z\in\A: r<\rho_\A(Z-Z_0)<R\}$ and $\bar{\AA}^\sp_\A(Z_0,r,R) \coloneqq \{Z\in\A: r\leq\rho_\A(Z-Z_0)\leq R\}$ are called the open and the closed spectral $r$-$R$-annuli of $\A$ respectively.
	\end{enumerate}
\end{definition}

\begin{remark}
	Let $\A\in\BanAlg_\KK$ and $R>0$. 
	Then $\bar{\AA}^\sp_\A(Z_0,0,R) \defeq \wbar{\BB}_\A^\sp(Z_0,R)$.
\end{remark}
	
	\section{The Notion of $\varphi$-Holomorphy}\label{sec:phihol}
	

\textit{Unless explicitly stated otherwise, in the following all algebras are assumed finite-dimensional, commutative, and associative, \emph{but not necessarily unital}, and also equipped with some (any) submultiplicative norm. 
By a $\KK$-algebra $\fA$ we will always mean to take the maximal possible coefficient field $\KK$ for $\fA$.
Morphisms of unital algebras are automatically assumed to preserve the unit. $\KK^n$ will always have the standard basis. 
Furthermore, we make the following notational conventions: $\fA$, $\fB$, $\fC$ will always denote possibly non-unital algebras, whereas $\A$, $\B$, $\mathcal{C}$ will be reserved (strictly) for unital ones.}

Even though we are mostly interested in complex unital finite-dimensional commutative associative algebras, it is instructive to try and develop the basic theory for real, not necessarily unital finite-dimensional commutative associative algebras as far as possible in order to see where exactly unitality and the complex structure enter in a decisive manner.

Rather than working inside a single algebra $\fA$, our starting point is a morphism\footnote{following Grothendieck's point of view that one should look at morphisms instead of objects;} of $\KK$-algebras $\fA\xrightarrow{\varphi}\fB$.
Such a morphism turns $\fB$ into an $\fA$-algebra the usual way via $\forall a\in\fA\ \forall b\in\fB: ab \coloneqq \varphi(a)b$, and multiplication by $\fB$-elements on $\fA$ gives rise to a $\KK$-linear map $\fA\to\fB$ of \textit{$\KK$-vector spaces}.
Note that there exist non-trivial morphisms of finite-dimensional $\KK$-algebras: 

\begin{example}
	Let $\A\coloneqq \CC[X,Y]/(X^2+Y^3,XY^2,Y^4)$.
	Then $\A$ is a local $\CC$-algebra of $\dim_\CC\A=6$, and the quotient projection 
	\[
	q:\A\twoheadrightarrow\A/(\widebar{Y}^3)\cong\CC[X,Y]/(X^2,XY^2,Y^3)\eqqcolon\B
	\]
	with $\dim_\CC\B = 5$ is a non-trivial (necessarily local) homomorphism of local $\CC$-algebras. \qed
\end{example}

\begin{definition}[Fréchet $\varphi$-differentiability]\label{Frechet}
	Let $\fA\xrightarrow{\varphi}\fB$ be a morphism of $\KK$-algebras, $U\subseteq\fA$ open, $Z_0\in U$ a point, and $f:U\to\fB$ a map. Then: 
	\begin{enumerate}
		\item $f$ is called (Fréchet) $\varphi$-differentiable at $Z_0$ if there exists $B\in\fB$ such that
		\begin{equation}\label{frechet}
		f(Z_0+H) = f(Z_0) + B\varphi(H) + r(H),
		\end{equation} 
		where $\norm{r(H)}=o(\norm{\varphi(H)})$ as $H\to 0$.
		In this case $(Df)(Z_0)=f'(Z_0)\coloneqq B$ is called the derivative of $f$ at $Z_0$.
		
		\item $f$ is called $\varphi$-differentiable if it is $\varphi$-differentiable everywhere in $U$.
		
		\item $f$ is called $\varphi$-holomorphic at $Z_0$ if there exists an open neighbourhood $V\ni Z_0$ such that $f$ is $\varphi$-differentiable everywhere in $V$. 
		
		\item $f$ is called $\varphi$-holomorphic if it is $\varphi$-holomorphic everywhere in $U$. 
		
		\item The space of $\varphi$-holomorphic functions on $U$ will be denoted by $\O_{\varphi}(U)$.
		
		\item $f$ will also be called $\fA$-differentiable/-holomorphic (at $Z_0$) if $\varphi=\id_\fA:\fA\to\fA$. 
		In this case $\O_\fA(U)\coloneqq\O_{\id_\fA}(U)$.
	\end{enumerate}
\end{definition}

\begin{example}
	Let $Z\coloneqq z_1 + \eps z_2\in\A_2$ and define $f(Z) \coloneqq f(z_1,z_2) \coloneqq f_1(z_1,z_2) + \eps f_2(z_1,z_2)$, where
	\[
	\begin{aligned}
	f_1(z_1,z_2) &= z_1^3 - z_1^2 +1,\\
	f_2(z_1,z_2) &= 2z_1^3 + 3 z_1^2 z_2 + z_1^2 -2 z_1 z_2 + 3.
	\end{aligned}
	\]
	Then $f$ is $\A$-holomorphic: $f(Z) = (1+2\eps)Z^3 + (-1+\eps)Z^2 + (1+3\eps)$ with $f'(Z) = (3+6\eps)Z^2 + (-2+2\eps)Z$.
\end{example}

\begin{definition}[ $\varphi$-differentiability classically]\label{Classical}
	Let $\A\xrightarrow{\varphi}\B$ be a morphism of unital $\KK$-algebras, $U\subseteq\A$ open, $Z_0\in U$ a point, and $f:U\to\B$ a map. Then:	
	\begin{enumerate}
		\item $f$ is called classically $\varphi$-differentiable at $Z_0$ if the limit of the difference quotient
		\begin{equation}\label{classical}
		f'(Z_0)\coloneqq\lim_{\substack{H\to 0 \\ \ H\in\A^\times}}
		\frac{f(Z_0+H)-f(Z_0)}{\varphi(H)}
		\end{equation}
		exists.
		
		\item Classical $\varphi$- and $\A$-differentiability/-holomorphy (at $Z_0$) are defined in an analogous way as above.
	\end{enumerate}
\end{definition}

\begin{remarks}\
	\begin{enumerate}[topsep=-\parskip]
		\item Fréchet $\varphi$-differentiability is indeed a special case of Fréchet differentiability, that we also call \textit{inner} Fréchet differentiability because the linear operator $\L_B:\fA\to\fB$, $H\mapsto B\varphi(H)$, is given by inner\footnote{cfg. the inner automorphisms of a group are those given by conjugation with an element from the group;} multiplication in $\fB$. 
		In particular, the differential $Df=\L_B$ is $\fA$-linear (with respect to the $\fA$-module structure on $\fB$ induced by $\varphi$) due to the commutativity (and associativity) of the involved operations. 
		If $\fA=\A$ is unital, then all $\A$-linear maps are given that way because $\L(H)=\varphi(H)\L(1_\A)=\varphi(H)B$, where $B\coloneqq\L(1_\A)$.
		However, if $\fA$ is not unital, this is no longer the case, and inner Fréchet differentiability is stronger than mere $\fA$-linearity of $Df$: for instance, take the local algebra $\A_3\coloneqq\KK[X]/(X^3)$ with maximal ideal $\m=(\bar{X})$, put $\fA\coloneqq\m$, and define the $\m$-linear map $\L:\m\to\m$, $\bar{X}\to (1+\bar{X})\bar{X} = \bar{X}+\bar{X}^2$, given by multiplication with an element $1+\bar{X}\in\A^\times=\A\setminus\m$.
		Moreover, since $\fA$ is non-unitial, $\fA$-linearity of $Df$ does not in general guarantee $\KK$-linearity of $Df$.
		
		\item (Non-)Uniqueness of the $\varphi$-derivative\footnote{functions that have a unique derivative are sometimes called monogenic, though the term ``monogenic function'' has varying meaning throughout the literature;}: while the Fréchet derivative is always unique as a linear operator in $\Hom_\KK(\fA,\fB)$, an inner linear map between non-unital $\KK$-algebras is itself usually not uniquely represented. 
		For instance, consider the algebra $\fA_0 \coloneqq \CC \eps$ with $\eps^2 = 0$, then $0: \fA_0 \to \fA_0$ can be represented by multiplication with any $z\eps$, $z\in \CC$.
		More generally, if $\fM$ is a connected non-unital $\KK$-algebra, hence nilpotent, and $\nu\in\NN$ is the smallest integer such that $\fM^\nu = 0$, then $0: \fM \to \fM$ can be represented by multiplication with any element from $\fM^{\nu-1}$.
		On the other hand, if $\varphi: \cA \to \cB$ is a morphism of unital $\KK$-algebras, then $f'(Z) = B \in \cB$ in \cref{frechet} is uniquely determined as an element of $\cB$: if we have $f(Z+H) = f(Z) + B_{1,2} \varphi(H) + r_{1,2}(H)$ for two elements $B_1, B_2\in \cB$ and two $o(\norm{\varphi(H)})$-functions $r_1, r_2$ as $H \to 0$, then putting $B \coloneqq B_2 - B_1$ and $r(H) \coloneqq r_1(H) - r_2(H)$ we get $B \varphi(H) = r(H)$, where again $r(H) = o(\norm{\varphi(H)})$.
		Now taking $H\in\KK^\times \subseteq \cA$, $H \to 0$, we conclude that
		\[
		\norm{B} = \norm{\frac{r(H)}{\varphi(H)}} = \norm{\frac{r(H)}{H}} = \frac{o(\abs{H})}{\abs{H}} \xrightarrow{H \to 0} 0,  		
		\]
		where we have chosen $\norm{\cdot}$ to be itself normalized for convenience.
		Notice that the assumption of unitality of the algebras is essential to guarentee inclusion of the scalars and unconditional computation of the norm of the fraction.
		
		\item We remark that, since $\varphi$ is bounded (being a linear operator between finite-dimensional $\KK$-vector spaces), one also has $\norm{r(H)}=o(\norm{H})$ as $H\to 0$. 
		On the other hand, requiring only $\varphi(H)\to 0$ instead of $H\to 0$ would be insufficient, even though $\norm{r(H)}=o(\norm{\varphi(H)})$, since $\ker\varphi$ need not be trivial.
		
		
		\item If $f:U\to\fB$ is Fréchet $\varphi$-differentiable at $Z_0$, then, clearly, $f$ is totally $\KK$-differentiable at $Z_0$: even in the non-unital case, $f'(Z_0)$ being represented by an element $B\in\fB$ rather than being merely $\fA$-linear ensures that it is itself $\KK$-linear. 
		In particular, if $\KK=\CC$, then a $\varphi$-holomorphic function at $Z_0$ is also analytic at $Z_0$. 
		We are going to show later that if $\varphi: \cA\to\cB$ is a morphism of unital $\CC$-algebras, then the analytic expansion of $f$ is in fact of the special form $\B\{\varphi(Z)\}$ via an approach similar to the one used in the complex analysis of a single variable.
		
		\item Notice that the expression in \cref{classical} makes sense because $\A^\times$ is dense in $\A$. 
		If $\varphi$ is a morphism of unital $\KK$-algebras and $f$ is Fréchet $\varphi$-differentiable at $Z_0$, then restricting $H$ to $\A^\times$ in \cref{frechet} shows that $f'(Z_0)$ can be computed using \cref{classical}.
		
		\item In \cite{Water} it is claimed that $f$ is (Fréchet) $\A$-holomorphic at $Z_0$ if and only if $f$ is classically $\A$-holomorphic at $Z_0$. 
		The previous point shows that ``$\Rightarrow$'' is trivially valid already at the level of $\A$-differentiability at $Z_0$.  
		However, the converse direction ``$\Leftarrow$'' requires (classical) $\A$-differentiability in a neighbourhood.
		
		\item If $\A$ is infinite-dimensional, using \cref{Frechet} is in fact the only sensible approach (also used in \cite{Lor}) as one can no longer expect $\A^\times$ to remain dense in $\A$ \cite[see][]{DwsFnst}.
		
		\item However, since in general there also exist non-zero non-units, \cref{Classical} together with denseness of $\A^\times$ suggests that it is natural to try and extend the definition of (classical) $\varphi$-differentiability to other non-units: if $X\in\A\setminus\A^\times$, define
		\[
		f'_{(X)}(Z_0)\coloneqq\lim_{\substack{H\to X \\ \ H\in\A^\times}} \frac{f(Z_0+H)-f(Z_0)}{\varphi(H)},
		\]
		provided the limit exists. 
		It turns out that such limits exist when $f$ is $\varphi$-holomorphic at $Z_0$ and play a role in the proof of a generalized homological version of Cauchy's Integral Formula.
		
		\item Somewhat differently than the usual notions of differentiability, $\varphi$-differentiability of $f$ is entwined in both the domain and the codomain for it depends on a choice of a morphism $\varphi$. 
		This becomes apparent once we state the generalized Cauchy-Riemann equations for the morphism $\varphi$ (see \cref{CR}). 
		In contrast, the Cauchy-Riemann equations for holomorphic functions taking values in, say, any Banach space $f:\CC\to E$ are always the same.
		
		\item Essentially, the definition of $\varphi$-differentiability uses only the fact that $\fB$ is a topological $\fA$-module. 
		So, how does the algebra structure of $\fB$ come into play? An important feature of the definition is that if $f:U\to\fB$ is a $\varphi$-holomorphic function, then the derivative $f'$ is again a function $U\to\fB$. 
		This is in contrast to the situation when $\fA$ and $\fB$ are only $\KK$-vector spaces. 
		Thus one might suggest that a better analogy would be to take holomorphic functions with values in Banach spaces $f:\CC\to E$, since their derivatives retain the same domain and target $f':\CC\to E$. 
		Unfortunately, in general these cannot be composed with each other without further ado. 
		One can overcome this defi(ni)ciency by introducing holomorphic functions between Banach spaces $f:E_1\to E_2$ (via Fréchet derivatives) which then can be composed $E_1\to E_2\to E_3$, but their derivatives have targets again different than the original ones, $f':E_1\to BL(E_1,E_2)$. 
		We have come a full circle.
		
		\item It is a natural question whether the inclusion of a morphism $\varphi$ in the definition is trivial for whatever reasons. 
		There are two possibilities for this.
		Perhaps a $\varphi$-differentiable function $f:\fA\to\fB$ is always given as a composition $\varphi\circ g$ for an $\fA$-differentiable function $g:U\to\fA$ for some (open) $U\subseteq\fA$? The answer is negative: $f(Z)\coloneqq B\varphi(Z)$ for $B\notin\im\varphi$ clearly cannot be written that way. 
		The second possibility is that $f:U\to\fB$ perhaps lifts to a $\fB$-differentiable function $h:\varphi(U)\to\fB$ with $f=h\circ\varphi$. Notice that unless $\varphi$ is surjective, which requires $\dim_\KK\fB\leq\dim_\KK\fA$, $\varphi(U)$ is not open. 
		At least locally we are going to give a positive answer to the lifting question as a consequence of analyticity.
	\end{enumerate}
\end{remarks}

\begin{definition}[Canonical Projections]\ 
	\begin{enumerate}[topsep=-\parskip]
		\item Let $\fA\cong\bigoplus_{k=1}^N\fA_k$ be the decomposition of a not necessarily unital $\KK$-algebra $\fA$ into connected $\KK$-algebras. 
		Then
		\[
		\pr_k\coloneqq\pr^\fA_k:\fA\twoheadrightarrow\fA_k
		\]
		will denote the $k$-th canonical projection, $1\leq k\leq N$.
		
		\item Let $\A \cong \bigoplus_{k=1}^N\A_k$ be the decomposition of a unital $\KK$-algebra $\A$ into Artin local $\KK$-algebras $(\A_k,\m_k)$, where $\A_k=\KK_{(k)}\oplus\m_k$ as vector spaces and $\KK_{(k)}$ denotes the $k$-th copy of $\KK$. 
		Then
		\[
		\begin{tikzcd}
		\sigma_k\coloneqq\sigma^\A_k:\A \ar[two heads,r,"\pr_k"] & \A_k \ar[two heads,r,"\sigma_{\cA_k}"] & \A_k/\fm_k \cong \KK_{(k)}
		\end{tikzcd}
		\]
		will denote the $k$-the canonical (``spectral'') quotient projection.
	\end{enumerate}
\end{definition}

\begin{remark}
If $\cA$ is unital and $Z\in\A$, then $\sigma_k(Z)$ is precisely the $k$-th eigenvalue of $Z$.
Moreover, $\sigma_k$ is an epimorphism of $\KK$-algebras, projecting onto the scalars.
\end{remark}

Let us fix some further notations. 
Let $\fA \cong \bigoplus_{k=1}^M \fA_k$ and $\fB \cong \bigoplus_{\ell=1}^N \fB_\ell$ be two not necessarily unital $\KK$-algebras decomposed into connected $\KK$-algebras and let $\varphi:\fA \to \fB$ be a morphism of $\KK$-algebras.
Then by \cref{MorFact} there exists (up to index permutation) a unique mapping of indices $\tau: \{1,\dots N\} \to \{1,\dots,M\}$ and a unique morphism of $\KK$-algebras $\bar{\varphi} = \oplus_{\ell=1}^N \bar{\varphi}_\ell$ with $\bar{\varphi}_\ell: \fA_{\tau(\ell)} \to \fB_\ell$, $1\leq \ell\leq N$, such that we have a commutative diagram
\[
\begin{tikzcd}[column sep = scriptsize]
\bigoplus_{k=1}^M \fA_k \ar[r,"\varphi"] \ar[two heads,d,"\Pi_\tau",swap] & \bigoplus_{\ell=1}^N \fB_\ell \\
\bigoplus_{\ell=1}^N \fA_{\tau(\ell)} \ar[ur,"\exists_1\bar{\varphi}",swap,dashed]
\end{tikzcd}
\]
where $\Pi = \Pi_\tau = (\pr^\fA_{\tau(1)},\dots,\pr^\fA_{\tau(N)}): (Z_1,\dots,Z_M) \mapsto (Z_{\tau(1)},\dots,Z_{\tau(N)})$ is the obvious epimorphism of $\KK$-algebras.

\begin{definition}[Projection Closure and Polycylindrical Closure]\label{def:projclosure} 
Let $I \subseteq \{1,\dots,M\}$ be an index subset.
	\begin{enumerate}
		\item Projection closure of $U$ with respect to $I$: we shall denote by
		\[
		\what{U} \coloneqq \what{U}_I \coloneqq \bigcap_{k\in I} \pr_k^{-1}(\pr_k(U)) = \prod_{k=1}^M 
		\begin{cases}
		\pr_k(U), \text{ if } k\in I\\
		\fA_k, \text{ otherwise }
		\end{cases}
		\]
		the biggest subset of $\fA$ containing $U$ and satisfying $\forall k\in I: \pr_k(\what{U}_I) = \pr_k(U)$.
		
		\item Spectral polycylindrical closure of $U$ with respect to $I$: if $\fA_k = (\cA_k,\fm_k)$, $1\leq k\leq M$, and $\fB_\ell = (\cB_\ell,\fn_\ell)$, $1\leq \ell\leq N$, are unital local $\KK$-algebras, then we shall denote by
		\[
		\wtilde{U} \coloneqq \wtilde{U}_I \coloneqq \bigcap_{k\in I} \sigma_k^{-1}(\sigma_k(U)) =\prod_{k=1}^M
		\begin{cases}
		\sigma_k(U)\times\m_k, \text{ if } k\in I\\
		\cA_k, \text{ otherwise }
		\end{cases}
		\]
		the biggest subset of $\A\coloneqq\fA$ containing $U$ and satisfying $\forall k\in I: \sigma_k(\wtilde{U}_I) = \sigma_k(U)$.
		
		\item If $\fA \xrightarrow{\varphi} \fB$ is a morphism of $\KK$-algebras, we define
		\[
		\what{U} \coloneqq \what{U}_\varphi \coloneqq \what{U}_{\im\tau}.
		\]
		Moreover, if $\fA=\cA$ and $\fB=\cB$ are unital, then
		\[
		\wtilde{U} \coloneqq \wtilde{U}_\varphi \coloneqq \wtilde{U}_{\im\tau}.
		\]
	\end{enumerate}
\end{definition}

\begin{remarks}\ 
	\begin{enumerate}[topsep=-\parskip]
		\item Clearly, if $U$ is open, then so are $\what{U}$ and $\wtilde{U}$.
		Moreover, we have $U\subseteq \widehat{U}\subseteq \widetilde{U}$, $\widehat{\widehat{U}}=\widehat{U}$, $\widetilde{\widetilde{U}}=\widetilde{U}$, and $\widetilde{\widehat{U}}= \widehat{\widetilde{U}} = \widetilde{U}$.
		
		\item Let $\fA \cong \bigoplus_{k=1}^M \fA_k$ and $\fB \cong \bigoplus_{\ell=1}^N \fB_\ell$ be two $\KK$-algebras decomposed into a direct sum of connected $\KK$-algebras.
		Let $\tau: \{1,\dots,N\} \to \{1,\dots,M\}$ be an arbitrary mapping, let $\varphi_\ell: \fA_{\tau(\ell)} \to \fB_\ell$ be a morphism of $\KK$-algebras, $1\leq \ell\leq N$, and let $\bar{f_\ell}: U_{\tau(\ell)} \to \fB_\ell$ be a $\varphi_\ell$-differentiable function on an open subset $U_{\tau(\ell)} \subseteq \fA_{\tau(\ell)}$, $1\leq\ell\leq N$.
		Then 
		\[
		\varphi\coloneqq (\varphi_1\circ\pr_{\tau(1)},\dots,\dots,\varphi_N\circ\pr_{\tau(N)}): \fA \to \fB
		\]
		is a morphism of $\KK$-algebras and 
		\[
		f \coloneqq (\bar{f}_1\circ\pr_{\tau(1)},\dots,\bar{f}_N\circ\pr _{\tau(N)}): \bigcap_{\ell=1}^N \pr^{-1}_{\tau(\ell)}(U_{\tau(\ell)})\to\fB
		\]
		is a $\varphi$-differentiable function.
	\end{enumerate}
\end{remarks}

Conversely, all $\varphi$-differentiable functions can be written that way and then some:

\begin{lemma}\label{inclocal}
	Let $\varphi=(\varphi_1,\dots,\varphi_N):\fA = \bigoplus_{k=1}^M \fA_k \to \fB = \bigoplus_{\ell=1}^N \fB_\ell$ be a morphism of $\KK$-algebras that are decomposed into connected $\KK$-(sub)\footnote{in the category of \textit{non-unital} algebras we have canonical inclusions of algebras $\fA_{1,2} \hookrightarrow \fA_1\oplus\fA_2$;}algebras $\fA_k$, $1\leq k\leq M$, and $\fB_\ell$, $1\leq\ell\leq N$, respectively. 
	Let $U\subseteq\fA$ be open and let $f=(f_1,\dots,f_N):U\to\fB$ be a $\varphi$-differentiable function.
	Then:
	\begin{enumerate}
		\item $f$ is $\varphi$-differentiable if and only if $\forall 1\leq\ell\leq N: f_\ell:U\to\fB_\ell$ is $\varphi_\ell$-differentiable.
		
		\item Furthermore, $\forall 1\leq\ell\leq N\ \exists_1 \bar{f}_\ell: \pr_{\tau(\ell)}(U)\to\fB_\ell$ $\bar{\varphi}_\ell$-differentiable function such that the diagrams
		\[
		\begin{tikzcd}
		\fA \ar[two heads]{d}[swap]{\pr_{\tau(\ell)}} \ar{r}{\varphi_\ell} & \fB_\ell\\
		\fA_{\tau(\ell)} \ar[dashed]{ur}[swap]{\bar{\varphi}_\ell}
		\end{tikzcd}
		\Longrightarrow
		\begin{tikzcd}
		U \ar[two heads]{d}[swap]{\pr_{\tau(\ell)}} \ar{r}{f_\ell} & \fB_\ell\\
		\pr_{\tau(\ell)}(U) \ar[dashed]{ur}[swap]{\bar{f}_\ell}
		\end{tikzcd}
		\]
		are commutative.
		
		\item Automatic extension: thus, if $\bar{f} \coloneqq \oplus_{\ell=1}^N \bar{f}_\ell$, then $f = (\bar{f}_1 \circ \pr^\fA_{\tau(1)},\dots,\bar{f}_N \circ \pr^\fA_{\tau(N)}) \defeq \bar{f} \circ \Pi_\tau$, where $\bar{f}$ is a $\bar{\varphi}$-differentiable function, that is, the following diagrams
		\[
		\begin{tikzcd}[column sep = scriptsize]
		\fA \ar[r,"\varphi"] \ar[two heads,d,"\Pi_\tau",swap] & \fB \\
		\bigoplus_{\ell=1}^N \fA_{\tau(\ell)} \ar[ur,"\exists_1\bar{\varphi}",swap,dashed]
		\end{tikzcd}
		\Longrightarrow
		\begin{tikzcd}[column sep = scriptsize]
		U \ar[r,"f"] \ar[two heads,d,"\Pi_\tau",swap] & \fB \\
		\Pi_\tau(U) \ar[ur,"\exists_1\bar{f}",swap,dashed]
		\end{tikzcd}
		\Longrightarrow
		\begin{tikzcd}[column sep = scriptsize]
		\what{U} \ar[r,"\what{f}"] \ar[two heads,d,"\Pi_\tau",swap] & \fB \\
		\Pi_\tau(U) \ar[ur,"\exists_1\bar{f}",swap,dashed]
		\end{tikzcd}
		\]
		are commutative.
		In particular, $f$ extends uniquely to a $\varphi$-differentiable function $\what{f}:\what{U} \to \fB$ in a natural way.
		
		\item Reduction to inclusions of connected $\KK$-algebras: if $\fa\lhd\fA$ is such that $\fa\subseteq\ker\varphi$ and $q_\fa:\fA\twoheadrightarrow\fA/\fa$ denotes the canonical quotient projection, then $\exists_1 \wbar{\varphi}_\fa: \fA/\fa\to\fB$ morphism of $\KK$-algebras $\exists_1 \wbar{f}_\fa: q_\fa(U)\to\fB$ $\wbar{\varphi}_\fa$-differentiable:
		\[
		\begin{tikzcd}
		\fA \ar[d,two heads,swap,"q_\fa"] \ar[r,"\varphi"] & \fB \\
		\fA/\fa \ar[ur,dashed,swap,"\wbar{\varphi}_\fa"]
		\end{tikzcd}
		\Longrightarrow
		\begin{tikzcd}
		U \ar[d,two heads,swap,"q_\fa"] \arrow[r,"f"] & \fB \\
		q_\fa(U) \arrow[ur,dashed,swap,"\wbar{f}_\fa"]
		\end{tikzcd}
		\]
		are commutative diagrams.
		In particular, if $\fa=\ker\varphi$ and $q\coloneqq q_{\ker\varphi}$, we have commutative diagrams
		\[
		\begin{tikzcd}
		\fA \ar[d,two heads,swap,"q"] \ar[r,"\varphi"] & \fB \\
		\fA/\ker\varphi \ar[ur,hook,dashed,swap,"\wbar{\varphi}"]
		\end{tikzcd}
		\Longrightarrow
		\begin{tikzcd}
		U \ar[d,two heads,swap,"q"] \arrow[r,"f"] & \fB \\
		q(U) \arrow[ur,dashed,swap,"\wbar{f}"]
		\end{tikzcd}
		\]
	\end{enumerate} 
\end{lemma}

\begin{proof}
	To (1): It is immediate that $f$ is $\varphi$-differentiable at $Z\in U$ iff every $f_\ell$ is $\varphi_\ell$-differentiable at $Z$, $1\leq\ell\leq N$.
	
	To (2): Without loss of generality suppose for notational simplicity $\ell=1$ and $\tau(1)=1$, i.e. $\bar{\varphi}_1: \fA_1\to\fB_1$ and $\varphi = \varphi_1 \circ \pr_1$.
	Writing $H=H_1\oplus\bar{H}\in\fA$ and $Z=Z_1\oplus\bar{Z}\in\fA$, it follows from (1) that $f_1(Z+H) = f_1(Z) + B\varphi(H)+o(\norm{\varphi(H)}) = f_1(Z) + B\varphi(H_1) + o(\norm{\varphi(H_1)})$, hence letting $H_1\to 0$ yields $f_1(Z+\bar{H})=f_1(Z)$, in other words $f_1$ depends only on $Z_1$.
	Thus $f_1$ canonically gives rise to a function $\bar{f}_1(Z_1)\coloneqq f_1(Z)$ that is clearly $\bar{\varphi}_1$-differentiable at $Z_1$. 
	Notice that $\pr_1(U)$ is open since $\pr_1$ is an open map.
	
	To (3): $f=\bar{f} \circ \Pi_\tau$ follows from (2), while $\bar{\varphi}$-differentiability of $\bar{f}$ is immediate from the $\bar{\varphi}_\ell$-differentiability of $\bar{f}_\ell$ for all $1\leq\ell\leq N$.
	Extension to $\what{f}:\what{U}\to\fB$ follows from the diagram and $\varphi$-differentiability of $\what{f}$ follows from the discussion in the last remark.
	
	To (4):	If $H\in \fa\subseteq\ker\varphi$, then $f(Z+H) = f(Z) + B\varphi(H) + o(\norm{\varphi(H)}) = f(Z)$.
	In other words, if $\varphi(Z) = \varphi(W)$, then also $f(Z) = f(W)$. 
	Thus $f$ induces a well-defined function $\wbar{f}_\fa(Z\mod \fa)\coloneqq f(Z)$ that is readily verified to be $\wbar{\varphi}_\fa$-differentiable.
	Notice that $q_\fa(U)$ is open since $q_\fa$ is a projection.
\end{proof}

\begin{remarks}\ 
	\begin{enumerate}[topsep=-\parskip]
		\item In the case of $\KK=\CC$ and $\fA=\A$ and $\fB=\cB$ being unital, we will later show that $f\in\cO_\varphi(U)$ extends $\varphi$-holomorphically further to $\wtilde{U}$ in a natural way.
		
		\item Successive application of (1), (2), and (3) reduces the problem of studying $\varphi$-differentiable functions to the case of $\varphi:\fA\hookrightarrow\fB$ being an inclusion of connected $\KK$-algebras.
		In the presence of units, this means $\varphi:(\A,\m)\hookrightarrow(\B,\n)$ is an embedding of local (unital) $\KK$-algebras, which itself is automatically a local morphism.
		This suggests that the only interesting morphisms for the purposes of a function theory and our conjectural category $\fdCFkth$ are inclusions of local $\CC$-algebras, not unlike the category of field extensions\footnote{a bad, bad category with nice morphisms;} over a given base field.
		
		\item Thus the question of whether a $\varphi$-differentiable function $f:U\to\fB$, $U\subseteq\fA$ open, has the form $f=g\circ\varphi$ for a $\fB$-differentiable function $g$ turns into a question of whether $f$ is always a restriction of a $\fB$-differentiable function $g$, and in particular, if every $\fA$-differentiable function $f$ admits an extension to a $\fB$-differentiable function $g$ along an inclusion $\fA\subseteq\fB$ of $\KK$-algebras.
		For a morphism $\varphi:\A\to\B$ of unital $\CC$-algebras, this will become clear once we establish analyticity (\cref{analyt}) in $\varphi$.
		In other words we will have the following commutative diagram:
		\[
		\begin{tikzcd}[column sep=scriptsize, row sep=normal]
		\A \ar[d,hook,"\varphi",swap] \ar[r,hookleftarrow] & U \ar[d,hook,"\varphi",swap] \ar[r,"f"] & \B\\
		\B \ar[r,hookleftarrow] & V  \ar[ur,dashed,"\exists_1 g",swap]
		\end{tikzcd}
		\]
		Even though $\A$ can be a subspace of $\B$ of an arbitrary codimension, the $\varphi$-holomorphic function of several complex variables $f$ extends locally in a unique fashion to a $\B$-holomorphic function $g$.	
		
		\item A version of the decomposition of $\A$-differentiable function $f$ into $\A_k$-differentiable components $f_k$, $1\leq k\leq M$, where $\A=\bigoplus_{k=1}^M \A_k$ is the decomposition of $\A$ into Artinian local $\KK$-algebras, is given in \cite{Water}, but we are not convinced of his argument\footnote{we are confident to mention this in here since at any time an arbitrary master thesis is read only by a very few selected people, and moreover, almost no one reads footnotes;}: he proves this first for the case of products of copies of $\RR$ and $\CC$, where he uses the multiplicativity of the absolute value on $\KK$ in an essential way, and then later states that the same proof applies to the case of products of local $\KK$-algebras. However, algebra norms are in general only submultiplicative, which breaks the argument used in the first case.
		
		\item \cref{inclocal}, reducing to inclusions $\fA\subseteq\fB$, raises the question as to why one should still bother with (general) morphisms of $\KK$-algebras at all. 
		Here are some conceptual reasons for doing this:
		\begin{enumerate}[(i)]
			\item As soon as one has $\fA$-holomorphic functions $f:U\to\fA$ available, one realizes in particular that we have ridden ourselves of $\CC$ being the inevitable choice of a codomain that appears when studying complex spaces by means of the holomorphic functions $f:X\to\CC$ or $f:X\to\PP^1$ defined on them. 
			In other words, there is no longer a canonical choice of a codomain, and $\varphi$ enables the only necessary compatibility condition between domain and codomain while simultaneously also giving a notion of $\fA$-holomorphy with values in some (but certainly not all) finite-dimensional $\fA$-modules $\fB$. 
			
			\item Continuing the previous point: tautologically, without introducing $\varphi$-holomorphy, we cannot even speak of $\varphi$ being ``$\fA$-holomorphic with values in $\fB$'' (since it takes values in a different algebra, for one), even though it seems intuitively so, being a linear map and all, or, by extension, that a composition of a $\fB$-holomorphic map $f$ with $\varphi$ also has nice, holomorphy-like properties. 
			After all, in general, one does not expect (and rightly so) that a holomorphic map when composed with some ``arbitrary'' map would still remain equally well-behaved. 
			
			\item While $\CC$ has exactly two continuous $\RR$-linear automorphisms, $\Aut_\KK(\fA)$ can be trivial or very rich; in fact, in the spirit of the previous point, $\Hom_{\Alg_\KK}(\fA,-)$ is never boring. 
			Thus $\varphi$-holomorphy is a vast generalization of anti-holomorphy, but not as involutive. 
			
			\item Finally, as we shall see, the basic theory works seamlessly right away for $\varphi$-holomorphic functions, so there really is no point in always proving two separate versions of each result (or for a fact in even remarking so).   	
		\end{enumerate}
	\end{enumerate}
\end{remarks}

\begin{lemma}[The Jacobian of $f$]\label{jacobian}
	Let $U\subseteq\fA$ be open and let $f:U\to\fA$ be $\fA$-differentiable at $Z_0\in U$. Then:
	\begin{equation}\label{Jac}
	\lambda(f'(Z_0)) = (J_{\KK}f)(Z_0),
	\end{equation}
	where $\lambda:\fA\to\rM_n(\KK)$ denotes the regular representation of $\fA$.
	Thus, if we identify $\fA$ with its regular matrix representation, we have $f'(Z)=(J_\KK f) (Z)$.
\end{lemma}

\begin{proof}
	Let $\{a_1,\dots,a_n\}$ be a $\KK$-basis for $\fA$ and let $f=(f^1,\dots,f^n)$ be the components of $f$. We have $\forall 1\leq i\leq n:$
	\[
	f'(Z)a_i \defeq \pdv{f}{z^i} = \pdv{}{z^i}\sum_{r=1}^n f^r a_r = (a_1,\dots,a_n)\pdv{}{z^i}(f^1,\dots,f^n)^T,
	\]
	hence putting them all together we get $f'(Z)(a_1,\dots,a_n) = (a_1,\dots,a_n)(J_\KK f)(Z)$, i.e. $(J_\KK f)(Z)$ is the representation matrix of the multiplication by $f'(Z)$ in $\fA$ as required.
\end{proof}

\begin{remark}
	If $\varphi:\fA\to\fB$ is a morphism of $\KK$-algebras with $\dim_\KK\fA\neq\dim_\KK\fB$ and $f$ is $\varphi$-differentiable, then $f'$ and $J_{\KK}f$ have different sizes.
\end{remark}

\begin{corollary}[Jacobian Conjecture for $\A$-holomorphic regular maps]\label{jaconj}
	Let $\A=\bigoplus_{k=1}^N \A_k$ be a fully decomposed unital $\KK$-algebra such that $\forall 1\leq k\leq N: \A_k/\m_k\cong\KK$ and identify $\A\cong\KK^n$ as vector spaces.
	Then the Jacobian Conjecture holds trivially for $\A$-holomorphic regular maps: if $P=P(z_1,\dots,z_n)=(P_1,\dots,P_n): \KK^n\to\KK^n$ is a regular and $\A$-holomorphic map with $\det J_\KK P = \const\neq 0$, then $P$ is biregular.
\end{corollary}

\begin{proof}
	Since $\A_k/\m_k\cong\KK$ for all $1\leq k\leq N$, we can identify $\A$ without loss of generality with its lower-triangular representation: a change of basis for $\A$ translates into a composition of $P$ with invertible linear maps; in fact, even the value of the constant $\det J_\KK P$ itself does not change. 
	Then $\det P'(Z) = \det (J_\KK P)(Z)$. Let $P(Z)=\bigoplus_{k=1}^N Q_k(Z_k)$ with respect to the decomposition of $\A$.
	Then also $P'(Z) = \bigoplus_{k=1}^N Q'_k(Z_k)$, hence
	\[
	\det P'(Z) = \prod_{k=1}^{N}\det Q'_k(Z_k)\in\KK^{\times} = U(\KK[z_1,\dots,z_n])
	\]
	Therefore it suffices to assume that $\A=(\A,\m)$ is a local $\KK$-algebra with $\A/\m\cong\KK$. 
	By \hyperref[triangular]{lower-triangular form (\cref*{triangular})} of the local $\A$ we get:
	\[
	\forall i<j: \pdv{P_i}{z_j} = 0\ \text{and}\ \forall 1\leq i\leq n: \pdv{P_i}{z_i} = \pdv{P_1}{z_1},
	\]
	i.e. $\forall 1\leq i\leq n: P_i = P_i(z_1,\dots,z_i)$ is a polynomial depending only on the first $i$ variables. 
	Furthermore
	\[
	0\neq\const = \det (J_\KK P)(Z) = \det P'(Z) = \det \pdv{P}{z_1}(Z) = \left(\pdv{P_1}{z_1}\right)^n \Rightarrow P_1 = cz_1+d
	\]
	for some $c\in\KK^{\times}$ and $d\in\KK$, and more generally, $P_i = c z_i + p_i(z_1,\dots,z_{i-1})$ for some $p_i\in\KK[z_1,\dots,z_{i-1}]$. 
	Setting $w_i\coloneqq P_i$, one can now write the inverse map recursively as follows:
	\[
	z_1 = \frac{1}{c}(w_1-d),\ z_i = \frac{1}{c}(w_i - p_i(z_1,\dots,z_{i-1})),\ 2\leq i\leq n
	\]
\end{proof}

\begin{remark}
	Note that $\KK$ is allowed to be $\RR$ under the additional assumption that $\forall 1\leq k\leq N: \A_k/\m_k \cong \RR$.
	In general, however, the Jacobian conjecture is known to be false over $\RR$.
	Moreover, this is a special instance (the main diagonal consisting of the same element) of the case when $\det J_\KK P$ has triangular form, which is dealt with in a similar fashion.
	Note also that there exist examples of biregular maps with non-triangulable Jacobians.
\end{remark}

\begin{lemma}\label{rules}
	Let $\fA\xrightarrow{\varphi}\fB\xrightarrow{\psi}\fC$ be morphisms of $\KK$-algebras and $b\in\fB$. Let $f,f_1,f_2:U\to\fB$ $\varphi$-differentiable and $g:V\to\fC$ be $\psi$-differentiable such that $f(U)\subseteq V$. Then:
	\begin{enumerate}
		\item Linearity: $(bf)' = bf'$, $(bg)' = bg'$, and $(f_1+f_2)'=f_1'+f_2'$;
		
		\item Leibniz Rule: $(f_1 f_2)' = f_1'f_2 + f_1 f_2'$;
		
		\item Chain Rule: $g\circ f$ is $\psi\circ\varphi$-differentiable and $(g\circ f)'(Z_0) = g'(f(Z_0)) f'(Z_0)$;
		
		\item Constants: if $\fB=\B$ is unital, then $\varphi$ is $\varphi$-differentiable with $\varphi'= 1_\cB$.
	\end{enumerate}
\end{lemma}

\begin{proof}
	To (2): We have:
	\[
	\begin{aligned}
	f_1(Z_0+H)f_2(Z_0+H) &= \left(f_1(Z_0) + f'_1(Z_0)\varphi(H) + o(\norm{\varphi(H)})\right) \left(f_2(Z_0) + f'_2(Z_0)\varphi(H) + o(\norm{\varphi(H)})\right) = \\ 
	&= f_1(Z_0)f_2(Z_0) + \left(f_1(Z_0)f'_2(Z_0) + f'_1(Z_0) f_2(Z_0)\right)\varphi(H) \\
	&+ \underbrace{f'_1(Z_0)f'_2(Z_0)\varphi(H)^2}_{o(\norm{\varphi(H)})} +
	\underbrace{\left(f_1(Z_0) + f_2(Z_0) + \left(f'_1(Z_0) + f'_2(Z_0)\right)\varphi(H)\right)
	o(\norm{\varphi(H)})}_{o(\norm{\varphi(H)})} + o(\norm{\varphi(H)})^2
	\end{aligned}
	\]
	as $H\to 0$. 
	
	To (3): We have:
	\[
	\begin{aligned}
	g(f(Z_0+H)) &= g(f(Z_0) + \underbrace{f'(Z_0)\varphi(H) + o(\norm{\varphi(H)})}_{\to 0\ \text{as}\ H\to 0}) = \\ 
	&= g(f(Z_0)) + g'(f(Z_0))\left(f'(Z_0)\varphi(H) + o(\norm{\varphi(H)})\right) + o\left(\norm{f'(Z_0)\varphi(H) + o(\norm{\varphi(H)})}\right) =\\
	&= g(f(Z_0)) + g'(f(Z_0))f'(Z_0)\varphi(H) + \underbrace{g'(f(Z_0))o(\norm{\varphi(H)})}_{o(\norm{\varphi(H)})} + o\left(\norm{f'(Z_0)\varphi(H) + o(\norm{\varphi(H)})}\right).
	\end{aligned}
	\]
	For the last term we get:
	\[
	\begin{aligned}
	\lim_{H\to 0}\frac{o\left(\norm{f'(Z_0)\varphi(H) + o(\norm{\varphi(H)})}\right)}{\norm{\varphi(H)}} =
	\lim_{H\to 0}\frac{o(\norm{f'(Z_0)\varphi(H) + o(\norm{\varphi(H)})})}{\norm{f'(Z_0)\varphi(H) + o(\norm{\varphi(H)})}}
	\underbrace{\frac{\norm{f'(Z_0)\varphi(H) + o(\norm{\varphi(H)})}}{\norm{\varphi(H)}}}_{\text{bounded as}\ H\to 0} = 0,
	\end{aligned}
	\]
	i.e. $o\left(\norm{f'(Z_0)\varphi(H) + o(\norm{\varphi(H)})}\right) = o(\norm{\varphi(H)})$ as $H\to 0$.
	
	To (4):
	$\varphi(Z_0+H) = \varphi(Z_0)+1_\fB\varphi(H)$.
\end{proof}

\begin{corollary}
$\O_{\varphi}$ is a sheaf of $\fA$-algebras with a distinguished derivation.
\end{corollary}

\begin{example}[Sanity check]
	Let $\fA$ be connected and non-unital.
	Then $\exists\nu\in\NN\ \forall Z\in\fA: Z^\nu = 0$, but $\exists Z_0\in\fA: Z_0^{\nu-1}\neq 0$. 
	In particular, for $f(Z)\coloneqq Z^{\nu-1}$ we have $f\neq 0$.
	On the other hand, $0=(Z^\nu)' = \nu Z^{\nu-1} = \nu f(Z)$, even though $\ch\fA=0$.
	This seeming paradox is resolved by the fact that $\id_\fA:\fA\to\fA$ is actually not $\fA$-differentiable since $\fA$ does not possess a $1_\fA$, so Leibniz' rule does not apply successively to the function $Z^n$ on $\fA$. 
	On the other hand, if $A\in\fA$, then $g(Z)\coloneqq AZ$ is $\fA$-differentiable with $g'(Z) = A$. \qed
\end{example}
	
	\section{The Generalized Cauchy-Riemann Equations}\label{sec:CR}
	

We will need at first to make a superficial distinction between the real and complex case. 
This is most easily done away with by making the convention $\pd\coloneqq\d$ whenever $\KK=\RR$. 
Note that if $\KK=\CC$, this relation is actually implied since $\d=\pd+\pdbar$ and $\pdbar f=0$ as $f$ is in particular holomorphic in each of its (complex) variables. 
Let $(\fA,a_1,\dots,a_n)$ and $(\fB,b_1,\dots,b_m)$ be based $\KK$-algebras, not necessarily unital.

\begin{lemma}[Generalized Cauchy-Riemann-Scheffers\footnote{in honour of Georg Scheffers, who was the first to generalize them equations;} Equations] 
	Let $\fA\xrightarrow{\varphi}\fB$ be a morphism of $\KK$-algebras with structure tensor $(\gamma^i_{jk})_{1\leq j\leq n, 1\leq i,k\leq m}$, $U\subseteq\fA$ open, $Z\coloneqq z^1 a_1+\dots+z^n a_n \in U$ a general $\fA$-variable, $Z_0\in U$ a point, and $f\coloneqq f^1 b_1+\dots+f^m b_m:U\to\fB$ a map.
	\begin{enumerate}
		\item General form of $\varphi$-differentiability: $f$ is $\varphi$-differentiable at $Z_0$ if and only if $f$ is totally $\KK$-differentiable at $Z_0$ and there exists a function $g\coloneqq g^1 b_1+\dots+g^m b_m$ at $Z_0$ with values in $\fB$ such that $\forall 1\leq j\leq n\ \forall 1\leq i\leq m$:
		\begin{equation}\label{CR}
		\pdv{f^i}{z^j} = \sum_{s=1}^m \gamma^i_{js} g^s 
		\end{equation}
		in $Z_0$. 
		In this case, $f'(Z_0)=g(Z_0)$.\footnote{The functions $g^k$, $1\leq k\leq m$, will be the coordinates of the derivative, which should not be confused with the derivatives of the coordinate functions $f^i$, $1\leq i\leq n$.}
		
		\item Generalized Unital Cauchy-Riemann Equations \emph{(GUCR)}: Suppose that $\fA=\A$ is unital. Then $f$ is $\varphi$-differentiable at $Z_0$ if and only if $f$ is totally $\KK$-differentiable at $Z_0$ and $\forall 1\leq j\leq n\ \forall 1\leq i\leq m:$
		\begin{equation}\label{GUCR}
		\pdv{f^i}{z^j} = \sum_{r=1}^n \sum_{s=1}^m \eps^r \gamma^i_{js} \pdv{f^s}{z^r}
		\end{equation}
		in $Z_0$. In this case:
		\begin{equation}
		f'(Z_0) = \sum_{r=1}^n \eps^r \pdv{f}{z^r} (Z_0)
		\end{equation}
		
		\item Generalized Cauchy-Riemann Equations with Unit \emph{(GCRU)}: Suppose that $\fA=\A$ is unital with $a_1 = 1_\A$. Then $f$ is $\varphi$-differentiable at $Z_0$ if and only if $f$ is totally $\KK$-differentiable at $Z_0$ and $\forall 2\leq j\leq n$ $\forall 1\leq i\leq m:$
		\begin{equation}\label{GCRU}
		\pdv{f^i}{z^j} = \sum_{s=1}^m \gamma^i_{js} \pdv{f^s}{z^1} 
		\end{equation}
		in $Z_0$. In this case:
		\begin{equation}\label{pdunit}
		f'(Z_0) = \pdv{f}{z^1} (Z_0)
		\end{equation}
		In particular, if $\varphi=\id_\A$, then $\forall 2\leq j\leq n\ \forall 1\leq i\leq n$:
		\begin{equation}
		\pdv{f^i}{z^j} = \pdv{}{z^1} \sum_{r=1}^n \alpha^i_{jr} f^r
		\end{equation}
		at $Z_0$.
		
		\item Generalized Cauchy-Riemann-Scheffers Equations \emph{(GCRS)}: If $f$ is $\varphi$-differentiable at $Z_0$, then $f$ is totally $\KK$-differentiable at $Z_0$ and $\forall 1\leq j, k\leq n$ $\forall 1\leq i\leq m:$
		\begin{equation}\label{GCRS}
		\left(\sum_{r=1}^n \alpha^r_{jk} \pdv{}{z^r}\right)f^i = 
		\pdv{}{z^k} \left(\sum_{s=1}^m \gamma^i_{js} f^s\right) 
		\end{equation}
		in $Z_0$. The converse holds if $\fA=\A$ is unital.   
		
		\item Cauchy-Riemann-Scheffers Equations \emph{(CRS)}: Let $(\alpha^i_{jk})_{1\leq i,j,k\leq n}$ be the structure tensor of $\fA$. If $f:U\to\fA$ is $\fA$-differentiable at $Z_0$, then $f$ is totally $\KK$-differentiable at $Z_0$ and $\forall 1\leq i,j,k\leq n:$
		\begin{equation}\label{CRS}
		\left(\sum_{r=1}^n \alpha^r_{jk} \pdv{}{z^r}\right)f^i = \pdv{}{z^k}\left(\sum_{r=1}^n \alpha^i_{jr} f^r\right) 
		\end{equation}
		in $Z_0$. The converse holds if $\fA=\A$ is unital. 
	\end{enumerate}
\end{lemma}

\begin{proof}
	To (1): 
	Recasted in differential form(s), \hyperref[frechet]{inner Fr\'echet differentiability (\cref*{frechet})} is equivalent to 
	\[
	\d f=\pd f=g(Z)\d Z
	\] 
	(at $Z_0$) for some function $g(Z)$ defined in $Z_0$, and then $f'(Z_0)=g(Z_0)$ by definition. 
	We have:
	\[
	\pd f=\sum_{s=1}^m b_s \pd f^s =\sum_{s=1}^m \sum_{r=1}^n \pdv{f^s}{z^r} b_s\d z^r.
	\]
	On the other hand:
	\[
	\begin{aligned}
	g(Z)\d Z = \bigg(\sum_{t=1}^m b_t g^t\bigg)\bigg(\sum_{r=1}^n a_r \d z^r\bigg) = \sum_{t=1}^m \sum_{r=1}^n g^t a_r b_t \d z^r = 
	\sum_{t=1}^m \sum_{r=1}^n g^t \sum_{s=1}^m \gamma_{rt}^s b_s \d z^r = \sum_{s=1}^m \sum_{r=1}^n \bigg(\sum_{t=1}^m \gamma_{rt}^s g^t \bigg) b_s \d z^r.
	\end{aligned}
	\]
	Now, a direct comparison of the coefficients yields the desired result.
	
	To (2): By definition, we have $\forall 1\leq i\leq n:$
	\[
	\pdv{f}{z^i}=f'(Z)a_i.
	\]
	Using \cref{unitcoord}, we get:
	\[
	\sum_{r=1}^n \eps^r \pdv{f}{z^r} = \sum_{r=1}^n f'(Z) \eps^r a_r = f'(Z) \defeq g(Z).
	\]
	Therefore taking the $j$-th component gives
	\[
	g^j(Z) = \sum_{r=1}^n \eps^r \pdv{f^j}{z^r}
	\]
	for all $1\leq j\leq m$. 
	Now substituting this in \cref{CR} proves both directions.
	
	To (3): Special case of (2). Notice that for $j=1$ the system is vacuous since $\gamma^i_{1r} = \delta^i_r$.
	
	To (4):
	``$\Rightarrow$'': Using \cref{CR} and \cref{assA1}, we get:
	\[
	\begin{aligned}
	\sum_{r=1}^n \alpha^r_{jk} \pdv{f^i}{z^r} = \sum_{r=1}^n \alpha^r_{jk} \sum_{t=1}^m \gamma^i_{rt} g^t = 
	\sum_{t=1}^m g^t \sum_{r=1}^n \gamma^i_{rt} \alpha^r_{jk} = \sum_{t=1}^m g^t \sum_{s=1}^m \gamma^i_{js} \gamma^s_{kt} =
	\sum_{s=1}^m \gamma^i_{js} \sum_{t=1}^m \gamma^s_{kt} g^t = \sum_{s=1}^m \gamma^i_{js} \pdv{f^s}{z^k}.
	\end{aligned}
	\]
	We note that this only uses the multiplicativity of $\varphi$ and the associativity of the $\fA$-operation on $\fB$.
	
	``$\Leftarrow$'': If $\fA=\A$ is unital, we can use \cref{unit} to obtain:
	\[
	\begin{aligned}
	\sum_{t=1}^n \sum_{s=1}^m \eps^t \gamma^i_{js} \pdv{f^s}{z^t} = \sum_{t=1}^n \eps^t \pdv{}{z^t} \left(\sum_{s=1}^m \gamma^i_{js} f^s\right) =
	\sum_{t=1}^n \eps^t \left(\sum_{r=1}^n \alpha^r_{jt} \pdv{}{z^r}\right)f^i =
	\left(\sum_{r=1}^n \delta^r_j \pdv{}{z^r}\right)f^i = \pdv{f^i}{z^j}
	\end{aligned}	
	\]
	as required.
	
	To (5):
	Special case of (4) with $\fA\xrightarrow{\mathrm{id}}\fA$.
\end{proof}

\begin{remarks}\ 
	\begin{enumerate}[topsep=-\parskip]
		\item Regardless of $\KK$, total $\KK$-differentiability implies existence of the partial $\KK$-derivatives, so it is in any case sufficient. Depending on $\KK$, however, the assumption of total differentiability can be weakened appropriately. 
		If $\KK=\CC$, then by Hartogs' theorem partial complex differentiability (on an open neighbourhood) implies total complex differentiability there. If we write $z^j = x^j + \sqrt{-1} y^j$, then one can impose weaker conditions on $\pdv{f}{x^j}$ and $\pdv{f}{y^j}$ to ensure (partial) complex differentiability, cfg. Loomen-Menchoff theorem. 
		For a nice discussion on the topic of sufficient conditions ensuring that the Cauchy-Riemann equations imply holomorphy in the complex plane the interested reader can consult \cite{GrayMor}. 
		In the case of $\fA$ being a $\KK$-algebra other than $\CC$, it would be interesting to know if there exists an analogue of Loomen-Menchoff or of other results in the aforementioned article that extend to $\KK$-algebras $\fA$.
		
		\item If $\KK=\CC$, the (G)CRS-equations are PDEs of holomorphic functions in several complex variables with possibly complex coefficients unlike the classical CR-equation which has only real coefficients. 
		Thus, implicitly, $f$ additionally also satisfies the usual real CR-equations for each of its complex variables.
		
		\item More generally, recall that any morphism $\fA \xrightarrow{\varphi} \fB$ of $\KK$-algebras gives a choice of scalars via the $\fA$-module structure it induces on $\fB$. 
		The $\KK$-(G)CRS-equations correspond to the $\KK$-module structure on $\fB$ (when existent), which, however, does not arise from a morphism $\psi: \KK \to \fB$, unless $\fA = \cA$ is unital.
		In this sense, $\varphi: \fA \to \fB$ together with $\varphi$-linearity of $Df$ also gives rise to (G)CRS-equations in Several (usually not free) $\fA$-variables, but for the moment we shall not pursue this perspective here.		
		
		\item All of the above PDEs can be written down for arbitrary structure constants $(\alpha^i_{jk})$ and $(\gamma^r_{st})$, even though the different versions of the equations will no longer remain equivalent. 
		It is an open question if these equations still enjoy as a meaningful interpretation in the case of noncommutative and/or non-associative $\fA$ and $\fB$.
		
		\item On the other hand, if the family of solutions of a linear PDE of the form
		\[
		\pdv{f^i}{z^j} = \sum_{s=1}^m \alpha^i_{jk} \pdv{f^k}{z^1} 
		\]
		is to be stable under composition, then $(\alpha^i_{jk})_{1\leq i,j,k\leq n}$ is the structure tensor of a left-unital associative $\CC$-algebra.
		Indeed, a comparison of coefficients in
		\[
		\pdv{f^i}{z^1} = \sum_{k=1}^m \alpha^i_{1k} \pdv{f^k}{z^1} 
		\]
		implies $\forall 1\leq i,k\leq n: \alpha^i_{1k} = \delta^i_k$, i.e. the basis vector $a_1$ is a left unit.
		Furthermore we obtain for two solutions $f$ and $g$ with $h\coloneqq f\circ g$ also being a solution:
		\[
		\begin{aligned}
		\pdv{h_i}{z_j} &= \pdv{}{z_j} f_i(g_1,\dots,g_n) = \sum_{s=1}^n \pdv{f_i}{z_s}(g_1,\dots,g_n) \pdv{g_s}{z_j} =
		\sum_{s=1}^n \bigg(\sum_{\ell=1}^n \alpha^i_{s\ell} \pdv{f_\ell}{z_1}(g_1,\dots,g_n) \bigg) \bigg(\sum_{k=1}^n \alpha^s_{jk} \pdv{g_k}{z_1} \bigg) =\\
		&= \sum_{k,\ell=1}^n \bigg(\sum_{s=1}^n \alpha^i_{s\ell}\alpha^s_{jk} \bigg) \pdv{f_\ell}{z_1}(g_1,\dots,g_n) \pdv{g_k}{z_1}		
		\end{aligned}	
		\]
		and
		\[
		\begin{aligned}
		\pdv{h_i}{z_j} &= \sum_{s=1}^n \alpha^i_{js} \pdv{h_s}{z_1} = \sum_{s=1}^n \alpha^i_{js} \pdv{}{z_1} f_\lambda(g_1,\dots,g_n) =
		\sum_{s=1}^n \alpha^i_{js} \sum_{k=1}^n \pdv{f_s}{z_k}(g_1,\dots,g_n) \pdv{g_k}{z_1} = \\
		&= \sum_{s=1}^n \alpha^i_{js} \sum_{k=1}^n \bigg( \sum_{\ell=1}^n \alpha^s_{k\ell} \pdv{f_\ell}{z_1}(g_1,\dots,g_n) \bigg) \pdv{g_k}{z_1} =
		\sum_{k,\ell=1}^n \bigg(\sum_{s=1}^n \alpha^i_{js} \alpha^s_{k\ell} \bigg) \pdv{f_\ell}{z_1}(g_1,\dots,g_n) \pdv{g_k}{z_1}.
		\end{aligned}
		\]
		Hence a comparison of the coefficients yields $\forall 1\leq i,j,k,\ell\leq n$:
		\[
		\sum_{s=1}^n \alpha^i_{s\ell} \alpha^s_{jk} = \sum_{s=1}^n \alpha^i_{js}\alpha^s_{k\ell},
		\]
		which is precisely \cref{ass}.
		
		\item If $\fA$ is not unital, it is not clear in how far the Scheffers' form of the \hyperref[GCRS]{(generalized) Cauchy-Riemann equations (\ref*{GCRS})} fails to characterize $\fA$-differentiability of $f$ when compared with \cref{CR}.
		Heuristically at least, if $\fA=\m$ is the maximal ideal of a local $\KK$-algebra $\A=(\A,\m)$, then all the structure of $\A$ is contained in $\m$.
		
		\item A short comparison between the different forms of generalized Cauchy-Riemann equations is in order.
		The PDE system in \cref{GCRU} gives an explicit expression for each partial derivative of each component as a linear combination of the ``free parameters'', played by the $\pdv{}{z^1}$-derivatives, and thus consists of only $(n-1)m$ equations, which is optimal.
		In comparison, \hyperref[GCRS]{Scheffers' PDE system (\cref*{GCRS})} consists of $n^2 m$ equations, which is an order of magnitude more than in \cref{GCRU}, and thus inevitably contains some ``garbage'' as can be illustrated already in the case $\KK=\RR$, $n=2$, $\A=\CC$.
		Indeed, for the structure constants of $\CC/\RR$ one has
		\[
		\alpha^1_{11} = 1, \alpha^2_{11} = 0, \alpha^1_{12} = \alpha^1_{21} = 0, 
		\alpha^2_{12} = \alpha^2_{21} = 1, \alpha^1_{22} = -1, \alpha^2_{22} = 0
		\]
		and only two equations by \hyperref[GCRU]{(\ref*{GCRU})} as follows:
		\[
		\begin{aligned}
		\pdv{f^1}{x^2} = \alpha^1_{21} \pdv{f^1}{x^1} + \alpha^1_{22} \pdv{f^2}{x^1} = -\pdv{f^2}{x^1},\  
		\pdv{f^2}{x^2} = \alpha^2_{21} \pdv{f^1}{x^1} + \alpha^2_{22} \pdv{f^2}{x^1} = \pdv{f^1}{x^1}.
		\end{aligned}
		\]
		These are precisely the classical Cauchy-Riemann equations.
		In comparison, the \hyperref[GCRS]{PDE system (\ref*{GCRS})}  additionally yields several vacuous or repeated equations:
		\begin{itemize}[noitemsep,nolistsep]
			\item $i=1,j=1,k=1$: $0=0$;
			
			\item $i=1,j=1,k=2$: $\pdv{f^1}{x^2} = \pdv{f^1}{x^2}$;
			
			\item $i=1,j=2,k=1$: $\pdv{f^1}{x^2} = - \pdv{f^2}{x^1}$;
			
			\item $i=1,j=2,k=2$: $-\pdv{f^1}{x^1} = - \pdv{f^2}{x^2}$;
			
			\item $i=2,j=1,k=1$: $\pdv{f^2}{x^1} = \pdv{f^2}{x^1}$;
			
			\item $i=2,j=1,k=2$: $0=0$;
			
			\item $i=2,j=2,k=1$: $\pdv{f^2}{x^2} = \pdv{f^1}{x^1}$;
			
			\item $i=2,j=2,k=2$: $-\pdv{f^1}{x^1} = \pdv{f^1}{x^2}$.
		\end{itemize}
		
		\item The \hyperref[GCRU]{PDE system (\ref*{GCRU})} can also be written in the form
		\begin{equation}\label{pdesys}
		\Gamma_j \pdv{f^T}{z^1} = \pdv{f^T}{z^j}
		\end{equation}
		for $2\leq j\leq n$, where $f=(f_1,\dots,f_m)$ and $\Gamma_j \defeq (\gamma^i_{jk})_{1\leq i,k\leq m}$.
		If $\varphi = \id_\A$, then $\Gamma_j = A_j = (\alpha^i_{jk})_{1\leq i,k\leq n} = \lambda(a_j)$ is the regular representation of the basis vector $a_j$, $1\leq j\leq n$.
		Thus, if for example $\A=(\A,\m)$ is local with a choice of basis $a_1\coloneqq 1_\A$, $a_2,\dots,a_n\in\m$, we have $\forall 2\leq j\leq n: \det(A_j)=0$, so there is no a priori reason to expect the matrices $\Gamma_j$ to be invertible.
		Thus it is not always possible to express $f'(Z) = \pdv{f}{z^1}$ via the remaining $\pdv{f^i}{z^j}$, $2\leq j\leq n$, as shown by the next two examples below.
		Nevertheless, it might be possible to paste together different parts of the PDE systems for different $2\leq j\leq n$ to obtain a full rank linear system for the vector $\pdv{f^T}{z^1}$.
		In any case, it is dependent on the choice of basis for $\A$.
		In fact, since $\A^\times$ is open, it contains a basis for $\A$, and for any such choice of basis $a_1,\dots,a_n\in\A^\times$ we have $A_1,\dots,A_n\in\GL_n(\KK)$ and hence
		\[
		\forall 2\leq j\leq n: f'(Z) = \pdv{f}{z^1} = \pdv{f}{z^j} (A_j^{-1})^T.
		\]
		If we start with a basis $a_1\coloneqq 1_\A$ and $a_2,\dots,a_n\in\m$, then we can obtain a new basis $a'_1\coloneqq 1_\A, a'_2,\dots, a'_n$ with the property that $a'_j\in\A^\times$ for some $2\leq j\leq n$ as follows: let $\vareps_2,\dots,\vareps_n\in\{0,1\}$, not all of which are $0$, and set
		\[
		\begin{pmatrix}
		a'_1 \\
		a'_2 \\
		a'_3 \\
		\vdots \\
		a'_n
		\end{pmatrix}
		\coloneqq
		\begin{pmatrix}
		1 & 0 & 0 & \hdots & 0\\
		\vareps_2 & 1 & 0 & \hdots & 0\\
		\vareps_3 & 0 & 1 & \hdots & 0\\
		\vdots & \vdots & & \ddots & \vdots \\
		\vareps_n & 0 & 0 & \hdots & 1
		\end{pmatrix}
		\begin{pmatrix}
		a_1 \\
		a_2 \\
		a_3 \\
		\vdots \\
		a_n
		\end{pmatrix}.
		\]
		This is clearly again a basis for $\A$, since the transition matrix has $\det=1$, and $a'_j\in\A^\times$ for all $2\leq j\leq n$ with $\vareps_j=1$.
	\end{enumerate}
\end{remarks}

\begin{examples}\ 
	\begin{enumerate}[topsep=-\parskip]
		\item If $\A\coloneqq\CC[\eps]$ with $\eps^2=0$ (dual numbers), then $\alpha^1_{11} = 1,\ \alpha^2_{11} = 0,\ \alpha^1_{12} = \alpha^1_{21} = 0,\  \alpha^2_{12} = \alpha^2_{21} = 1,\ \alpha^1_{22} = \alpha^2_{22} = 0$, thus:
		\[
		\pdv{f^1}{z^2} = 0,\ \pdv{f^2}{z^2} = \pdv{f^1}{z^1}.
		\]
		In particular, we have no expression for $\pdv{f^2}{z^1}$.
		On the other hand, if we choose basis $a_1\coloneqq 1_\A$ and $a_2\coloneqq 1+\eps$, then $\alpha^1_{11} = 1,\ \alpha^2_{11} = 0,\ \alpha^1_{12} = \alpha^1_{21} = 0,\  \alpha^2_{12} = \alpha^2_{21} = 1,\ \alpha^1_{22} = -1,\ \alpha^2_{22} = 2$, thus
		\[
		\pdv{f^1}{z^2} = 2\pdv{f^2}{z^1},\ \pdv{f^2}{z^2} = \pdv{f^1}{z^1} + 2\pdv{f^2}{z^1} = \pdv{f^1}{z^1} + \pdv{f^1}{z^2},
		\]
		hence
		\[
		f'(Z) = \pdv{f}{z^1} = \bigg(\pdv{f^2}{z^2} - \pdv{f^1}{z^2}\bigg) + \frac{1+\eps}{2} \pdv{f^1}{z^2}.
		\]
		
		\item If $\A\coloneqq\CC[j]$ with $j^2=1$ (hyperbolic aka. split-complex numbers), then $\alpha^1_{11} = 1,\ \alpha^2_{11} = 0,\ \alpha^1_{12} = \alpha^1_{21} = 0,\ \alpha^2_{12} = \alpha^2_{21} = 1,\ \alpha^1_{22} = 1,\ \alpha^2_{22} = 0$, thus:
		\[
		\pdv{f^1}{z^2} = \pdv{f^2}{z^1},\ \pdv{f^2}{z^2} = \pdv{f^1}{z^1}.
		\]
	\end{enumerate}
\end{examples}



\begin{lemma}
	Let $U\subseteq\KK^n$ be open, let $f\coloneqq(f^1,\dots,f^m)^T:U\to\KK^m$ be a partially $\KK$-differentiable function and $g\coloneqq(g^1,\dots,g^m)^T:U\to\KK^m$ a function, satisfying the PDE system
	\[
	\pdv{f^i}{z^j} = \sum_{s=1}^m \gamma^i_{js} g^s
	\]
	for some constants $(\gamma^i_{jk})\subseteq\KK$, $1\leq j\leq n$, $1\leq i,k\leq m$. Furthermore define $G:(\KK^n)^m\to\rM_m(\KK)$
	\[
	G(Z_1,\dots,Z_m)\coloneqq(g(Z_1),\dots,g(Z_m)) = 
	\begin{pmatrix}
	g^1(Z_1) & \hdots & g^1(Z_m)\\
	\vdots   & \ddots & \vdots \\
	g^m(Z_1) & \hdots & g^m(Z_m)
	\end{pmatrix}.
	\]
	Then $(\gamma^i_{jk})$ are uniquely determined if and only if $\span_\KK g(U) = \KK^m$. In this case
	\begin{equation}
	\gamma^i_{jk} = \det
	\begin{pmatrix}
	g^1(Z_1) & \hdots & g^1(Z_m)\\
	\vdots   &  & \vdots \\
	\pdv{f^i}{z^j}(Z_1) & \hdots & \pdv{f^i}{z^j}(Z_m) - k\text{-th row}\\
	\vdots   &  & \vdots \\			
	g^m(Z_1) & \hdots & g^m(Z_m)
	\end{pmatrix}/
	\det 
	\begin{pmatrix}
	g^1(Z_1) & \hdots & g^1(Z_m)\\
	\vdots   & \ddots & \vdots \\
	g^m(Z_1) & \hdots & g^m(Z_m)
	\end{pmatrix} =\const
	\end{equation}
	whenever the denominator is non-zero.		
\end{lemma}

\begin{proof}
	We have $\span_\KK\im g = \KK^m$ if and only if there exist points $P_1,\dots,P_m\in\KK^n$ such that the vectors $g(P_1),\dots,g(P_n)$ are $\KK$-linearly independent.
	For each $1\leq j\leq n$ and $1\leq i\leq n$ we get a system of $m$ linear equations with $m$ variables
	\[
	(\gamma^i_{j1},\dots,\gamma^i_{jm})
	\begin{pmatrix}
	g^1(P_1) & \hdots & g^1(P_m)\\
	\vdots   & \ddots & \vdots \\
	g^m(P_1) & \hdots & g^m(P_m)
	\end{pmatrix} =
	\left(\pdv{f^i}{z^j}(P_1),\dots,\pdv{f^i}{z^j}(P_m)\right),
	\]
	hence the row $(\gamma^i_{jk})_{1\leq k\leq m}$ is uniquely determined if and only if $\rk G(P_1,\dots,P_m) = m$. Conversely, the $m$ unknowns $(\gamma^i_{jk})_{1\leq k\leq m}$ for each fixed $i,j$ are uniquely determined if and only if we can find $m$ points $P_1,\dots,P_m$ with the above property.
\end{proof}

\begin{corollary}
	Let $\fA\cong\KK^n$ be a finite-dimensional commutative associative $\KK$-algebra, let $U\subseteq\KK^n$ be an open subset and $f:U\to\KK^n\cong\fA$ an $\fA$-differentiable map.
	\begin{enumerate}
		\item If $\span_\KK f'(U)=\KK^n$, then $\fA$ is uniquely determined by $f$. In particular, if $\fB$ is another algebra structure on $\KK^n$ such that $f$ is also $\fB$-differentiable, then $\fA\cong\fB$.
		
		\item If $\fA\ni A$ with $A^2\neq 0$, then $\fA$ can be recovered from $\O_\fA(U)$ for an arbitrary open $U\subseteq\fA$.
		
		\item If $\span_\KK f'(U)=\KK^n$, then commutativity of $\fA$ is equivalent to 
		\begin{equation}
		\det
		\begin{pmatrix}
		f'^1(Z_1) & \hdots & f'^1(Z_n)\\
		\vdots   &  & \vdots \\
		\pdv{f^i}{z^j}(Z_1) & \hdots & \pdv{f^i}{z^j}(Z_n) - k\text{-th row}\\
		\vdots   &  & \vdots \\			
		f'^n(Z_1) & \hdots & f'^n(Z_n)
		\end{pmatrix} = \det
		\begin{pmatrix}
		f'^1(Z_1) & \hdots & f'^1(Z_n)\\
		\vdots   &  & \vdots \\
		\pdv{f^i}{z^k}(Z_1) & \hdots & \pdv{f^i}{z^k}(Z_n) - j\text{-th row}\\
		\vdots   &  & \vdots \\			
		f'^n(Z_1) & \hdots & f'^n(Z_n)
		\end{pmatrix}
		\end{equation}
		for all $1\leq i,j,k\leq n$ and any $Z_1,\dots,Z_n\in\fA$ with $f'(Z_1),\dots,f'(Z_n)$ $\KK$-linearly independent.
		
		\item If $\span_\KK f'(U)=\KK^n$, then associativity of $\fA$ is equivalent to
		\begin{equation}
		\begin{aligned}
		&\sum_{r=1}^n \det
		\begin{pmatrix}
		f'^1(Z_1) & \hdots & f'^1(Z_n)\\
		\vdots   &  & \vdots \\
		\pdv{f^r}{z^j}(Z_1) & \hdots & \pdv{f^r}{z^j}(Z_n) - k\text{-th row}\\
		\vdots   &  & \vdots \\			
		f'^n(Z_1) & \hdots & f'^n(Z_n)
		\end{pmatrix} 
		\begin{pmatrix}
		f'^1(Z_1) & \hdots & f'^1(Z_n)\\
		\vdots   &  & \vdots \\
		\pdv{f^i}{z^r}(Z_1) & \hdots & \pdv{f^i}{z^r}(Z_n) - \ell\text{-th row}\\
		\vdots   &  & \vdots \\			
		f'^n(Z_1) & \hdots & f'^n(Z_n)
		\end{pmatrix} = \\
		&\sum_{r=1}^n \det
		\begin{pmatrix}
		f'^1(Z_1) & \hdots & f'^1(Z_n)\\
		\vdots   &  & \vdots \\
		\pdv{f^r}{z^k}(Z_1) & \hdots & \pdv{f^r}{z^k}(Z_n) - \ell\text{-th row}\\
		\vdots   &  & \vdots \\			
		f'^n(Z_1) & \hdots & f'^n(Z_n)
		\end{pmatrix}
		\begin{pmatrix}
		f'^1(Z_1) & \hdots & f'^1(Z_n)\\
		\vdots   &  & \vdots \\
		\pdv{f^i}{z^j}(Z_1) & \hdots & \pdv{f^i}{z^j}(Z_n) - r\text{-th row}\\
		\vdots   &  & \vdots \\			
		f'^n(Z_1) & \hdots & f'^n(Z_n)
		\end{pmatrix}
		\end{aligned}
		\end{equation}
		for all $1\leq i,j,k,\ell\leq n$ and any $Z_1,\dots,Z_n\in\fA$ with $f'(Z_1),\dots,f'(Z_n)$ $\KK$-linearly independent.
		
		\item If $\span_\KK f'(U)=\KK^n$, then unitality of $\fA$ is equivalent with the existence of $(\eps_r)_{1\leq r\leq n}\subseteq\KK$ such that
		\begin{equation}
		\begin{aligned}
		\delta^i_k &=
		\sum_{r=1}^n \det
		\begin{pmatrix}
		f'^1(Z_1) & \hdots & f'^1(Z_n)\\
		\vdots   &  & \vdots \\
		\eps_r\pdv{f^i}{z^r}(Z_1) & \hdots & \eps_r\pdv{f^i}{z^r}(Z_n) - k\text{-th row}\\
		\vdots   &  & \vdots \\			
		f'^n(Z_1) & \hdots & f'^n(Z_n)
		\end{pmatrix} =
		\sum_{r=1}^n \det
		\begin{pmatrix}
		f'^1(Z_1) & \hdots & f'^1(Z_n)\\
		\vdots   &  & \vdots \\
		\eps_r\pdv{f^i}{z^k}(Z_1) & \hdots & \eps_r\pdv{f^i}{z^k}(Z_n) - r\text{-th row}\\
		\vdots   &  & \vdots \\			
		f'^n(Z_1) & \hdots & f'^n(Z_n)
		\end{pmatrix}
		\end{aligned}
		\end{equation}
		for all $1\leq i,k,\ell\leq n$ and any $Z_1,\dots,Z_n\in\fA$ with $f'(Z_1),\dots,f'(Z_n)$ $\KK$-linearly independent.
	\end{enumerate}
\end{corollary}

\begin{proof}
	To (1): Since $f$ is $\fA$-differentiable, after a choice of basis we have  $\forall 1\leq i,j\leq n$:
	\[
	\pdv{f^i}{z^j} = \sum_{s=1}^n\alpha^i_{js}f'^s,
	\]
	hence $\forall 1\leq i,j,k\leq n$:
	\[
	\alpha^i_{jk} = 
	\det
	\begin{pmatrix}
	f'^1(Z_1) & \hdots & f'^1(Z_n)\\
	\vdots   &  & \vdots \\
	\pdv{f^i}{z^j}(Z_1) & \hdots & \pdv{f^i}{z^j}(Z_n) - k\text{-th row}\\
	\vdots   &  & \vdots \\			
	f'^n(Z_1) & \hdots & f'^n(Z_n)
	\end{pmatrix}/
	\det 
	\begin{pmatrix}
	f'^1(Z_1) & \hdots & f'^1(Z_n)\\
	\vdots   & \ddots & \vdots \\
	f'^n(Z_1) & \hdots & f'^n(Z_n)
	\end{pmatrix}.
	\]
	
	To (2): Consider the globally defined $\fA$-differentiable map $f:\fA\to\fA$, $f(Z)\coloneqq\frac{1}{2}Z^2$, $f\neq 0$. Then $f'(Z) = Z$, i.e. $f'=\id_\fA$, and the claim follows with $f\big|_U\in\O_\fA(U)$ since every open set $U\subseteq\KK^n$ contains a $\KK$-basis.
	Notice that $f'$ itself is not necessarily $\fA$-differentiable.
	
	To (3): This is simply a restatement of \cref{commut} using the expression from the proof of (1).
	
	To (4): This is simply a restatement of \cref{ass} using the expression from the proof of (1).
	
	To (5): This is simply a restatement of \cref{unit} using the expression from the proof of (1).
\end{proof}

\begin{remarks}\ 
	\begin{enumerate}[topsep=-\parskip]
		\item In particular, this suggests that if a function $f:U\to\KK^n$ (e.g. a polynomial or an analytic function) can be written as an $\fA$-differentiable function over some $\fA$ and $f(Z)=O(Z^2)$, then it can be written so in an essentially unique way up to $\Aut_\KK(\fA)$. 
		This in turn is also suggestive as to what the morphisms of \textit{Funktionentheorien} should look like.
		
		\item The $1$-dimensional non-unital nilpotent $\KK$-algebra $\fA_2 \coloneqq \{  
		\begin{psmallmatrix}
		0 & x\\
		0 & 0
		\end{psmallmatrix} : x\in\KK \}$ is a pathological case that oftens needs to be excluded.
		It is also the only connected commutative associative non-unital $\KK$-algebra with the property that $\forall Z\in\fA_2: Z^2 = 0$.
		
		\item If $\rk(f'(Z_1),\dots,f'(Z_n))<n$ or even varies, we get parametric families of candidates $(\alpha^i_{jk})$ to be structure constants of $n$-dimensional $\KK$-algebras (commutativity or associativity does not follow a priori), and it might be possible to write $f$ non-uniquely as a function over different or isomorphic $\KK$-algebras. 
		It is not immediately clear that there is an algebra in the parametric space that is valid for the whole $U$. 
	\end{enumerate}
\end{remarks}


Finally, we consider a certain differential operator that lies at the heart of the theory of $\varphi$-differentiable functions:

\begin{definition}
	Let $\fA \xrightarrow{\varphi} \fB$ be a morphism of $\KK$-algebras, $U\subseteq \fA$ an open subset, and $f:U \to \fB$ a $\varphi$-differentiable function.
	If $\{a_1,\dots,a_n\}$ is a $\KK$-basis of $\fA$ and $Z= z^1 a_1 + \cdots + z^n a_n$, then we define the linear operators
	\begin{equation}
	\d_{ij} \coloneqq a_i \pdv{}{z^j} - a_j \pdv{}{z^i} \defeq \varphi(a_i) \pdv{}{z^j} - \varphi(a_j) \pdv{}{z^i},\ 1\leq i,j\leq n.
	\end{equation}
	Furthermore, if $\cA \xrightarrow{\varphi} \cB$ is a morphism of unital $\KK$-algebras with $a_1 = 1_\cA$, then we also put 
	\begin{equation}
	\d_j \coloneqq \d_{1 j} \defeq \pdv{f}{z^j} - a_j \pdv{f}{z^1} ,\ 2\leq j\leq n.
	\end{equation}
\end{definition}

\begin{remarks}\ 
	\begin{enumerate}[topsep=-\parskip]
		\item Clearly, $\forall 1\leq i,j\leq n: \d_{ii}=0,\ \d_{ij} = -\d_{ji}$.
		
		\item If $(a'_1,\dots,a'_n) = (a_1,\dots,a_n) (u^i_j)_{1\leq i,j\leq n}$ is a change of basis for $\fA$ with $Z \eqqcolon w^1 a'_1 + \dots + w^n a'_n$ and one defines
		\[
		\d'_{k \ell} \coloneqq a'_k \pdv{}{w^\ell} - a'_\ell \pdv{}{w^k}
		\]
		with respect to this new basis, one checks immediately that $\forall 1\leq k,\ell\leq n$:
		\begin{equation}\label{operatornewbasis}
		\d'_{k \ell} = \sum_{i,j=1}^n u^i_k u^j_\ell \d_{ij}.
		\end{equation}
	\end{enumerate}
\end{remarks}

\begin{lemma}\label{delop}
	Let $\fA \xrightarrow{\varphi} \fB$ be a morphism of $\KK$-algebras, $U\subseteq \fA$ an open subset, and $f:U \to \fB$ a totally $\KK$-differentiable function.
	\begin{enumerate}
		\item If $f$ is $\varphi$-differentiable, then $\forall 1\leq i<j\leq n: \d_{ij} f = 0$, where the last condition is independent of the choice of basis for $\fA$.
		
		\item If $\cA \xrightarrow{\varphi} \cB$ is a morphism of unital $\KK$-algebras, then the converse of (1) is true.
		In fact, if $a_1 = 1$, then it suffices to have only $\forall 2\leq j\leq n: \d_j f \defeq \d_{1 j} f=0$ satisfied.
	\end{enumerate}
\end{lemma}

\begin{proof}
	The independence of the condition $\forall 1\leq i<j\leq n: \d_{ij} f = 0$	of the choice of basis for $\fA$ follows directly from \hyperref[operatornewbasis]{bilinearity (\cref*{operatornewbasis})} of the operators $\d_{ij}$.
	
	``$\Rightarrow$'': clear, since $\pdv{f}{z_i}(Z) = f'(Z) a_i \defeq f'(Z) \varphi(a_i)$, $1\leq i\leq n$, and $a_i a_j = a_j a_i$.
	
	``$\Leftarrow$'': We show that $f$ satisfies \hyperref[GCRU]{the Generalized Cauchy-Riemann Equations with Unit (\cref*{GCRU})}.
	Writing $f=f^1 b_1 + \dots + f^m b_m$, we have:
	\[
	\pdv{f}{z^1} \varphi(a_j) = \sum_{s=1}^m \pdv{f^s}{z^1} \varphi(a_j) b_s = \sum_{s=1}^m \pdv{f^s}{z^1} \sum_{r=1}^m \gamma^r_{js} b_r = 
	\sum_{r=1}^m b_r \bigg( \sum_{s=1}^m \gamma^r_{js} \pdv{f^s}{z^1} \bigg).
	\] 
	On the other hand,
	\[
	\pdv{f}{z^j} = \pdv{f}{z^1} a_j,
	\]
	since $\d_j f=0$ for all $2\leq j\leq n$. 
	Therefore $\forall 1\leq i\leq m\ \forall 2\leq j\leq n$:
	\[
	\pdv{f^i}{z^j} = \sum_{s=1}^m \gamma^i_{js} \pdv{f^s}{z^1}
	\]
	as desired.
\end{proof}
	
	\section{Function Theory over Non-Unital $\KK$-Algebras}\label{sec:nonunitalfunc}
	
\emph{Most of the proofs in this section are routine verifications that the given classical theorems from the Complex Analysis of One Variable transfer almost verbatim to our setting and can therefore be safely skipped by the reader.}

\begin{definition}
	Let $\fA\xrightarrow{\varphi}\fB$ a morphism of not necessarily unital $\KK$-algebras, $U\subseteq\fA$ open, and $f\in\Cc^0(U,\fB)$. 
	Then $f$ is called $\varphi$-integrable if $f$ has a $\varphi$-primitive, i.e. if there exists $F\in\O_\varphi(U)$ such that $F'=f$.
\end{definition}

Recall that continuous functions can in general be integrated over rectifiable paths. 
However, in many of the following statements we will need for technical reasons slightly stronger regularity assumptions for the paths, namely piece-wise smoothness. 

\begin{lemma-definition}
	Let $\lambda_\fA:\fA\hookrightarrow\rM_n(\KK)$ be the stable lower-triangular representation of $\fA$, let $\norm{\cdot}_\fA$ be a submultiplicative norm on $\fA$, and let $\gamma\in\Cc^1_\pw(I,\fA)$ be a piece-wise smooth path, where $I\coloneqq[0,1]$.
	\begin{enumerate}
		\item $L_\fA(\gamma)\coloneqq \int_0^1 \norm{\gamma'(t)}_\fA\d t$ is the length of $\gamma$ with respect to $\norm{\cdot}_\fA$.
		
		\item If $\norm{\cdot}_\fA\coloneqq\norm{\cdot}_\rF$, then $L_\fA(\gamma)\defeq L(\lambda_\fA\circ\gamma) = \int_0^1 \norm{\gamma'(t)}_\rF\d t$ is the Euclidean length of $\gamma$ in $\rM_n(\KK)$.
		
		\item If $\fA\coloneqq\A$ is a unital $\KK$-algebra and $\gamma$ is a \emph{scalar} curve, then $L_\A(\gamma) = \norm{1_\A} L_\KK(\gamma)$. 
		In particular, if $\norm{\cdot}_\A = \norm{\cdot}_\rF$, then $ L_\A(\gamma) = \sqrt{n} L_\KK(\gamma)$, where $n=\dim_\KK\A$.
		
		\item If $\fA \xrightarrow{\varphi} \fB$ is a morphism of not necessarily unital $\KK$-algebras, then $L_\fB(\gamma)\coloneqq L_\fB(\varphi\circ\gamma) = \int_0^1 \norm{\varphi(\gamma'(t))}_\B\d t$.
		
		\item If $\A\xrightarrow{\varphi}\B$ is a morphism of unital $\KK$-algebras and $\gamma$ is a \emph{scalar} curve, then so is $\varphi\circ\gamma$ and $L_\B(\varphi\circ\gamma) = \norm{1_\B}_\B L_\KK(\gamma)$. 
		In particular, if $\norm{\cdot}_\B=\norm{\cdot}_\rF$, then $L_\B(\varphi\circ\gamma) = \sqrt{m} L_\KK(\gamma)$, where $m=\dim_\KK\B$.
	\end{enumerate}
\end{lemma-definition}

\begin{proof}
	To (3): $\gamma$ being scalar in $\A$ implies $\norm{\gamma'(t)}_\mathrm{F} = \abs{\gamma'(t)}\norm{1_\A}_\mathrm{F} = \abs{\gamma'(t)} \sqrt{n}$.
	
	To (4): Since $\varphi$ is $\RR$-linear in all cases, $(\varphi(\gamma(t)))' = \varphi(\gamma'(t))$.
	
	To (5): Clearly, $\varphi\circ\gamma$ is a scalar curve since $\varphi$ is $\KK$-linear. 
	We have $\norm{\varphi(\gamma'(t))}_\B = \abs{\gamma'(t)}\norm{\varphi(1_\A)}_\B = \abs{\gamma'(t)}\norm{1_\B}_\B$.
\end{proof}

\begin{remark}
When $\gamma$ is a scalar curve, it is more convenient to choose a unital norm.
\end{remark}

\begin{lemma}[``$\varphi$-integrable functions have $\oint=0$'']
	Let $\fA\xrightarrow{\varphi}\fB$ be a morphism of not necessarily unital $\KK$-algebras, $U\subseteq\fA$ open, $f\in\Cc^0(U,\fB)$ $\varphi$-integrable, and $\gamma\in\Cc^1_\pw(\SS^1,U)$. Then
	\[
	\oint_\gamma f(Z)\d Z = 0.
	\]
\end{lemma}

\begin{proof}
	Let $F$ be a $\varphi$-primitive of $f$, then 
	\[
	\begin{aligned}
	\oint_{\gamma}f(Z)\d Z &= \oint_\gamma F'(Z)\d Z = \int_0^1 F'(\gamma(t))\gamma'(t)\d t \defeq \int_0^1 F'(\gamma(t))\varphi(\gamma'(t))\d t =
	\int_0^1 \d(F\circ\gamma) = F(\gamma(1))-F(\gamma(0)) = 0.
	\end{aligned}
	\]
\end{proof}

\begin{lemma}\label{intineq}
	Let $\fA\xrightarrow{\varphi}\fB$ be a morphism of not necessarily unital $\KK$-algebras, $\gamma\in\Cc^1_\pw(\SS^1,\fA)$, and $f\in\Cc^0(\gamma,\fB)$. Then:
	\[
	\norm{\oint_\gamma f(Z)\d Z}_\fB \leq \norm{f}_{\fB,\gamma} L_\fB(\gamma).
	\]
\end{lemma}

\begin{proof}
	Using the submultiplicativity of the norm, we obtain
	\[
	\begin{aligned}
	\norm{\oint_\gamma f(Z)\d Z}_\fB &= \norm{\int_0^1 f(\gamma(t)) \varphi(\gamma'(t)) \d t}_\fB \leq \int_0^1\norm{f(\gamma(t))}_\fB \norm{\varphi(\gamma'(t))}_\fB \d t
	\leq \sup_{Z\in\gamma} \norm{f(Z)}_\fB \int_0^1 \norm{\varphi(\gamma'(t))}_\fB \d t
	\end{aligned}	
	\]
	as required.
\end{proof}

\begin{lemma}[Pre-Morera: $\varphi$-integrability of continuous functions]\label{premorera}
	Let $\fA\xrightarrow{\varphi}\fB$ be a morphism of not necessarily unital $\KK$-algebras.
	\begin{enumerate}
		\item Let $U\subseteq\fA$ be open and path-connected and let $f\in\Cc^0(U,\fB)$ be such that for any rectifiable loop $\gamma\in\Cc^0(\SS^1,U)$
		\[
		\oint_\gamma f(Z)\d Z = 0.
		\]
		Then $f$ is (globally) $\varphi$-integrable.
		
		\item Let $\mystar\subseteq\fA$ be open and star-convex with center $\ast\in\mystar$ and let $f\in \Cc^0(\mystar,\fB)$ be such that for any triangle $\triangle\subseteq\mystar$ with vertex at $\ast$
		\[
		\oint_{\pd\triangle}f(Z)\d Z = 0.
		\]
		Then $f$ is $\varphi$-integrable in $\mystar$ with a $\varphi$-primitive given by
		\[
		F(Z)\coloneqq\int_\ast^Z f(W)\d W.
		\]
		It follows in particular that $\forall \gamma \in \Cc^1_\pw(\SS^1,\mystar)$:
		\[
		\oint_\gamma f(Z)\d Z = 0.
		\]
		
		\item Let $U\in\fA$ be open and path-connected and let $f\in\Cc^0(U,\fB)$ be such that $\forall\ \square\subseteq U$:
		\[
		\oint_{\pd\square}f(Z)\d Z = 0.
		\]
		Then $f$ is $\varphi$-integrable with primitive
		\[
		F(Z)\coloneqq\int_{Z_0\stairup Z} f(W)\d W,
		\]
		where the integral is taken along a stairway from $Z_0$ to $Z$ with sufficiently small stairs. 
		It follows in particular that $\forall\gamma\in\Cc^1_\pw(\SS^1,U)$:
		\[
		\oint_\gamma f(Z)\d Z = 0.
		\]
	\end{enumerate}
\end{lemma}

\begin{proof}
	To (1): By assumption, the function
	\[
	F(Z)\coloneqq\int_{Z_0}^Z f(W)\d W
	\]
	is well-defined for any fixed $Z_0\in U$ as it is independent of the choice of an integration path. 
	We are going to show that $F$ is a $\varphi$-primitive of $f$. 
	For a fixed $Z$ we have:
	\[
	\begin{aligned}
	F(Z+H)-F(Z) &= \int_Z^{Z+H}f(W)\d W = f(Z)\int_Z^{Z+H}\d W + \int_Z^{Z+H}g(W)\d W = f(Z)H + \int_Z^{Z+H}g(W)\d W,
	\end{aligned}
	\]
	where $g(W)\coloneqq f(W)-f(Z)\xrightarrow{W\to Z}0$ by continuity of $f$. 
	Since $U$ is open, for sufficiently small $H$ we can take an open ball around $Z$, which is convex\footnote{finite-dimensional $\KK$-vector spaces are locally convex TVS}, containing $Z+H$ and the $\RR$-line segment $[Z,Z+H]$, which we also take to be the integration path. 
	Now, using \cref{intineq}, we can estimate:
	\[
	\norm{\int_Z^{Z+H}g(W)\d W}_\fB \leq \norm{g}_{\fB,[Z,Z+H]} L_\fB([Z,Z+H]),
	\]
	where $\norm{g}_{\fB,[Z,Z+H]}\xrightarrow{H\to 0}0$ because $g(W)\xrightarrow{W\to Z}0$. 
	Since $\varphi$ is in any case $\RR$-linear, 
	\[
	L_\fB([Z,Z+H])\defeq L_\fB(\varphi([Z,Z+H])) = L_\fB([\varphi(Z),\varphi(Z)+\varphi(H)]) = \norm{\varphi(H)}_\fB.
	\]
	Therefore $F(Z+H)-F(Z) = f(Z)H + o(\norm{\varphi(H)}_\fB)$ as required.
	
	To (2): For $H$ small enough we have $\triangle\coloneqq [\ast,Z+H,Z]\subseteq\mystar$, and $\pd\triangle = [\ast,Z+H] + [Z+H,Z] + [Z,\ast]$. 
	Since $\oint_{\pd\triangle}f(W)\d W=0$, we get:
	\[
	F(Z+H)-F(Z) = \int_Z^{Z+H}f(W)\d W,
	\]
	from where one proceeds analogously to (1).
	
	To (3): It suffices to prove the claim for a single stair $\lrcorner$. Let $H_1$ and $H_2$ correspond to the sides of $\lrcorner\coloneqq Z\lrcorner (Z+H)$, i.e. $H_1+H_2 = H$. 
	Then $L_\fB(\varphi(\lrcorner))\leq \norm{\varphi(H_1)}_\fB + \norm{\varphi(H_2)}_\fB \geq \norm{\varphi(H)}_\fB$, and we proceed as in (1): 
	\[
	\norm{\int_{\lrcorner}g(W)\d W}_\fB \leq \norm{g}_{\fB,\lrcorner} L_\fB(\lrcorner)\leq \norm{g}_{\fB,\lrcorner}\left(\norm{\varphi(H_1)}_\fB + \norm{\varphi(H_2)}_\fB\right),
	\]
	the latter being again $o(\norm{\varphi(H)}_\fB)$ as desired.
\end{proof}


\begin{proposition}[$\d$-closedness of $\varphi$-holomorphic forms]\label{closed}
	Let $\fA\xrightarrow{\varphi}\fB$ be a morphism of not necessarily unital $\KK$-algebras, $U\subseteq\fA$ open, $f:U\to\fB$ a totally $\KK$-differentiable function, and define $\omega\coloneqq f(Z)\d Z$. 
	We have:
	\begin{enumerate}
		\item If $f$ is also $\varphi$-differentiable, then $\omega$ is $\pd$-closed.
		
		\item If $\varphi: \cA \to \cB$ is a morphism of unital $\KK$-algebras, then the converse of (1) also holds. 
		
		\item If $\KK=\CC$ and $f$ is also $\varphi$-differentiable, then $\omega$ is also $\d$-closed.
	\end{enumerate}
\end{proposition}

\begin{proof}
	To (1) \& (2):
	Let $\{a_1,\dots,a_n\}$ be a basis of $\fA$ and write $Z=z^1 a_1 + \dots + z^n a_n$.
	Since $\d z^i\wedge\d z^j = -\d z^j\wedge\d z^i$, we have:
	\[
	\begin{aligned}
	\pd \omega &= \sum_{i=1}^n\pdv{f}{z^i}\d z^i \wedge \d Z = \sum_{i=1}^n\pdv{f}{z^i}\d z^i \wedge \bigg(\sum_{j=1}^n a_j \d z^j\bigg) = 
	\sum_{i,j=1}^n \pdv{f}{z^i} \varphi(a_j) \d z^i \wedge \d z^j =\\
	&= \!\!\!\!\!\! \sum_{1\leq i<j \leq n} \bigg(\pdv{f}{z^i} \varphi(a_j) - \pdv{f}{z^j} \varphi(a_i) \bigg) \d z^i \wedge \d z^j = 
	\sum_{1\leq i<j \leq n} (\d_{ij} f) \d z^i \wedge \d z^j,
	\end{aligned}
	\]
	from which both claims follow by \cref{delop}.
	
	To (3): If $\KK=\CC$, claim follows from (1), since $\d=\partial+\pdbar$ and $\pdbar\omega = \sum_{i=1}^n\pdv{f}{\bar{z}^i}\d\bar{z}^i\wedge\d Z = 0$ as $f$ is holomorphic (in each variable).
\end{proof}


\begin{remarks}\ 
	\begin{enumerate}[topsep=-\parskip]
		\item Much of what follows hinges upon the fact that $\omega=f(Z)\d Z$ is $\d$-closed. 
		The proof of (1) shows that to this effect it suffices for $f$ to only be directionally $\varphi$-differentiable in directions of some basis of $\fA$.
		Another byproduct of the proof is the interpretation of $\d_{ij} f$ as the coordinates of $\pd(f(Z)\d Z)$.
		
		\item Recall that if $\omega$ is a $\d$-closed 1-form, then it is locally exact by Poincare's lemma. 
		This is one way to show that $\varphi$-holomorphic functions admit local $\varphi$-primitives.
		We will give another proof of this based on a generalized Cauchy-Goursat theorem.
		
		\item The existence of local primitives in turn enables the integration of $\omega$ along \emph{continuous} paths \cite[see][Ch.1, §7]{BerGay}. 
		In the same vein, homotopy arguments about $\int\omega$ work in the continuous category, so we won't always have to take extra care of path regularity.
		Thus the next proposition is stated for continuous paths, but for the purpose of obtaining another proof of the existence of local primitives without appealing to Poincare's lemma one should assume the paths to be rectifiable or $\Cc^1_\pw$-regular.		
	\end{enumerate}	
\end{remarks}

\begin{proposition}[Cauchy-Goursat Integral Theorem over $\A$]\label{GCIT}
	Let $\fA\xrightarrow{\varphi}\fB$ a morphism of not necessarily unital $\KK$-algebras, $U\subseteq\fA$ open and path-connected, and $f\in\Cc^0(U,\fB)$.
	\begin{enumerate}
		\item Homological Cauchy-Goursat: if $f\in\O_\varphi(U)$, where additionally $f'$ is assumed continuous in the case $\KK=\RR$, and $\Gamma\in Z_1(U,\ZZ)$ is a 1-cycle with $[\Gamma]=0$ in $H_1(U,\ZZ)$, then 
		\begin{equation}\label{CIT}
		\int_{\Gamma}f(Z)\d Z = 0.
		\end{equation}
		\item Homotopical Cauchy-Goursat: in particular, if $f\in\O_\varphi(U)$, where additionally $f'$ is assumed continuous in the case $\KK=\RR$, and $\pi_1(U)=0$, then for any $\gamma\in\Cc^0(\SS^1,U)$: 
		\begin{equation}
		\oint_{\gamma}f(Z)\d Z = 0.
		\end{equation}
		
		\item Cauchy-Goursat minus a point: if $f\in\O_\varphi(U\setminus\{\pt\})$, then $\forall\triangle\subseteq U$:
		\begin{equation}
		\oint_{\pd\triangle}f(Z)\d Z = 0.
		\end{equation}
		
		\item Minus a finite number of points: if $f\in\O_\varphi(U\setminus\{p_1,\dots,p_s\})$, then $\forall\triangle\subseteq U$:
		\begin{equation}
		\oint_{\pd\triangle}f(Z)\d Z = 0.
		\end{equation}
	\end{enumerate}
\end{proposition}

\begin{proof}
	To (1):
	If $\KK=\CC$, then $\omega\coloneqq f(Z)\d Z$ is $\Cc^1$-regular (in fact, $\Cc^\omega$-regular), so we can apply Stokes' theorem: since $\Gamma\in B_1(U,\ZZ)$, there exists $\square\in Z_2(U,\ZZ)$ such that $\Gamma = \pd\square$. 
	But $\omega$ is closed by \cref{closed}, therefore
	\[
	\int_\Gamma\omega = \int_{\pd\square}\omega = \int_{\square}\d\omega = 0.
	\]
	
	To (2): Since $\pi_1(U)=0$, any $\gamma\in\Cc^0(\SS^1,U)$ is homotopical to a point, hence null-homologous in particular. 
	Alternatively, use homotopy-invariance of closed 1-forms.
	
	To (3): We adapt Goursat-Pringsheim's approach following Rudin's exposition, which is applicable to $\KK$ without the need of additional $\Cc^1$-regularity of $\omega$. 
	First suppose that $p\notin\triangle$ and let $\triangle=[a,b,c]$. 
	Let $a',b',c'$ be the mid points on the opposite sides and let $\triangle'_j$, $1\leq j\leq 4$, be the respective newly formed triangles by connecting $a',b',c'$ with each other. 
	Then
	\[
	J\coloneqq \oint_{\pd\triangle}\omega = \sum_{j=1}^4\oint_{\pd\triangle'_j}\omega
	\]
	and therefore
	\[
	\norm{\oint_{\pd\triangle'_j}\omega}_\fB \geq \frac{\norm{J}_\fB}{4}
	\]
	for certain fixed $1\leq j\leq 4$. 
	Repeating the same procedure for $\triangle'_j$ in place of $\triangle$, we get a sequence of triangles $\triangle_1,\triangle_2,\dots,\triangle_k,\dots$ such that $\triangle\supset\triangle_1\supset\triangle_2\supset\dots$ with $L_\fA(\pd\triangle_k) = 2^{-k}L_\fA(\pd\triangle)$ and 
	\[
	\norm{J}_\fB \leq 4^k\norm{\oint_{\pd\triangle_k}\omega}_\fB,\ k\in\NN.
	\]
	Since $\varphi$ is $\RR$-linear in any case, we also have $L_\fB(\pd\triangle_k) = 2^{-k}L_\fB(\pd\triangle)$. 
	By Cantor's ``nested intervals theorem'', it follows that there is a (unique) point $ Z_0\in\bigcap_{k\in\NN}\triangle_k\subseteq\triangle$, and $f$ is $\varphi$-differentiable at $Z_0$. 
	Therefore $\forall\vareps>0\ \exists\delta>0\ \forall \norm{Z-Z_0}_\fA<\delta:\ \norm{f(Z)-f(Z_0)-f'(Z_0)(Z-Z_0)}_\fB\leq \vareps\norm{\varphi(Z-Z_0)}_\fB$. 
	There exists $k\in\NN$ such that $\forall Z\in\triangle_k: \norm{Z-Z_0}_\fA<\delta$. 
	Then also $\norm{\varphi(Z-Z_0)}_\fB \leq \diam_\fB\varphi(\triangle_k)<L_\fB(\partial\triangle_k) = 2^{-k}L_\fB(\partial\triangle)$. 
	Now, since $\oint_{\pd\triangle_k}\d Z = 0$ and $\oint_{\pd\triangle_k}(Z-Z_0)\d Z = 0$ (they have the obvious $\varphi$-primitives), we have:
	\[
	\oint_{\pd\triangle_k}\omega = \oint_{\pd\triangle_k}\big(f(Z)-f(Z_0)-f'(Z_0)(Z-Z_0)\big)\d Z,
	\]
	whence
	\[
	\begin{aligned}
	\norm{J}_\fB &\leq 4^k \norm{\oint_{\pd\triangle_k}\big(f(Z)-f(Z_0)-f'(Z_0)(Z-Z_0)\big)\d Z}_\fB
	\leq 4^k\sup_{Z\in\pd\triangle_k} \norm{f(Z)-f(Z_0)-f'(Z_0)(Z-Z_0)}_\fB L_\fB(\pd\triangle_k) \\
	&\leq 4^k \sup_{Z\in\pd\triangle_k} \vareps \norm{\varphi(Z-Z_0)}_\fB L_\fB(\pd\triangle_k) \leq 4^k \vareps 2^{-k} L_\fB(\pd\triangle) L_\fB(\pd\triangle_k) = 
	\vareps L_\fB(\pd\triangle)^2
	\end{aligned}
	\]
	for all $\vareps>0$.
	Thus $J=0$. 
	
	Now suppose without loss of generality that $p=a$ is a vertex of $\triangle$. 
	If $a,b,c$ are colinear, then the claim is trivial. 
	If not, choose $x\in [a,b]$ and $y\in [a,c]$ both close to $p=a$. 
	Then
	\[
	\oint_\triangle\omega = \left(\oint_{\pd [a,x,y]} + \oint_{\pd [x,b,y]} + \oint_{\pd [b,c,y]}\right)\omega,
	\]
	the last two being 0, since they do not contain $p=a$. 
	Therefore:
	\[
	\oint_\triangle\omega = \left(\int_a^x + \int_x^y + \int_y^a\right)\omega.
	\]
	Since $f$ is bounded on $\triangle$, letting $x,y\to p=a$ we get the desired result. Finally, if $p\in\triangle$ arbitrary, then apply the above to the triangles $[a,b,p]$, $[b,c,p]$ and $[c,a,p]$.
	
	To (4): follows from (3) (or its proof).
\end{proof}

\begin{remarks}Let $\fA\xrightarrow{\varphi}\fB$ a morphism of not necessarily unital $\KK$-algebras.
	\begin{enumerate}
		\item If $U\subseteq\fA$ is open and $f\in\O_\varphi(U\setminus\{p_1,\dots,p_s\})$, then $\forall\ \square\subseteq U: \oint_{\pd\square}f(Z)\d Z = 0$, because $\square = \triangle + \triangledown$.
		
		\item If $U\subseteq\fA$ is open and 1-connected and $f\in\O_\varphi(U)$, then $f$ is $\varphi$-integrable, i.e. $f$ has a $\varphi$-primitive. 
		This is more general than what Poincare's lemma delivers because simply connected domains are not necessarily contractible, e.g. $\SS^n$ for $n\geq 2$, or better yet, $\BB^n\setminus\{\pt\}$ for $n\geq 3$.
		
		\item In particular, if $U\subseteq\fA$ is an arbitrary open set and $f\in\O_{\varphi}(U)$, then $f$ is locally $\varphi$-integrable, i.e. for every $p\in U$ there exists an (open) neighbourhood $V_p\ni p$ inside $U$ such that $f|_{V_p}$ has a $\varphi$-primitive.
		This follows also from Poincare's lemma.
		
		\item For instance, if $\dim_\RR\fA\geq 3$, $U\subseteq\fA$ path-connected, and $\emptyset\neq S\subset U$ such that $\pi_1(U\setminus S) = 0$, then all $f\in\O_\varphi(U\setminus S)$ are (globally) $\varphi$-integrable, e.g. take $U$ open with $\pi_1(U)=0$ and $S\coloneqq\BB$ a topological ball inside $U$. 
		This is in stark contrast to the complex plane, where removing even a single point breaks global integrability, cfg. $1/z$ on $\CC$. 
		
		\item If $\mystar\subseteq\fA$ is star-convex and $f\in\Cc^0(\mystar,\fB)$ such that $f\in\O_\varphi(\mystar\setminus\{p_1,\dots,p_s\})$, then $f$ has a $\varphi$-primitive. \qed
	\end{enumerate}
\end{remarks}
	
	\section{Function Theory over Unital $\CC$-Algebras: The Cornerstones}\label{sec:unitalfunc}
		
\subsection{Admissible Points and 1-Cycles}\label{subs:adm}

We have reached the point where both the real and complex and the non-unital and unital theory diverge from each other. 
Let us first fix some notations. 

\begin{definition}
	Let $\A\in\fdCAlg_\KK$. Then $\A_\sing\coloneqq\A\setminus\A^\times$ is called the singular cone of $\A$.
\end{definition}

Let $\A$ be a \textit{unital} $\CC$-algebra and let $\Spm\A=\{\fM_1,\cdots,\fM_M\}$ be its (maximal) spectrum. 
In particular, $\#\Spm\A = M$ and $\A_\sing = \bigcup_{k=1}^M\fM_k$. 
Accordingly, the decomposition of $\A$ into local Artin algebras gives a \textit{vector space} decomposition 
\[
\A = \A_1\times\cdots\times\A_M = (\CC_{(1)}\oplus\m_1)\times\cdots\times(\CC_{(M)}\oplus\m_M), 
\]
where $\CC_{(k)}$ are indexed copies of $\CC$ and in this notation $\fM_k = \A_1 \times \cdots \times \A_{k-1} \times \m_k \times \A_{k+1} \times \cdots \times \A_M$, $1\leq k\leq M$.
As a \textit{topological space} $\A^\times$ sits coordinate-wise inside said vector space in the form of
\[
\A^\times = (\CC^{\times}_{(1)}\times\m_1)\times\cdots\times(\CC^{\times}_{(M)}\times\m_M).
\]
Notice that here we make no use of the structure of $\A^\times_k$, $1\leq k\leq M$, as a \emph{Lie group}, but only as a topological space. 
Let us denote by $I_k\coloneqq (0,\dots,0,1_{\A_k},0,\dots,0)\in\A$, $1\leq k\leq M$, the canonical idempotents for said decomposition.
They satisfy the usual relations $I_k I_\ell = \delta_{k\ell} I_k$, $1\leq k,\ell\leq M$, and $\sum_{k=1}^M I_k = 1_\A$. 
Every $Z\in\A$ can be uniquely written as $Z=\bigoplus_{k=1}^M Z_k = \sum_{k=1}^M I_k Z_k$ for some $Z_k\in\A_k$. 
Recall the projections $\pr_k:\A\twoheadrightarrow\A_k$ and $\sigma_k:\A \twoheadrightarrow \CC_{(k)}$, $1\leq k\leq M$. 
In particular, $Z_k = \pr_k(Z)$, and for a given $\gamma\in\Cc^0(\SS^1,\A)$ we shall also write $\gamma_k\coloneqq\pr_k(\gamma) \coloneqq \pr_k\circ\gamma:\SS^1\to\A_k$, and $\gamma^\sp_k \coloneqq \sigma_k(\gamma) \coloneqq \sigma_k \circ \gamma:\SS^1\to\CC_{(k)}=\CC$.
Moreover, if $\iota^\sp_k:\CC_{(k)}\hookrightarrow\A$ and $\iota_k:\A_k\hookrightarrow\A$ are the respective canonical inclusions of \textit{vector spaces}, we shall also consider $\iota^\sp_k\circ\gamma^\sp_k = I_k\gamma^\sp_k$ and $\iota_k\circ\pr_k(\gamma) = I_k\gamma_k$, where by slight abuse of notation $(I_k\gamma_k^{(\sp)})(t) \coloneqq I_k \gamma_k^{(\sp)}(t)$. 
Furthermore, for the homotopy classes of closed curves in $\A^\times$ we have
\[
[\SS^1;\A^\times] = \prod_{k=1}^M [\SS^1;\CC^{\times}_{(k)}\times\m_k]= \prod_{k=1}^M [\SS^1;\CC^{\times}_{(k)}],
\]
where namely $\gamma \simeq \sum_{k=1}^M \gamma_k I_k \simeq \sum_{k=1}^M \gamma^\sp_k I_k$ in $\pi_1(\A^\times)$.
Being projections, $\pr_k$ and $\sigma_k$ are continuous and open. 
Therefore, if $U\subseteq\A$ is open and 0-connected, then so are $\pr_k(U)\subseteq \A_k$ and $\sigma_k(U)\subseteq \CC_{(k)}=\CC$, $1\leq k\leq M$.
Let $\Gamma \coloneqq \sum_{j=1}^m n_j \gamma_j \in Z_1(U,\ZZ)$ be a 1-cycle.
Then $\pr_k$ and $\sigma_k$ induce chain maps 
\[
\begin{aligned}
&(\pr_k)_\#: Z_1(U,\ZZ) \to Z_1(\pr_k(U),\ZZ),\ \Gamma_k \coloneqq (\pr_k)_\#\Gamma \defeq \sum_{j=1}^m n_j \pr_k(\gamma_j)\\
&(\sigma_k)_\#: Z_1(U,\ZZ) \to Z_1(\sigma_k(U),\ZZ),\ \Gamma^\sp_k \coloneqq (\sigma_k)_\# \Gamma \defeq \sum_{j=1}^m n_j \sigma_k(\gamma_j),
\end{aligned}
\]
which in turn induce as usual maps in homology $(\pr_k)_*: H_1(U,\ZZ)\to H_1(\pr_k(U),\ZZ)$ and $(\sigma_k)_*: H_1(U,\ZZ)\to H_1(\sigma_k(U),\ZZ)$ respectively, $1\leq k\leq M$.
Like in the case of loops, we shall use the same letter to denote both the cycle and its support in order to avoid unnecessary additional notational clutter.
Furthermore, if $\cA=(\cA,\fm)$ is local, we shall write $\sigma \coloneqq \sigma_1$, $\gamma^\sp \coloneqq \gamma_1^\sp$, $\Gamma^\sp \coloneqq \Gamma_1^\sp \coloneqq \sigma_* \Gamma$, and $z \coloneqq z^1 \equiv \sigma(Z)$.

\begin{definition}[Admissible points, sets, and 1-cycles]
	Let $\cA \cong \bigoplus_{k=1}^M (\cA_k,\fm_k)$ and $\cB \cong \bigoplus_{\ell=1}^N (\cB_\ell,\fn_\ell)$ be decompositions of $\cA$ and $\cB$ into local Artin $\CC$-algebras.	 
	\begin{enumerate}
		\item Let $I\subseteq \{1,\dots,M\}$ be an index subset and let $\gamma\in\Cc^0(\SS^1,\A)$.
		Then the set
		\[
		\fF_I(\gamma) \coloneqq \bigcup_{k\in I} \sigma_k^{-1}(\gamma^\sp_k) = 
		\Big\{(z_k\oplus X_k)_k \in \prod_{k=1}^M \CC_{(k)}\oplus\m_k\ |\ \exists k\in I: z_k\in\gamma^\sp_k\Big\}
		\]
		is called the $I$-restricted set of forbidden points (forbidden zone) for $\gamma$.
		In the special case $I=\{1,\dots,M\}$ we shall also write $\fF(\gamma) \coloneqq \fF_\cA(\gamma)$ for the full forbidden zone of $\gamma$.
		Its complement
		\[
		\begin{aligned}
		\Adm_I(\gamma) \coloneqq \A \setminus \fF_I(\gamma) = \big\{Z\in\A\ |\ \forall k\in I: \sigma_k(Z) \notin \gamma^\sp_k\big\} = 
		\prod_{k=1}^M
		\begin{cases}
		(\CC_{(k)} \setminus \gamma^\sp_k) \times \m_k, \text{ if } k\in I\\
		\cA_k, \text{ otherwise }
		\end{cases}
		\end{aligned}
		\]
		is called the $I$-restricted set of admissible points for $\gamma$.
		In the special case $I=\{1,\dots,M\}$ we shall also write $\Adm(\gamma) \coloneqq \Adm_\cA(\gamma)$ for the full admissible set of $\gamma$.
		
		\item Let $I\subseteq \{1,\dots,M\}$ be an index subset, let $U\subseteq\A$ be open and path-connected, and let $\Gamma\coloneqq\sum_{j=1}^m n_j \gamma_j\in Z_1(U,\ZZ)$ be a 1-cycle. 
		Then 
		\[
		\fF_I(\Gamma) \coloneqq \bigcup_{j=1}^m \fF_I(\gamma_j) \defeq \big\{Z\in\cA |\ \exists k\in I: \sigma_k(Z) \in \Gamma_k^\sp \big\}
		\]
		is called the $I$-restricted forbidden zone of the 1-cycle $\Gamma$.
		In the special case $I=\{1,\dots,M\}$ we shall also write $\fF(\Gamma) \coloneqq \fF_\cA(\Gamma)$ for the full forbidden zone of $\Gamma$.
		Its complement
		\[
		\begin{aligned}
		\Adm_I(\Gamma) &\coloneqq \A\setminus\fF_I(\Gamma) = \bigcap_{j=1}^m \Adm_I(\gamma_j) =
		\prod_{k=1}^M
		\begin{cases}
		\Big(\CC_{(k)}\setminus\bigcup_{j=1}^m \gamma^\sp_{j,k}\Big)\times\m_k, \text{ if } k\in I\\
		\cA_k, \text{ otherwise }
		\end{cases}
		\!\!\!\!\!\!\defeq \\
		&\defeq 
		\prod_{k=1}^M
		\begin{cases}
		(\CC_{(k)}\setminus\Gamma^\sp_k)\times\m_k, \text{ if } k\in I\\
		\cA_k, \text{ otherwise }
		\end{cases}
		\end{aligned}
		\]
		is called the $I$-restricted set of admissible points for $\Gamma$.
		In the special case $I=\{1,\dots,M\}$ we shall also write $\Adm(\Gamma) \coloneqq \Adm_\cA(\Gamma)$ for the full admissible set of $\Gamma$.
		
		\item If $\varphi = \bar{\varphi} \circ \Pi_\tau: \cA \to \cB$ with $\tau=\tau_\varphi:\{1,\dots,N\} \to \{1,\dots,M\}$ is a morphism of $\CC$-algebras in canonical factorization form, $U\subseteq\A$ is open and path-connected, and $\Gamma\coloneqq\sum_{j=1}^m n_j \gamma_j\in Z_1(U,\ZZ)$ is a 1-cycle, then $\fF_\varphi(\Gamma) \coloneqq \fF_{\im\tau_\varphi}(\Gamma)$ is called the set of $\varphi$-forbidden points for $\Gamma$.
		Its complement $\Adm_\varphi(\Gamma) \coloneqq \Adm_{\im\tau_\varphi}(\Gamma)$ is called the set of $\varphi$-admissible points for $\Gamma$.
		
		\item A subset $V\subseteq\A$ is called $\varphi$-admissible for $\Gamma$ if all its points are $\varphi$-admissible, that is, if $V\subseteq\Adm_\varphi(\Gamma)$.
		Symmetrically, given a subset $V\subseteq\A$, a 1-cycle $\Gamma\in Z_1(U,\ZZ)$ is called $\varphi$-admissible for $V$ if $V$ is a $\varphi$-admissible set for $\Gamma$.
		One defines analogously $\varphi$-forbidden subsets of $\cA$ with respect to a given 1-cycle and $\varphi$-forbidden 1-cycles with respect to a given subset of $\cA$.
	\end{enumerate}
\end{definition}

\begin{remarks}
	Let $\cA = \bigoplus_{k=1}^M \cA_k \xrightarrow{\varphi} \cB = \bigoplus_{\ell=1}^N \cB_\ell$ be a morphism of fully decomposed $\CC$-algebras with corresponding index subset $I\coloneqq \im \tau_\varphi \subseteq \{1,\dots,M\}$, let $U\subseteq\cA$ be open and path-connected, and let $\Gamma \coloneqq \sum_{j=1}^m n_j \gamma_j \in Z_1(U,\ZZ)$.
	\begin{enumerate}
		\item Clearly, if $I\subseteq J\subseteq \{1,\dots,M\}$, then $\fF_I(\Gamma) \subseteq \fF_J(\Gamma)$, or equivalently, $\Adm_I(\Gamma)\supseteq \Adm_J(\Gamma)$. 
		In particular, one always has $\fF_\varphi(\Gamma) \subseteq \fF_\cA(\Gamma)$ and $\Adm_\cA(\Gamma) \subseteq \Adm_\varphi(\Gamma)$.
		
		\item If $\A=(\A,\m)$ is local, then $\fF(\Gamma) = \Gamma^\sp\times\m \defeq \wtilde{\Gamma}$ is the cylindrical closure of (the support of) $\Gamma$.
		In general, if $\A$ is not connected/local, we only have $\fF_I(\Gamma)\supseteq\wtilde{\Gamma}_I$, i.e. $\Adm_I(\Gamma)\subseteq\A\setminus\wtilde{\Gamma}_I$. 
		
		\item By definition, one has
		\[
		\Adm(\Gamma) = \prod_{k=1}^M 
		\begin{cases}
		\Adm(\Gamma_k), \text{ if } k\in I\\
		\cA_k, \text{ otherwise }
		\end{cases}
		\!\!\!\!\!\!= \prod_{k=1}^M
		\begin{cases}
		\Adm(\Gamma^\sp_k)\times\m_k, \text{ if }  k\in I\\
		\cA_k, \text{ otherwise },
		\end{cases}
		\]
		where $\Adm(\Gamma_k)$ and $\Adm(\Gamma^\sp_k)$ are taken in the $\CC$-algebras $\A_k$ and $\CC_{(k)}$ respectively, $1\leq k\leq N$.
		
		
		\item A $\varphi$-admissible point means that the winding numbers of the $k$-th projections, $k\in I$, are well-defined: the projected points do not lie on the projected curves themselves. 
		On the other hand, the condition $(\varphi \circ T_{Z_0})_\#\Gamma\subseteq \cB^\times$ means that the integral of the $\cB$-valued differential form $\omega\coloneqq \d Z/\varphi(Z-Z_0)$ is well-defined over $\Gamma$, and so both notions are well compatible. 
		We will see in a jiffy why this should not surprise us at all.
		
		\item $\fF_I(\Gamma)$ is a closed subset of $\A$ as a finite union of continuous preimages of closed sets, hence the set of admissible points $\Adm_I(\Gamma)$ is open. 
		If $\Gamma$ consists of sufficiently non-pathological curves such that $\CC_{(k)}\setminus\gamma^\sp_{j,k}$ is dense in $\CC_{(k)}$ for all $k\in I$ and $1\leq j\leq m$, then $\Adm_I(\gamma_j)$, $1\leq j\leq m$, are finite products of open dense sets, hence $\Adm_I(\Gamma)$ too is dense open in $\A$ as an intersection of the finitely many $\Adm_I(\gamma_j)$-s.
		Thus, if $V\subseteq\A$ is open and $\Gamma$ is a not too pathological 1-cycle, then $V$ always contains admissible points for $\Gamma$.
		
		
		\item Finally, if $Z\in\Adm_\varphi(\Gamma)$ and keeping in mind that $\varphi$ is in particular continuous, then $\varphi(Z)\in\Adm_\cB(\varphi_\#\Gamma)$ since $\forall \ell\in \{1,\dots,N\}: \bar{\varphi}_\ell(\cA_{\tau(\ell)}^\times)\subseteq\cB_\ell^\times$. \qed	
	\end{enumerate}
\end{remarks}

\subsection{The Index}\label{subs:ind}

\begin{definition}[Index over $\A$]
	Let $\A\xrightarrow{\varphi}\B$ be a morphism of unital $\CC$-algebras, $U\subseteq\A$ an open and path-connected subset, and $\Gamma\coloneqq\sum_{j=1}^m n_j \gamma_j \in Z_1(U,\ZZ)$ a 1-cycle.
	Then the index of $\Gamma$ around $Z_0$ is given by
	\begin{equation}
	\begin{aligned}
	\Ind_\varphi(\Gamma,Z_0) &\coloneqq \Ind_\varphi^\ZZ(\Gamma,Z_0) \coloneqq \frac{1}{2\pi i} \int_\Gamma \frac{\d Z}{\varphi(Z-Z_0)} \defeq
	\sum_{j=1}^m n_j \frac{1}{2\pi i} \oint_{\gamma_j} \frac{\d Z}{\varphi(Z-Z_0)} \defeq \sum_{j=1}^m n_j \Ind_\varphi(\gamma_j,Z_0)
	\end{aligned}
	\end{equation}
	where $\Ind_\varphi^\ZZ$ indicates that $\Gamma$ has coefficients in $\ZZ$. 
	Moreover, we set $\Ind_\A \coloneqq \Ind_{\id_\A}$ for short.
\end{definition}

\begin{remarks}\
	\begin{enumerate}[topsep=-\parskip]
		\item For example, if $\gamma\in\Cc^0(\SS^1,\CC)$ is a loop and $z\notin\gamma$, then above notation says $\ind(\gamma,z) = \Ind_\CC(\gamma,z) = \Ind_\CC^\ZZ(\gamma,z) \eqqcolon \ind_\CC(\gamma,z)$, the latter being used to emphasize the fact that the index is taken in the complex plane.
		
		\item If $\cA \xrightarrow{\varphi} \cB$ is a morphism of $\CC$-algebras, $U\subseteq \cA$ open and path-connected, $\Gamma \in Z_1(U,\ZZ)$ a 1-cycle, and $Z_0\in \Adm(\Gamma)$, then $\Ind_\cA(\Gamma,Z_0) \in \cA$ and $\Ind_\varphi(\Gamma,Z_0) \in \cB$. 
	\end{enumerate}
\end{remarks}

\begin{lemma}[Periods of Cauchy's Reproducing Kernel \& Index Calculation over $\A$]\label{indexcalc}
	Let $U\subseteq\cA$ be open and path-connected and let $\Gamma\in Z_1(U,\ZZ)$ be a $\Cc^1_\pw$-regular 1-cycle.
	\begin{enumerate}
		\item Index with respect to composition of morphisms: if $\cA \xrightarrow{\psi} \cB \xrightarrow{\varphi} \cC$ are morphisms of $\CC$-algebras, then $\forall Z\in\Adm_\psi(\Gamma)$:
		\begin{equation}
		\Ind_{\varphi\circ\psi}(\Gamma,Z) = \Ind_\varphi(\psi_\# \Gamma,\psi(Z)).
		\end{equation}
		In particular, if $\varphi=\id_\cB$, then $\forall Z\in\Adm_\psi(\Gamma)$:
		\begin{equation}
		\Ind_\psi(\Gamma,Z) = \Ind_\cB(\psi_\# \Gamma,\psi(Z)).
		\end{equation}
		
		\item Index over a local $\A$: if $\A=(\A,\m)$ is local and $\sigma:\A\cong_{\Vect}\CC\oplus\m\twoheadrightarrow\A/\m\cong\CC$ is the canonical ``spectral'' quotient projection onto the scalars, then $\forall Z\in \Adm_\cA(\Gamma)$:
		\begin{equation}\label{localind}
		\Ind_\A(\Gamma,Z) = \Ind_\sigma(\Gamma,Z) = \ind_\CC(\Gamma^\sp,z) \in \ZZ,
		\end{equation}
		where recall the notation $z \defeq \sigma(Z)$ and $\Gamma^\sp \defeq \sigma_\# \Gamma$ is the ``spectral'' part of the 1-cycle $\Gamma$.
		It follows that if $\varphi: (\cA,\fm) \to (\cB,\fn)$ is a morphism of local $\CC$-algebras, then $\forall Z\in\Adm_\varphi(\Gamma) \equiv \Adm_\cA(\Gamma)$:
		\begin{equation}\label{localmorphind}
		\Ind_\varphi(\Gamma,Z) = \Ind_\cA(\Gamma,z) = \ind_\CC(\Gamma^\sp,z) \in \ZZ.
		\end{equation}
		
		\item Index over a direct sum of morphisms: if $\varphi \coloneqq \oplus_{\ell=1}^N \varphi_\ell: \cA \coloneqq \bigoplus_{\ell=1}^N \cA_\ell \to \cB \coloneqq \bigoplus_{\ell=1}^N \cB_\ell$, $\Gamma_\ell \coloneqq (\pr_\ell)_\# \Gamma$, and $Z=\oplus_{\ell=1}^N Z_\ell$, then $\forall Z\in\Adm_\varphi(\Gamma) = \prod_{\ell=1}^N \Adm_{\varphi_\ell}(\Gamma_\ell):$
		\begin{equation}\label{indempotentcycle}
		\Ind_\varphi(\Gamma,Z) = \oplus_{\ell=1}^N \Ind_{\varphi_\ell}(\Gamma_\ell,Z_\ell) \in \cB.
		\end{equation}
		In particular, if $\varphi_\ell: (\cA_\ell,\fm_\ell) \to (\cB_\ell,\fn_\ell)$, $1\leq \ell\leq N$, are morphisms of local $\CC$-algebras, then $\forall Z\in\Adm_\varphi(\Gamma) \equiv \Adm_\cA(\Gamma)$:
		\begin{equation}\label{morphind}
		\begin{aligned}
		\Ind_\varphi(\Gamma,Z) &= \oplus_{\ell=1}^N \Ind_{\cA_\ell}(\Gamma_\ell,Z_\ell) =\oplus_{\ell=1}^N \ind_\CC(\Gamma^\sp_\ell,\sigma_\ell(Z))=
		\varphi(\Ind_\cA(\Gamma,Z)) \in \ZZ^{\oplus N} \subseteq \cB,
		\end{aligned}
		\end{equation}
		where recall that $\Gamma^\sp_\ell \defeq (\sigma_\ell)_\# \Gamma$ is the $\ell$-th eigenvalue part of the 1-cycle $\Gamma$.
		
		\item Index over a general morphism: if $\varphi = (\oplus_{\ell=1}^N \bar{\varphi}_\ell) \circ \Pi_\tau: \cA \cong \bigoplus_{k=1}^M (\cA_k,\fm_k) \to \cB \cong \bigoplus_{\ell=1}^N (\cB_\ell,\fn_\ell)$ is a general morphism of $\CC$-algebras in canonical factorization form, then $\forall Z\in\Adm_\varphi(\Gamma):$
		\begin{equation}\label{indphi}
		\begin{aligned}
		\Ind_\varphi(\Gamma,Z) &= \oplus_{\ell=1}^N \Ind_{\bar{\varphi}_\ell} (\Gamma_{\tau(\ell)},Z_{\tau(\ell)}) = 
		\oplus_{\ell=1}^N \ind_{\CC_{\tau(\ell)}} (\Gamma^\sp_{\tau(\ell)},z_{\tau(\ell)}) = \varphi(\Ind_\cA(\Gamma,Z)) \in \ZZ^{\oplus N} \subseteq \cB,
		\end{aligned}
		\end{equation}
		provided $Z \in \Adm_\cA(\Gamma)$ for the last part of the equality.			
	\end{enumerate}
\end{lemma}

\begin{proof}
	By $\ZZ$-linearity of the definition of $\Ind_\varphi(-,Z)$ it suffices to verify the claims for a single loop $\gamma \in \Cc^1_\pw(\SS^1,U)$.	
	
	To (1):
	Since $\psi$ is linear, we have
	\[
	\begin{aligned}
	\Ind_{\varphi\circ\psi}(\gamma,Z) \defeq \int_\gamma \frac{\d W}{(\varphi\circ\psi)(W-Z)} = \int_0^1 \frac{\varphi(\psi(\gamma'(t))) \d t}{\varphi(\psi(\gamma(t))-\psi(Z))}= 
	\int_0^1 \frac{\varphi((\psi\circ\gamma)'(t)) \d t}{\varphi(\psi(\gamma(t))-\psi(Z))} = \Ind_\varphi(\psi\circ\gamma,\psi(Z))
	\end{aligned}
	\]	
	as required.	
	
	To (2):
	First assume (without loss of generality) that $Z=0\in\Adm_\cA(\gamma)$.
	By \cref{closed}, $\omega\coloneqq\d W/W$ is a $\d$-closed 1-form over $\A^\times$, hence its integral is invariant under continuous homotopy $\gamma \simeq \gamma^\sp$ there. 
	Thus
	\[
	\begin{aligned}
	\Ind_\cA(\gamma,0) &\defeq \frac{1}{2 \pi i} \oint_\gamma \frac{\d W}{W} = \frac{1}{2 \pi i} \oint_{\gamma^\sp} \frac{\d W}{W} = 
	\frac{1}{2 \pi i} \int_0^1 \frac{(\gamma^\sp)'(t) 1_{\A}}{\gamma^\sp(t)}\d t = 
	\frac{1}{2 \pi i} \oint_{\gamma^\sp}\frac{\d z}{z} = \ind_\CC(\gamma^\sp,0) = \Ind_\sigma(\gamma,0)
	\end{aligned}
	\]
	by (1).
	Next, substituting $V\coloneqq W-Z$ and $\delta\coloneqq\gamma-Z$, we get $\forall Z\in\Adm_\cA(\gamma)$:
	\[
	\begin{aligned}
	\Ind_\cA(\gamma,Z) &\defeq \frac{1}{2 \pi i} \oint_\gamma \frac{\d W}{W-Z} = \frac{1}{2 \pi i} \oint_\delta \frac{\d V}{V} = \ind_\CC(\sigma\circ\delta,0) =
	\ind_\CC(\sigma(\gamma)-\sigma(Z),0) = \ind_\CC(\sigma(\gamma),\sigma(Z)) = \Ind_\sigma(\gamma,Z)
	\end{aligned}
	\]
	as desired.
	Furthermore, if $(\cA,\fm) \xrightarrow{\varphi} (\cB,\fm)$ is a morphism of local $\CC$-algebras with corresponding canonical projections $\sigma_\cA:\cA\twoheadrightarrow\cA/\fm$ and $\sigma_\cB:\cB\twoheadrightarrow\cB/\fn$, then
	\[
	\begin{aligned}
	\Ind_\varphi(\gamma,Z) &= \Ind_\cB(\varphi\circ\gamma,\varphi(Z)) = \ind_\CC(\sigma_\cB\circ\varphi\circ\gamma,\sigma_\cB(\varphi(Z))) =
	\ind_\CC(\sigma_\cA\circ\gamma,\sigma_\cA(Z)) = \Ind_\cA(\gamma,Z)
	\end{aligned}
	\]
	by the previous discussion, (1), and \hyperref[scalarcom]{Diagram \ref*{scalarcom}}.
	
	To (3):
	We have
	\[
	\begin{aligned}
	\Ind_\varphi(\gamma,Z) &\defeq \frac{1}{2\pi i} \oint_\gamma \frac{\d W}{\varphi(W-Z)} = 
	\frac{1}{2\pi i} \oint_{\oplus_{\ell=1}^N \gamma_\ell} \frac{\d (\oplus_{\ell=1}^N W_\ell)}{\oplus_{\ell=1}^N \varphi_\ell(W_\ell-Z_\ell)} =
	\bigoplus_{\ell=1}^N \frac{1}{2\pi i} \oint_{\gamma_\ell} \frac{\d W_\ell}{\varphi_\ell(W_\ell-Z_\ell)} \defeq \bigoplus_{\ell=1}^N \Ind_{\varphi_\ell}(\gamma_\ell,Z_\ell),
	\end{aligned}
	\]
	where notice that the direct sum happens in $\cB=\bigoplus_{\ell=1}^N \cB_\ell$.
	If $\varphi_\ell$ are morphisms of local $\CC$-algebras, then the last equality follows from (2) and the application of $\varphi$ on $\ZZ^N \subseteq \cA$ to arrive in $\ZZ^N \subseteq \cB$.
	
	To (4):
	Since $\cA \cong \bigoplus_{k=1}^M (\cA_k,\fm_k)$, we have $\Ind_\cA(\gamma,Z) = \oplus_{k=1}^M \Ind_{\cA_k}(\gamma_k,Z_k)$ by (3).
	Writing $\bar{\varphi} \coloneqq \oplus_{\ell=1}^N \bar{\varphi}_\ell$ and thus $\varphi = \bar{\varphi} \circ \Pi_\tau$, we therefore have:
	\[
	\begin{aligned}
	\varphi(\Ind_\cA(\gamma,Z)) &= \bar{\varphi}(\Pi_\tau(\oplus_{k=1}^M \Ind_{\cA_k}(\gamma_k,Z_k))) = 
	\bar{\varphi}(\oplus_{\ell=1}^N \Ind_{\cA_{\tau(\ell)}}(\gamma_{\tau(\ell)},Z_{\tau(\ell)})) =
	\Ind_{\bar{\varphi}}(\oplus_{\ell=1}^N \gamma_{\tau(\ell)},\oplus_{\ell=1}^N Z_{\tau(\ell)}) = \\
	&= \Ind_{\bar{\varphi}}(\Pi_\tau(\gamma),\Pi_\tau(Z)) = \Ind_{\bar{\varphi}\circ\Pi_\tau}(\gamma,Z) \defeq \Ind_\varphi(\gamma,Z)
	\end{aligned}
	\]
	by (3) and (1), as desired.
\end{proof}

\begin{definition}
	If $\cA$ is a $\CC$-algebra and $Z_0\in \cA$, then $[Z_0] \coloneqq Z_0 + \nil\cA \in \cA/\nil\cA \cong \CC^{\oplus M}$ will denote the equivalence class of $Z_0 \mod \nil\cA$ as a subset of $\cA$.
\end{definition}

\begin{corollary} 
	Let $\cA \xrightarrow{\varphi} \cB$ be a morphism of $\CC$-algebras, let $U\subseteq\cA$ be open and path-connected, and let $\Gamma\in Z_1(U,\ZZ)$ be a 1-cycle.
	\begin{enumerate}
		\item If $Z_0\in\Adm_\varphi(\Gamma)$ and $X_0\in\nil\cA$, then $\Ind_\varphi(\Gamma,Z_0+X_0) = \Ind_\varphi(Z_0)$.
		Hence we have a well-defined map $\Ind_\varphi(\Gamma,-): \Adm_\varphi(\Gamma)/\nil\cA \to \ZZ^N \subseteq \cB$, $[Z_0] \mapsto \Ind_\varphi(\Gamma,[Z_0])$, where notice that $\Adm_\varphi(\Gamma)/\nil\cA \subseteq \CC^M$.
		In particular, if $0\in\Adm_\varphi(\Gamma)$, then $\Ind_\A(\Gamma,X_0) = \Ind_\A(\Gamma,0)$.
		
		\item If $(\cA,\fm) \xrightarrow{\varphi} (\cB,\fn)$ is a morphism of local $\CC$-algebras, then the Index of $\Gamma$ is invariant under $\varphi$, that is, for a fixed $Z\in\Adm(\Gamma)$ the quantity $\Ind_\cB(\varphi_\# \Gamma,\varphi(Z))$ does not depend on $\varphi$.
		In particular, if $\cA = (\cA,\fm)$ is local, then $\Ind_\cA(\Gamma,Z)$ is invariant under $\CC$-algebra endomorphisms of $\cA$.
	\end{enumerate}
\end{corollary}

\begin{proof}
	To (1): Since $\nil\cA = \bigoplus_{k=1}^M \fm_k$, this follows from \cref{indphi}, which depends only on the eigenvalues of $Z_0 + X_0$.
	
	To (2): This is immediate from the second part of \cref{indexcalc} (2).
\end{proof}

\begin{remarks}\ 
	\begin{enumerate}[topsep=-\parskip]
		\item Using \cref{indphi}, $\Ind_\varphi$ can also be well-defined for merely continuous 1-cycles by means of the topological degree since the projections always have well-defined topological degrees in the respective complex plane(s).
		
		\item \cref{indphi} makes it clear that, even though in the local case the Index is $\ZZ$-valued like in the complex plane, $\Ind_\varphi$ is more accurately seen as $\cB$-valued.
		This illuminates the fact that in our setting it is only natural to consider 1-cycles with more general coefficients than $\ZZ$, for example $\cA$ or $\cB$, as we shall do in \cite{MG04}.
		
		\item Let $Z_0\coloneqq 0$, then $\Ind_\A(-,0): \gamma \mapsto \Ind_\A(\gamma,0) = \oplus_{k=1}^M \ind(\gamma^\sp_k,0) \in \ZZ^M \subseteq \cA$ is only a one-way homotopy invariant in $\A^\times$. 
		Indeed, if $\gamma,\delta\in\Cc^0(\SS^1,\A^\times)$ with $\gamma \simeq \delta$ in $\A^\times$, then $\gamma^\sp_k \simeq \delta^\sp_k$ in $\CC_{(k)}^{\times}$, $1\leq k\leq M$, and hence $\Ind_\A(\gamma,0) = \Ind_\A(\delta,0)$, but clearly the converse need not hold.
		
		\item Let $V$ be a $\CC$-vector space and $\A=(\A,\m)$ a choice of a local $\CC$-algebra structure on $V$. 
		Then \cref{localind} shows that, as expected, $\Ind_\A$ depends on this choice because it is determined by the choice of one-dimensional complex subspace of $V$ onto which $\gamma \subseteq V$ and $Z_0 \in V$ are projected.
		On the bright side, $\Ind_\A$ is at least invariant under algebra endomorphisms by the previous corollary.
		
		\item By definition, we say that $\gamma:\SS^1\hookrightarrow\A^\times$ is a \textit{simple} loop around 0 or around the singular cone $\A_\sing$ iff $\Ind_\A(\gamma,0) = 1$. 
		Likewise, $\gamma$ is a \emph{simple} loop around a point $Z_0\in\A$ or around $Z_0+\A_\sing$ iff $\Ind_\A(\gamma,Z_0)=1$.
		
		\item Further notice that $\Ind_\varphi(\gamma,Z_0) = 0$ if and only if $\forall k\in I: \ind(\gamma^\sp_k,\sigma_k(Z_0))=0$. 
		Similarly, $\Ind_\varphi(\gamma,Z_0) \in\cB^\times$ if and only if $\forall k\in I: \ind(\gamma^\sp_k,\sigma_k(Z_0))\neq 0$.
		
	\end{enumerate}
\end{remarks}

\begin{lemma}[Index is locally constant]
	Let $U\subseteq\A$ be open and path-connected and $\Gamma\in Z_1(U,\ZZ)$. 
	Then the map
	\[
	\Ind_\varphi(\Gamma,-):\Adm_\varphi(\Gamma)\to\ZZ,\ Z\mapsto\int_\Gamma\frac{\d W}{\varphi(W-Z)}
	\]
	is locally constant. 
	In particular, $\Gamma$ divides $\Adm_\varphi(\Gamma)$ into two types of components: points with zero and non-zero Index.
\end{lemma}

\begin{proof}
	This is a parameter integral with a continuously $\A$-differentiable integrand, therefore:
	\[
	\frac{\d}{\d Z} \Ind_\varphi(\Gamma,Z) = \int_\Gamma \frac{\d}{\d Z} \left(\frac{1}{\varphi(W-Z)}\right)\d W = \int_\Gamma \frac{\d W}{\varphi(W-Z)^2} = 0,
	\]
	since $-1/\varphi(W-Z)$ is a global primitive of the integrand in the variable $W$.
\end{proof}

Notice that the zero or non-zero component of $\Ind_\varphi(\Gamma,-)$ is allowed to be empty.
We can now determine them:

\begin{corollary}[Components of $\Gamma$]
	Let $\cA = \bigoplus_{k=1}^M \cA_k \xrightarrow{\varphi} \cB = \bigoplus_{\ell=1}^N \cB_\ell$ be a morphism of $\CC$-algebras with corresponding index subset $I\coloneqq \im\tau_\varphi\subseteq \{1,\dots,M\}$, let $U\subseteq\A$ be open and path-connected, let $\Gamma\in Z_1(U,\ZZ)$ be a 1-cycle, and $Z_0\in\Adm_\varphi(\Gamma)$ a point.
	Then:
	\begin{enumerate}
		\item Components of $\Ind_\varphi(\Gamma,-)$: each component of $\Gamma$ is of the form 
		\[
		\prod_{k=1}^M 
		\begin{cases}
		C_k\times\m_k, \text{ if } k\in I\\
		\cA_k, \text{ otherwise,}
		\end{cases}		
		\]
		where $C_k$ is a component in $\CC_{(k)}=\CC$ of the locally constant function $\ind_\CC(\Gamma^\sp_k,-)$, $k\in I$. 
		
		\item Invertible Index: $\Ind_\varphi(\Gamma,Z_0)\in\cB^\times$ if and only if for all $k\in I$ the component of $\Gamma^\sp_k$ containing the projection $\sigma_k(Z_0)$ is bounded in the respective complex plane $\CC_{(k)}=\CC$.
	\end{enumerate}
\end{corollary}

\begin{proof}
	To (1): This is immediate from \cref{indphi}.
	
	To (2): Again by \cref{indphi} we have $\Ind_\varphi(\Gamma,Z_0)\in\cB^\times \Leftrightarrow \forall k\in I: \ind(\Gamma^\sp_k,\sigma_k(Z_0)) \neq 0 \Leftrightarrow$ the component of $\Gamma^\sp_k$ in $\CC$ containing $\sigma_k(Z_0)$ is not the unbounded one.
\end{proof}

\subsection{Cauchy's Integral Formula over $\cA$ and Consequences}\label{subs:CIF}

\begin{proposition}[Cauchy's Integral Formula: Cauchy's Reproducing Kernel over $\A$]\label{cif}
	Let $\cA\xrightarrow{\varphi}\cB$ be a morphism of $\CC$-algebras and let $Z_0\in\A$. 
	Let $\Delta\coloneqq\Delta_r(Z_0)$ be the open polydisc around $Z_0$ of radius $r$ and let $f\in\O_\varphi(\Delta)$. 
	If $\gamma\in\Cc^1_\pw(\SS^1,\Delta)$ is an admissible loop for $Z_0$, then
	\begin{equation}\label{CIF}
	f(Z_0) \Ind_\varphi(\gamma,Z_0) = \frac{1}{2\pi i} \oint_\gamma \frac{f(Z)}{\varphi(Z-Z_0)}\d Z.
	\end{equation}
	In particular, if $Z_0 \in \nil\A$ and $\gamma$ is a simple $\Cc^1_\pw$-loop in $\A^\times$ around $Z_0$, then
	\begin{equation}\label{CIFnil}
	f(Z_0) = \frac{1}{2\pi i} \oint_\gamma \frac{f(Z)}{Z-Z_0}\d Z.
	\end{equation}
\end{proposition}

\begin{proof}
	It suffices to prove the claim for $Z_0=0$, otherwise simply consider $\tilde{f}(Z)\coloneqq f(Z+Z_0)$ and $\tilde{\gamma}(t)\coloneqq \gamma(t)-Z_0$.
	Notice that in general both $f$ and $\Ind_\varphi$ take values in $\B$.
	We first show the claim for a local morphism $\varphi: (\cA,\fm) \to (\cB,\fn)$. 
	If we put $d\coloneqq\dim_\CC\m$, we have $\Delta^\times \coloneqq \Delta \cap \A^\times = \Delta \cap (\CC^{\times} \times \CC^d) = \DD^{\ast}_r(0) \times \DD_r(0)^d$ as topological spaces, where $\DD^{\ast}_r(0)$ denotes the punctured open disk with radius $r$ and center 0 in $\CC$.
	Therefore $\Delta^\times$ has the same homotopy structure as $\A^\times$, namely $\pi_0(\Delta^\times)=1$, $[\SS^1;\Delta^\times] = [\SS^1;\DD^{\ast}_r(0)]$, and hence $\pi_1(\Delta^\times)=\ZZ$, where explicitly $\gamma \simeq \gamma_\vareps: [0,1]\to \DD^{\ast}_r(0) \times \{0_\fm\}$, $t\mapsto \vareps e^{2\pi i \nu t}$, for an arbitrary fixed $0<\vareps<r$ and some winding number $\nu\in\ZZ$. 
	By \cref{localmorphind} we have:
	\[
	\varphi\bigg(\frac{\gamma'_\vareps(t)}{\gamma_\vareps(t)}\bigg) = \varphi (2\pi i \nu) = 2\pi i\nu \defeq 2\pi i \ind_\CC(\gamma_\vareps,0) = 2\pi i \Ind_\varphi(\gamma,0).
	\]
	Since $f\in\O_\varphi(\Delta)$, the $\B$-valued differential form
	\[
	\omega\coloneqq\frac{f(Z)-f(0)}{\varphi(Z)}\d Z
	\]
	defined over $\Delta^\times$ is $\d$-closed, hence its integrals are invariant under homotopy inside $\Delta^\times$. 
	Let $\norm{\cdot}_\B$ be any submultiplicative norm on $\B$. 
	We can now proceed as in the proof of the classical Cauchy Integral Formula:
	\[
	\begin{aligned}
	&\norm{\frac{1}{2\pi i} \oint_{\gamma} \frac{f(Z)}{\varphi(Z)}\d Z - \Ind_\varphi(\gamma,0) f(0)}_\B = 
	\norm{\frac{1}{2\pi i} \oint_{\gamma} \frac{f(Z)-f(0)}{\varphi(Z)}\d Z}_\B =
	\norm{\frac{1}{2\pi i} \oint_{\gamma_\vareps} \frac{f(Z)-f(0)}{\varphi(Z)}\d Z}_\B = \\
	&\norm{\frac{1}{2\pi i} \int_0^1 \frac{f(\gamma_\vareps(t))-f(0)}{\varphi(\gamma_\vareps(t))} \varphi(\gamma'_\vareps(t))\d t}_\B =
	\norm{\Ind_\varphi(\gamma,0) \int_0^1 \big(f(\gamma_\vareps(t))-f(0)\big)\d t}_\B \leq \abs{\Ind_\varphi(\gamma,0)} \norm{f-f(0)}_{\B,\gamma_\vareps} \xrightarrow{\vareps\to 0} 0
	\end{aligned}
	\]
	as desired.
	
	Next assume that $\varphi$ has the form $\varphi\coloneqq\oplus_{\ell=1}^N \varphi_\ell:\cA = \bigoplus_{\ell=1}^N (\cA_\ell,\fm_\ell)\to\bigoplus_{\ell=1}^N (\cB_\ell,\fn_\ell)$.
	Writing $Z = \oplus_{\ell=1}^N Z_\ell$, $f(Z)=\oplus_{\ell=1}^N f_\ell(Z_\ell)$ for $\varphi_\ell$-holomorphic $f_\ell$, and $\gamma(t) = \oplus_{\ell=1}^N \gamma_\ell(t)$, one obtains:
	\[
	\begin{aligned}
	\frac{1}{2\pi i} \oint_\gamma \frac{f(Z)}{\varphi(Z)}\d Z &= \bigoplus_{\ell=1}^N \frac{1}{2\pi i} \oint_{\gamma_\ell} \frac{f_\ell(Z_\ell)}{\varphi_\ell(Z_\ell)}\d Z_\ell =
	\bigoplus_{\ell=1}^N f_\ell(0) \Ind_{\varphi_\ell}(\gamma_\ell,0) = 
	\bigg(\bigoplus_{\ell=1}^N f_\ell(0)\bigg) \bigg(\bigoplus_{\ell=1}^N \Ind_{\varphi_\ell}(\gamma_\ell,0)\bigg) = f(0)\Ind_\varphi(\gamma,0)
	\end{aligned}
	\]
	by \cref{morphind} as required.
	Finally, if $\cA\xrightarrow{\varphi}\cB$ is a general morphism of $\CC$-algebras with canonical factorization $\varphi = \bar{\varphi}\circ \Pi_\tau$ as in \cref{MorFact}, then by \cref{inclocal} it induces a corresponding factorization of $f$ as
	\[
	\begin{tikzcd}[column sep = scriptsize, row sep=normal]
	U \ar[r,"f"] \ar[two heads,d,"\Pi_\tau",swap] & \cB \\
	\Pi_\tau(U) \ar[ur,"\bar{f}",swap]
	\end{tikzcd}
	\]
	for some $\bar{f}\in\cO_{\bar{\varphi}}(\Pi_\tau(U))$ of the form $\bar{f} = \oplus_{\ell=1}^N \bar{f}_\ell$.
	Putting $\bar{\gamma} \coloneqq \Pi_\tau\circ\gamma$ and noting $\d\bar{\gamma} \defeq \d(\Pi_\tau\circ\gamma) = (\Pi_\tau\circ\gamma)'(t)\d t = (\Pi_\tau\circ\gamma')(t)\d t$ as $\Pi_\tau$ is linear, we obtain
	\[
	\begin{aligned}
	\frac{1}{2\pi i} \oint_\gamma \frac{f(Z)}{\varphi(Z)}\d Z &= \frac{1}{2\pi i} \int_0^1 \frac{f(\gamma(t)) \varphi(\gamma'(t))}{\varphi(\gamma(t))}\d t = 
	\frac{1}{2\pi i} \int_0^1 \frac{(\bar{f} \circ \Pi_\tau)(\gamma(t)) (\bar{\varphi} \circ \Pi_\tau)(\gamma'(t))}{(\bar{\varphi} \circ \Pi_\tau)(\gamma(t))} \d t = \frac{1}{2\pi i} \oint_{\bar{\gamma}} \frac{\bar{f}(W) \d W}{\bar{\varphi}(W)} =\\
	&= \Ind_{\bar{\varphi}}(\bar{\gamma},0) \bar{f}(0) = \Ind_\varphi(\gamma,0) f(0)
	\end{aligned}
	\]
	by the previous discussion and \cref{indexcalc} as required.
\end{proof}

%
%
%

\begin{remarks}\ 
	\begin{enumerate}[topsep=-\parskip]
		\item A byproduct of the argument used in the proof of \cref{cif} is that every point $Z_0\in\A$ has an open neighbourhood containing an admissible loop of an arbitrary \textit{integral} index $\Ind_\varphi(\gamma,Z_0)\in\ZZ^{\oplus M}=\pi_1$.
		
		\item Putting $\fI_\gamma \coloneqq \Ind_\varphi(\gamma,-)$, it follows from \cref{cif} that $f|_\gamma$ determines $f$ in all components of the set $\Adm(\gamma)\cap\fI_\gamma^{-1}((\ZZ\setminus\{0\})^N)$, which contains stuff outside of $\Delta$: in particular, this set is always unbounded as it contains a copy of $\nil\A$. 
		In other words, one obtains a $\varphi$-holomorphic continuation of $f$.
		More precisely:
		
	\end{enumerate}
\end{remarks}

\begin{lemma}[Continuation to a local cylinder]\label{localcylindercont}
	Let $(\A,\m)\xrightarrow{\varphi}(\B,\n)$ be a morphism of local $\CC$-algebras, let $\Delta\coloneqq\Delta_r(Z_0)\subseteq\A$ be a polydisk, and $f\in\O_\varphi(\Delta)$.
	Then for any loop $\gamma\in\Cc^1_\pw(\SS^1,\Delta)$ the function $f$ has a $\varphi$-holomorphic continuation to $\Delta\cup(C_\gamma\times\m)$ given by $\forall Z\in \Delta\cap(C_\gamma\times\m)\ \forall X\in\m:$
	\begin{equation}
		f(Z+X) \coloneqq \frac{1}{2\pi i\Ind_\varphi(\gamma,Z)} \oint_\gamma \frac{f(W)}{\varphi(W-Z-X)}\d W,
	\end{equation}
\end{lemma}
where $C_\gamma\coloneqq\{z\in\CC\setminus\gamma^\sp: \ind(\gamma^\sp,z)\neq 0\}$ is the non-zero-index component(s) of $\gamma^\sp \coloneqq \sigma_\cA\circ\gamma$.
\begin{proof}
	Firstly, notice that $\Delta\cap(C_\gamma\times\m)\neq\emptyset$ is open and for $X=0$ the RHS agrees with $f$ on $\Delta\cap(C_\gamma\times\m)$ by \cref{cif}. 
	Secondly, since $\cA$ is local, $C_\gamma\times\m\subseteq\Adm(\gamma)$ and $\forall Z\in C_\gamma\times\m: \Ind_\varphi(\gamma,Z) = \ind(\gamma^\sp,z)\neq 0$, and since $X\in\m$ and $Z\in\Adm(\gamma)$, we have $Z+X\in\Adm(\gamma)$ and $\Ind_\varphi(\gamma,Z+X) = \Ind_\varphi(\gamma,Z)\neq 0$, so the RHS is well-defined.
	In particular, if $Z_1+X_1 = Z_2+X_2$ for some $Z_{1,2}\in U$ and $X_{1,2}\in\m$, i.e. $Z_2 = Z_1 + X$ for $X\coloneqq X_1-X_2\in\m$, then
	$Z_1\in\Adm(\gamma)\Leftrightarrow Z_2\in\Adm(\gamma)$ and $\Ind_\varphi(\gamma,Z_1+X_1) = \Ind_\varphi(\gamma,Z_1) = \Ind_\varphi(\gamma,Z_2) = \Ind_\varphi(\gamma,Z_2+X_2)$, hence $f(Z+X)$ itself is well-defined.
	Finally, it is immediate that the so defined function is $\varphi$-holomorphic since $\Ind_\varphi(\gamma,Z)$ is locally constant and one can $\varphi$-differentiate under the integral sign.
\end{proof}

\subsection{Analyticity}\label{subs:anal}

\begin{definition}
	Let $\fA\xrightarrow{\varphi}\fB$ be a morphism of not necessarily unital $\KK$-algebras and let $U\subseteq\fA$ be open. A function $f:U\to\fB$ is called $\varphi$-analytic at $Z_0\in U$ if in a neighbourhood of $Z_0$ it is given by a convergent power series of the form
	\[
	f(Z) = \sum_{\ell=0}^{\infty} B_\ell (Z-Z_0)^\ell \defeq \sum_{\ell=0}^{\infty} B_\ell \varphi(Z-Z_0)^\ell
	\]
	for some $B_\ell\in\fB$. 
	The $\KK$-algebra of $\varphi$-analytic functions at $0$ will be denoted by $\fB_\fA\{Z\} \coloneqq \fB_\varphi\{Z\} \coloneqq \fB\{\varphi(Z)\}$.
\end{definition}

\begin{remark}
	If $\fA\xrightarrow{\varphi}\fB$ is a morphism of not necessarily unital $\KK$-algebras, where $\fA$ is connected and non-unital, then $\exists \nu\in\NN\ \forall Z\in\fA: Z^\nu = 0$, and therefore $\fB\{\varphi(Z)\} = \fB[[\varphi(Z)]] = \fB[\varphi(Z)]$.
\end{remark}

\begin{definition}[Products of Spectral Balls and Annuli]
	Let $\A=\bigoplus_{k=1}^M \A_k$ be the decomposition of $\A$ into local Artinian $\CC$-algebras and let $Z_0\in\A$.
	Let $\underline{r}\coloneqq(r_1,\dots,r_M)\in [0,\infty]^M$ and $\underline{R}\coloneqq(R_1,\dots,R_M)\in [0,\infty]^M$.
	Then:
	\begin{enumerate}
		\item Spectral polyball (polycylinder): $\BB^\sp(Z_0,\underline{R}) \coloneqq \prod_{k=1}^M \BB^\sp_{\A_k}(\pr_k(Z_0),R_k)$ and $\wbar{\BB}^\sp(Z_0,\underline{R}) \coloneqq \prod_{k=1}^M \wbar{\BB}^\sp_{\A_k}(\pr_k(Z_0),R_k)$.
		
		\item Spectral polyannulus: $\AA^\sp(Z_0,\underline{r},\underline{R}) \coloneqq \prod_{k=1}^M \AA^\sp_{\A_k}(\pr_k(Z_0),r_k,R_k)$ and $\bar{\AA}^\sp(Z_0,\underline{r},\underline{R}) \coloneqq \prod_{k=1}^M \bar{\AA}^\sp_{\A_k}(\pr_k(Z_0),r_k,R_k)$.
	\end{enumerate}
\end{definition}


\begin{lemma}\label{limsupmax}
	Let $(a_i(k))_{k\in\NN}\subseteq\RR$, $1\leq i\leq n$, be $n$ sequences of real numbers.
	Then
	\[
	\max_{1\leq i\leq n} \limsup_{k\to\infty} a_i(k) =
	\limsup_{k\to\infty} \max_{1\leq i\leq n} a_i(k)
	\]
\end{lemma}

\begin{proof}
	To ``$\leq$'': We have $\forall k\in\NN\ \forall 1\leq j\leq n: a_j(k) \leq
	\max\limits_{1\leq i\leq n} a_i(k) \Rightarrow 
	\forall 1\leq j\leq n: \limsup\limits_{k\to\infty} a_j(k) \leq
	\limsup\limits_{k\to\infty} \max\limits_{1\leq i\leq n} a_i(k) \Rightarrow
	\max\limits_{1\leq j\leq n} \limsup\limits_{k\to\infty} a_j(k) \leq
	\limsup\limits_{k\to\infty} \max\limits_{1\leq i\leq n} a_i(k)$.
	
	To ``$\geq$'': Put $a(k)\coloneqq\max\limits_{1\leq i\leq n} a_i(k)$ and $a\coloneqq\limsup\limits_{k\to\infty} a(k)\in [-\infty,\infty]$.
	Let $I\subseteq\NN$ such that $a=\lim\limits_{k\in I} a(k)$.
	Clearly, $\exists 1\leq j\leq n: a(k) = a_j(k)$ for infinitely many $k\in I$, and denote their set by $J\subseteq I$.
	In particular, $\lim\limits_{k\in J} a_j(k) = a$ and $a\leq\limsup\limits_{k\to\infty} a_j(k)$.
	Thus $\limsup\limits_{k\to\infty} \max\limits_{1\leq i\leq n} a_i(k) = a \leq
	\limsup\limits_{k\to\infty} a_j(k) \leq 
	\max\limits_{1\leq i\leq n} \limsup\limits_{k\to\infty} a_j(k)$, i.e. $\max\limits_{1\leq i\leq n} \limsup\limits_{k\to\infty} a_j(k) \geq 
	\limsup\limits_{k\to\infty} \max\limits_{1\leq i\leq n} a_i(k)$.
\end{proof}

\begin{lemma}[Root test in Banach spaces]\label{rootest}
	Let $(E,\norm{\cdot})$ be a Banach space, let $a\coloneqq\sum_{n=0}^\infty a_n$ be a series in $E$, and define $\alpha\coloneqq\limsup_{n\to\infty}\norm{a_n}^{1/n}$.
	\begin{enumerate}
		\item If $\alpha<1$, then the series converges.
		\item If $\alpha>1$, then the series diverges.
		\item If $\alpha=1$, then no conclusion can be drawn.
		\item $\alpha$ is independent of the choice of an equivalent vector norm $\norm{\cdot}'$. 
	\end{enumerate}
	In particular, if $E$ is finite-dimensional, then $\alpha$ is independent of any choice of $\norm{\cdot}$ on $E$. 
\end{lemma}

\begin{proof}
	One readily verifies that the comparison and the root test are valid in any Banach space $(E,\norm{\cdot})$ since the proofs are the same and the (counter)-example(s) for the border case $\alpha=1$ are the same after a suitable choice of an identification $\RR\hookrightarrow E$ with a 1-dimensional subspace of $E$. 
	We only verify that $\alpha$ is independent of the choice of an equivalent vector norm
	$\norm{\cdot}$.
	So, let $\norm{\cdot}'$ and $\norm{\cdot}''$ be two equivalent vector norms on $E$ and let $c'>0$ and $c''>0$ be constants such that $\norm{\cdot}''> c'\norm{\cdot}'$ and $\norm{\cdot}'> c''\norm{\cdot}''$.
	Since $c^{1/n}\to 1$ as $n\to\infty$, it follows by symmetry that $\limsup_{n\to\infty}\norm{a_n}'^{1/n}=\limsup_{n\to\infty}\norm{a_n}''^{1/n}$.
\end{proof}

\begin{lemma-definition}[Analytic Radius]\label{analrad}
	Let $(\A,\norm{\cdot})\in\CBanAlg_\KK$ and let $f(Z)=\sum_{n=0}^\infty A_n Z^n\in\A[[Z]]$.
	Then the analytic radius of the power series $f$ is defined as
	\begin{equation}
	R\coloneqq R_f\coloneqq \frac{1}{\limsup\limits_{n\to\infty}\norm{A_n}^{1/n}}\in[0,\infty]
	\end{equation}
	and is independent of the choice of an equivalent vector norm.
	In particular, if $\A$ is finite-dimensional, then $R_f$ is independent of the choice of any (submultiplicative or not) norm $\norm{\cdot}$ for $\A$.
\end{lemma-definition}

\begin{proof}
	The proof is the same as of the previous lemma.
\end{proof}

\begin{lemma}\label{limsup}
	Let $(a_i(\ell))_{\ell\in\NN}\subseteq\CC$, $1\leq i\leq n$, be $n$ sequences of complex numbers, let $\alpha_i\coloneqq\limsup_{\ell\to\infty}\abs{a_i(\ell)}^{1/\ell}$, $1\leq i\leq n$, and put $\alpha\coloneqq\max_{1\leq i\leq n} \alpha_i$.
	If there is exactly one $1\leq i_0\leq n$ such that $\alpha=\alpha_{i_0}$, then 
	\[
	\limsup_{\ell\to\infty}\bigg\lvert\sum_{i=1}^n a_i(\ell)\bigg\rvert^{1/\ell} = \alpha
	\]
\end{lemma}

\begin{proof}
	It is more convenient to work with the radii $r_i\coloneqq 1/\alpha_i, r\coloneqq 1/\alpha=\min_{1\leq i\leq n} r_i\in [0,\infty]$, $1\leq i\leq n$.
	The proof is by induction on $n$.
	The case $n=2$, where $r_1\neq r_2$, is well-known from the convergence properties of power series in the Complex Analysis of One Variable.
	Now, let $n\geq 3$ and let $i_0$ be the unique index with $r = r_{i_0}$.
	Since $\forall 1\leq j\leq n,\ j\neq i_0: r_j>r_{i_0}$, the sequence $a'_j(\ell)\coloneqq a_{i_0}(\ell) + a_j(\ell)$ has again a radius $r_{i_0}$.
	Now iterate the same argument for the remaining indices $1\leq j\leq n$, $j\neq i_0$.
\end{proof}

\begin{remark}
	When $n=2$ and $r_1=r_2$ it is not possible to deduce $r$ only from $r_1$ and $r_2$ without closer analysis of the sum $a_1(\ell) + a_2(\ell)$.
\end{remark}

\begin{proposition}[Analytic Radius and Polycylinder of Convergence]\label{convergencepolycylinder}
	Let $\A\in\fdCAlg_\CC$, let $f(Z)=\sum_{\ell=0}^\infty A(\ell) Z^\ell\in\A[[Z]]$, let $\norm{\cdot}$ be an arbitrary submultiplicative norm on $\A$, and put $R\coloneqq R_f$. 
	Then:
	\begin{enumerate}
		\item If $\A=\bigoplus_{k=1}^M \A_k$ is the decomposition of $\A$ into Artin local $\CC$-algebras and $A(\ell)\eqqcolon\oplus_{k=1}^M a_k(\ell)$ accordingly, then $R=\min_{1\leq k\leq M} R_k$, where $R_k\coloneqq \limsup_{\ell\to\infty}\norm{a_k(\ell)}^{1/\ell}$, $1\leq k\leq M$, are the individual radii.
		
		\item If $R>0$, then for any choice of a submultiplicative norm $\norm{\cdot}$ the series $f$ is locally normally convergent on the $R$-ball $\BB_R^{\norm{\cdot}}(0)\coloneqq\{Z\in\A:\norm{Z}<R\}$.
		
		\item If $R<\infty$, then for any choice of a vector norm $\norm{\cdot}$ and any $r>R$ there exists a point $Z_0\in\A$ with $\norm{Z_0}=r$ such that $f$ is divergent at $Z_0$.
		
		\item If $R>0$, then $f$ is locally normally convergent on $\BB'_R(0) \coloneqq \bigcup_{\norm{\cdot}}\BB_R^{\norm{\cdot}}(0)$, where the union is taken over all submultiplicative norms on $\A$.
		
		\item In fact, we have $\BB'_R(0) = \BB_\A^\sp(0,R) \defeq \{Z\in\A: \rho_\A(Z)<R\}$ is the spectral $R$-ball of $\A$ around 0.
		
		\item Cylinder of convergence: if $(\A,\m)$ is local, then $\BB^\sp_\A(0,R) = \DD_R(0)\times\m$.
		
		\item If $\A=(\A,\m)$ is local, then $f$ is divergent\footnote{it has been recently brought to our attention that a proof of this fact is essentially contained in \cite{Ketch};} everywhere outside of $\wbar{\BB}^\sp_\A(0,R)$.
		
		\item Spectral radius of divergence: for a local $\CC$-algebra $\A=(\A,\m)$ define
		\[
		D^\sp\coloneqq\frac{1}{\limsup\limits_{\ell\to\infty}\rho_\A(A(\ell))^{1/\ell}}
		\in	[0,\infty]
		\]
		Then $D^\sp\geq R$ and $f$ diverges at any $Z_0\in\A$ with $\rho_\A(Z_0)>D^\sp$.
		
		\item Polycylinder of convergence: if $\cA\xrightarrow{\varphi}\cB$ is a morphism of $\CC$-algebras and $\cA=\bigoplus_{k=1}^M \cA_k$ is the decomposition of $\cA$ into connected $\CC$-algebras, then $g(Z)\in\B\{\varphi(Z-Z_0)\}$ is locally uniformly convergent on some spectral polycylinder 
		\[
		\BB^\sp(Z_0,\underline{R}) \defeq \prod_{k=1}^M \BB^\sp_{\A_k}(\pr_k(Z_0),R_k) = \prod_{k=1}^M \DD_{R_k}(\sigma_k(Z_0))\times\m_k,
		\]
		where $\underline{R}\coloneqq(R_1,\dots,R_M)\in (0,\infty]^M$, and is divergent outside of it.
	\end{enumerate}	
\end{proposition}

\begin{proof}
	To (1):
	Since $R$ is independent of the choice of norm, we can pick $\norm{\cdot}_\oplus$.
	Applying \cref{limsupmax}, we get:
	\[
	\begin{aligned}
	R &= \frac{1}{\limsup\limits_{\ell\to\infty}\norm{A(\ell)}^{1/\ell}_\oplus} = 
	\frac{1}{\limsup\limits_{\ell\to\infty} (\max\limits_{1\leq k\leq M} \norm{a_k(\ell)}_{\A_k})^{1/\ell}} =
	\frac{1}{\max\limits_{1\leq k\leq M}\limsup\limits_{\ell\to\infty} \norm{a_k(\ell)}_{\A_k}^{1/\ell}} = \min\limits_{1\leq k\leq M} R_k.
	\end{aligned}
	\]
	
	To (2):
	Let $Z_0\in\A$ be a point such that $\norm{Z_0}<R$ and choose $\norm{Z_0}<r_1<r_2<R$.
	Since $1/r_2>1/R$, there exists $L\in\NN$ such that $\forall \ell\geq L: \norm{A(\ell)}<1/r_2^\ell$. 
	Thus $\forall \ell\geq L: \norm{A(\ell) Z^\ell}\leq \norm{A(\ell)}\norm{Z}^\ell<q\coloneqq (r_1/r_2)^\ell$, hence $\sum_{\ell=0}^\infty\norm{A(\ell) Z^\ell}<\infty$ on the open $\{\norm{Z}<r_1\}\ni Z_0$, majorized by the geometric series in $q$.
	
	To (3):
	For example take $Z_0\coloneqq r$ and apply the root test. 
	Notice that, unless $\A=\CC$, $f$ does not diverge for all $Z_0$ with $\norm{Z_0}=r$ because there exist nilpotent elements of arbitrary norm.
	
	To (4):
	This follows from (2) since clearly any $Z\in\BB'_R(0)$ has an open neighbhourhood, on which $f$ converges normally.
	
	To (5):
	By \cref{normexist}, for $Z_0\in\A$ there exists a norm $\norm{\cdot}$ with $\norm{Z_0}<R$ if and only if $\rho_\A(Z_0)<R$.
	Indeed, if $\norm{Z_0}<R$, then $\rho_\A(Z_0)\leq\norm{Z_0}<R$. 
	Conversely, if $\rho_\A(Z_0)<R$ and $\vareps>0$ with $\rho_\A(Z_0)+\vareps<R$, then there exists an algebra norm $\norm{\cdot}$, depending on $Z_0$ and $\vareps$, with $\norm{Z_0}\leq\rho_\A(Z_0)+\vareps<R$.
	
	To (6): 
	follows from the triangular form of $(\cA,\fm)$.
	
	To (7): 
	In accordance with \cref{nicebasis} choose a filtering $\CC$-basis $\{e_1\coloneqq 1_\A, e_2,\dots,e_n\}$ of $\A$ with corresponding structure constants $(\gamma^i_{jk})_{1\leq i,j,k\leq n}$ such that $\forall 1\leq i,k\leq n: \gamma^i_{1k} = \gamma^i_{k1} = \delta^i_k$ and $\forall j,k\geq 2\ \forall i\leq\max\{j,k\}:\gamma^i_{jk} = \gamma^i_{kj} = 0$.
	Write $A(\ell) = \sum_{i=1}^n a_i(\ell) e_i$ for $a_i(\ell)\in\CC$, $\ell\in\NN$.
	Endow $\A$ with the maximum \emph{vector} norm $\norm{\cdot}_\infty$ and write $Z=z\oplus X\in\A$ for $z\in\CC$, $X\in\m$.
	Put $\alpha_i\coloneqq\limsup_{\ell\to\infty}\abs{a_i(\ell)}^{1/\ell}$, $\alpha\coloneqq\max_{1\leq i\leq n}\alpha_i$, and $r_i\coloneqq 1/\alpha_i$, $1\leq i\leq n$.
	By the same argument as in (1) we obtain that $R = \min_{1\leq i\leq n} r_i$, or equivalently, $R=1/\alpha$.
	Now assume that $Z\notin\BB^\sp_\A(0,R)$, i.e. $\rho_\A(Z)=\abs{z}>R=1/\alpha\geq 0$.
	We are going to show that $\limsup_{\ell\to\infty}\norm{A(\ell) Z^\ell}_\infty^{1/\ell}>1$.
	We have:
	\[
	\begin{aligned}
	\norm{A(\ell) Z^\ell}_\infty^{1/\ell} &= \norm{A(\ell) (z+X)^\ell}_\infty^{1/\ell} = \abs{z}\norm{A(\ell)\left(1+\frac{X}{z}\right)^\ell}_\infty^{1/\ell} = 
	\abs{z}\Bigg\lVert{A(\ell)\underbrace{\Bigg(1+\sum_{k=1}^\ell\binom{\ell}{k}\bigg(\frac{X}{z}\bigg)^k\Bigg)}_{\eqqcolon u(\ell)}}\Bigg\rVert_\infty^{1/\ell} =
	\abs{z}\norm{A(\ell)u(\ell)}_\infty^{1/\ell},
	\end{aligned}
	\]
	where $\forall \ell\in\NN: u(\ell)\in\A^\times$.
	Write $u(\ell) = \sum_{i=1}^n u_i(\ell) e_i$, where $u_1(\ell) = 1$.
	Moreover, if $\nu\in\NN$ is the smallest integer with $X^\nu\neq 0$, then $\forall \ell>\nu$:
	\[
	u(\ell) = 1 + \sum_{k=1}^\nu\binom{\ell}{k}\bigg(\frac{X}{z}\bigg)^k = \sum_{i=1}^n \Bigg(\sum_{k=1}^\nu c_{ik}\binom{\ell}{k}\Bigg) e_i = 
	\sum_{i=1}^n u_i(\ell) e_i
	\]
	for some constants $c_{ik}\in\CC$, $1\leq i\leq n$, $1\leq k\leq\nu$.
	Since $\binom{\ell}{k}=O(\ell^k)$ as $\ell\to\infty$ for fixed $1\leq k\leq \ell$, we have in any case $\forall 1\leq i\leq n: u_i(\ell) = O(\ell^{\nu+1})$ as $\ell\to\infty$.
	Carrying out the multiplication with the filtering basis, one obtains
	\[
	\begin{aligned}
	u(\ell)A(\ell) &= \sum_{i=1}^n e_i \Bigg( \sum_{j,k=1}^n \gamma^i_{jk} u_j(\ell) a_k(\ell) \Bigg) =
	a_1(\ell) + \sum_{i=2}^n e_i \Bigg(u_i(\ell) a_1(\ell) + \sum_{k=2}^{i-1} \bigg( \sum_{j=2}^{i-1} \gamma^i_{jk} u_j(\ell) \bigg) a_k(\ell) + a_i(\ell) \Bigg).
	\end{aligned}
	\]
	Now put $\alpha'_1 \coloneqq \limsup_{\ell\to\infty} \abs{a_1(\ell)}^{1/\ell} = \alpha_1$ and
	\[
	\forall 2\leq i\leq n: 
	\alpha'_i \coloneqq \limsup_{\ell\to\infty} \Bigg\lvert u_i(\ell) a_1(\ell) + \sum_{k=2}^{i-1} \bigg( \sum_{j=2}^{i-1} \gamma^i_{jk} u_j(\ell) \bigg) a_k(\ell) + a_i(\ell)\Bigg\rvert^{1/\ell}.
	\]
	We need to show that
	\[
	\limsup_{\ell\to\infty}\norm{A(\ell)u(\ell)}_\infty^{1/\ell} \equiv \max\limits_{1\leq i\leq n} \alpha'_i \geq \alpha \defeq \max\limits_{1\leq i\leq n} \alpha_i.
	\]
	Assume to the contrary that $\alpha'_1,\dots,\alpha'_n < \alpha$ and let $i_0$ be the smallest integer such that $\alpha=\alpha_{i_0}$.
	If $i_0=1$, then consider $\sigma(f(Z)) = \sum_{\ell=0}^{\infty} a_1(\ell) z^\ell$, which has analytic radius $r_1 \defeq 1/\alpha_1 = 1/\alpha = R$.
	But now we have $\abs{z}>R = r_1 \Rightarrow \sigma(f(Z)) = \infty \Rightarrow f(Z) = \infty$.
	If $i_0\geq 2$, then consider
	\[
	\begin{aligned}
	\alpha'_{i_0} &\defeq 
	\limsup_{\ell\to\infty} \Bigg\lvert u_{i_0}(\ell) a_1(\ell) + 
	\sum_{k=2}^{i_0-1} \underbrace{\bigg( \sum_{j=2}^{i_0-1} \gamma^{i_0}_{jk} u_j(\ell) \bigg)}_{\eqqcolon v^{i_0}_k(\ell)} a_k(\ell) + a_{i_0}(\ell) \Bigg\rvert^{1/\ell} =
	\limsup_{\ell\to\infty} \Bigg\lvert u_{i_0}(\ell) a_1(\ell) + \sum_{k=2}^{i_0-1} v^{i_0}_k(\ell) a_k(\ell) + a_{i_0}(\ell) \Bigg\rvert^{1/\ell}.
	\end{aligned}
	\]
	Since $u_i(\ell) = O(\ell^{\nu+1})$ as $\ell\to\infty$, we also have $v^{i_0}_k(\ell) = O(\ell^{\nu+1})$ as $\ell\to\infty$, hence $\limsup_{\ell\to\infty} \abs{u_{i_0}(\ell) a_1(\ell)}^{1/\ell} \leq \alpha_1$ and $\forall 2\leq k\leq i_0-1: \limsup_{\ell\to\infty} \abs{v^{i_0}_k(\ell) a_k(\ell)}^{1/\ell} \leq \alpha_k$.
	But by choice of $i_0$ we have $\alpha_1,\dots,\alpha_{i_0-1}<\alpha_{i_0}$, hence by \cref{limsup} it follows that $\alpha'_{i_0} = \alpha_{i_0} = \alpha$, a contradiction to our assumption.
	Thus
	\[
	\limsup_{\ell\to\infty}\norm{A_\ell Z^\ell}_\infty^{1/\ell} = \abs{z} \limsup_{\ell\to\infty} \norm{A(\ell)u(\ell)}_\infty^{1/\ell} > \frac{1}{\alpha}\alpha = 1
	\]
	as desired.
	
	
	To (8):
	For any submultiplicative $\norm{\cdot}$ and any $\ell\in\NN_0$ we have $\rho_\A(A(\ell))\leq\norm{A(\ell)}$, hence $D^\sp\geq R$.
	
	To (9): 
	If $\cA\xrightarrow{\varphi}\cB$ is a general morphism of $\CC$-algebras, then by \cref{MorFact} we have a canonical factorization $\varphi = \bar{\varphi} \circ \Pi_\tau$, where $\bar{\varphi} = \oplus_{\ell=1}^N \bar{\varphi}_\ell:\bigoplus_{\ell=1}^N (\A_{\tau(\ell)},\m_{\tau(\ell)}) \to \bigoplus_{\ell=1}^N (\B_\ell,\n_\ell)$.
	Since each $\bar{\varphi}_\ell$ is $\CC$-linear and local, we have
	\[
	\bar{\varphi}^{-1}\Big(\prod_{\ell=1}^N\DD_{R_\ell}(\sigma_\ell(Z_0))\times\n_\ell\Big) = \prod_{\ell=1}^N \DD_{R_{\tau(\ell)}}(\sigma_{\tau(\ell)}(Z_0))\times\m_{\tau(\ell)},
	\]
	and furthermore taking $\Pi_\tau^{-1}$ gives precisely $\BB^\sp(Z_0,\underline{R})$ for some $\underline{R}\in (0,\infty]^M$, where notice that $\BB^\sp_{\cA_k}(\pr_k(Z_0),\infty) \defeq \cA_k$.
	Finally, locally normal convergence of $g=f\circ\varphi\in\B\{\varphi(Z)\}$ for $f\in\B\{W\}$ follows from the locally normal convergence of $f$ and continuity of the morphism $\varphi$.
\end{proof}

\begin{remarks}\ 
	\begin{enumerate}[topsep=-\parskip]
		\item The product spectral topology on $\A$ is the one generated by the open spectral polycylinders $\BB^\sp(Z_0,\underline{R})$, $Z_0\in\A$, $\underline{R}\in (0,\infty]^M$, as a subbasis of open sets.
		
		\item Given $f\in\O_\varphi(U)$, we shall also sometimes write for short $\BB^\sp_f(Z_0) \coloneqq \BB^\sp(Z_0,\underline{R}_f(Z_0))$ to mean the (maximal) open spectral polycylinder of convergence of $f$ at the point $Z_0\in U$.
		
		\item \cref{convergencepolycylinder} (7) is specific to connected finite-dimensional commutative associative Banach algebras.
		If we drop commutativity and/or finite dimensionality, then we can have convergence outside of the spectral ball, for example whenever there exist elements $A,B\in\A$ with $\rho(A),\rho(B)>0$ but $AB=0$.
		Existence of such elements is not possible in finite-dimensional local commutative $\CC$-algebras by the triangular structure.
		For an explicit example consider the series
		\[
		f(Z) \coloneqq \sum_{\ell=0}^\infty 
		\begin{pmatrix*}[l]
		0 & 2^{-\ell}\\
		0 & 0
		\end{pmatrix*}
		Z^\ell,\ Z\in\rM_2(\CC).
		\]
		Again all norms are equivalent and we have $R_f=2$, but for $Z_0\coloneqq\begin{psmallmatrix*}[c]
		3 & 0\\
		0 & 0
		\end{psmallmatrix*}$ we have $\rho_\A(Z_0) = 3>R$ and $f(Z_0) = 0$.
		
		\item Since $D^\sp$ depends only on the behaviour of the first (unital) coordinate of $A(\ell)$, it is generally easier to compute than $R$, but it can be an arbitrarily suboptimal estimate.
		
		\item Notice that the well-known estimate $\abs{z_1+\dots+z_n}^{1/\ell} \leq \max_{1\leq i\leq n} \abs{z_i}^{1/\ell}$ for $z_1,\dots,z_n\in\CC$ and $\ell\in\NN$ does not help with the proof of (7) in \cref{convergencepolycylinder} as it gives the ``wrong'' direction of the estimate.
	\end{enumerate}
\end{remarks}

\begin{definition}[Banach-space-valued measures]
	Let $(X,\Sigma)$ be a measurable space and $E$ a Banach space. 
	An $E$-valued measure on $(X,\Sigma)$ is a map $\mu:\Sigma\to E$ such that
	\begin{enumerate}[(i)]
		\item $\mu(\emptyset) = 0$;
		\item $\mu(\bigcupdot_{k\in\NN}X_k) = \sum_{k\in\NN}\mu(X_k)$ for any countable disjoint family $(X_k)_{k\in\NN}\subseteq\Sigma$.
	\end{enumerate}
\end{definition}

\begin{remarks}\ 
	\begin{enumerate}[topsep=-\parskip]
		\item The only difference to the usual notion of a measure is the lack of non-negativity, which per se does not make sense in an arbitrary Banach space $E$.
		Moreover, the order of summation does not matter since the order of set-theoretic union does not.
		
		\item If $E=\A\in\fdCAlg_\CC$ and $\{a_1,\dots,a_n\}$ is a $\CC$-basis for $\A$, then a measure $\mu:\Sigma\to\A$ can be uniquely written as $\mu = \sum_{k=1}^n a_k \mu_k$, where $\mu_k:\Sigma\to\CC$, $1\leq k\leq n$, are complex-valued measures (in the usual sense of functional analysis).
	\end{enumerate}
\end{remarks}

Let $\cB=\bigoplus_{\ell=1}^N\cB_\ell$ be the decomposition of $\B$ into Artinian local $\CC$-algebras $(\cB_\ell,\fn_\ell,\norm{\cdot}_{\B_\ell})$, $1\leq \ell\leq N$, with unital submultiplicative norms and let $\norm{\cdot} \coloneqq \norm{\cdot}_\oplus=\max_{1\leq \ell\leq N}\norm{\cdot}_{\B_\ell}$ be the direct sum norm on $\B$.
For $Z\in\A=\bigoplus_{k=1}^M(\A_k,\m_k)$ we shall write $Z=\bigoplus_{k=1}^M (z_k\oplus Z_{\m_k})$, where $z_k\in\CC$ and $Z_{\m_k}\in\m_k$, $1\leq k\leq M$.

\begin{proposition}[Cauchy-Transform over $\A$]\label{cauchytrans}
	Let $\varphi = (\oplus_{\ell=1}^N \bar{\varphi}_\ell)\circ \Pi_\tau: \A=\bigoplus_{k=1}^M (\A_k,\m_k) \to \B = \bigoplus_{\ell=1}^N (\B_\ell,\n_\ell)$ be a morphism of $\CC$-algebras given in canonical factorization form.
	Let $\Omega$ be a measurable space, $\mu$ a finite $\B$-valued measure on $\Omega$, $G\coloneqq\bigoplus_{k=1}^M (g_k\oplus G_{\m_k}):\Omega\to\A$ a measurable function, and $U\subseteq\A$ open such that $\forall Z\in U\ \forall 1\leq \ell\leq N$:
	\begin{equation}
	\inf\limits_{t\in\Omega} \left(\abs{g_{\tau(\ell)}(t)-z_{\tau(\ell)}}-\norm{\bar{\varphi}_\ell(G_{\m_{\tau(\ell)}}(t)-Z_{\m_{\tau(\ell)}})}_{\cB_\ell}\right)>0.
	\end{equation}
	In particular, $\varphi(G(\Omega)-U) \subseteq \cB^\times$.
	Then $f:U\to\B$ defined by
	\[
	f(Z)\coloneqq\int_\Omega\frac{\d\mu(t)}{\varphi(G(t)-Z)}
	\]
	is $\varphi$-analytic on $U$ with derivatives
	\[
	\frac{f^{(k)}(Z)}{k!} = \int_\Omega \frac{\d\mu(t)}{\varphi(G(t)-Z)^{k+1}},\ k\in\NN_0.
	\]
\end{proposition}

\begin{proof}
	Fix $Z\in U$ and set
	\[
	R \coloneqq \min\limits_{1\leq \ell\leq N} \inf\limits_{t\in\Omega} \left(\abs{g_{\tau(\ell)}(t)-z_{\tau(\ell)}} - \norm{\bar{\varphi}_\ell(G_{\m_{\tau(\ell)}}(t)-Z_{\m_{\tau(\ell)}})}_{\cB_\ell}\right)>0.
	\]
	By continuity of $\varphi$ choose $r>0$ small enough such that $\BB^{\norm{\cdot}}_\A(Z;r)\subseteq U$ and $\varphi(\BB^{\norm{\cdot}}_\A(Z;r)) \subseteq \BB^{\norm{\cdot}}_\B(\varphi(Z);R)$.  
	By hypothesis $\forall 1\leq \ell\leq N\ \forall t\in\Omega$:
	\[
	\norm{\frac{\bar{\varphi}_\ell(G_{\m_{\tau(\ell)}}(t)-Z_{\m_{\tau(\ell)}})}{g_{\tau(\ell)}(t)-z_{\tau(\ell)}}}_{\cB_\ell}<1,
	\]
	hence $\forall t\in\Omega$:
	\[
	\begin{aligned}
	&\norm{\frac{1}{\varphi(G(t)-Z)}}_\cB =
	\norm{\bigoplus_{\ell=1}^N\frac{1}{(g_{\tau(\ell)}(t)-z_{\tau(\ell)})+\bar{\varphi}_\ell(G_{\m_{\tau(\ell)}}(t)-Z_{\m_{\tau(\ell)}})}}_\oplus \defeq\\
	&\defeq \max\limits_{1\leq \ell\leq N} \norm{\frac{1}{(g_{\tau(\ell)}(t)-z_{\tau(\ell)})+\bar{\varphi}_\ell(G_{\m_{\tau(\ell)}}(t)-Z_{\m_{\tau(\ell)}})}}_{\cB_\ell} =
	\max\limits_{1\leq \ell\leq N}\frac{1}{\abs{g_{\tau(\ell)}(t)-z_{\tau(\ell)}}} 
	\norm{\frac{1}{1-\frac{\bar{\varphi}_\ell(Z_{\m_{\tau(\ell)}}-G_{\m_{\tau(\ell)}}(t))}{g_{\tau(\ell)}(t)-z_{\tau(\ell)}}}}_{\cB_\ell} \\
	&\leq \max\limits_{1\leq \ell\leq N} \frac{1}{\abs{g_{\tau(\ell)}(t)-z_{\tau(\ell)}}}
	\frac{1}{1-\norm{\frac{\bar{\varphi}_\ell(G_{\m_{\tau(\ell)}}(t)-Z_{\m_{\tau(\ell)}})}{g_{\tau(\ell)}(t)-z_{\tau(\ell)}}}_{\cB_\ell}} =
	\max\limits_{1\leq \ell\leq N} \frac{1}{\abs{g_{\tau(\ell)}(t)-z_{\tau(\ell)}}-\norm{\bar{\varphi}_\ell(G_{\m_{\tau(\ell)}}(t)-Z_{\m_{\tau(\ell)}})}_{\cB_\ell}} 
	\leq\frac{1}{R}
	\end{aligned}
	\]
	by \cref{norminverse}.
	Thus $\forall W\in\BB^{\norm{\cdot}}_\A(Z;r)\ \forall t\in\Omega$:
	\[
	q \coloneqq \norm{\frac{\varphi(W-Z)}{\varphi(G(t)-Z)}} \leq \norm{\varphi(W-Z)}\norm{\frac{1}{\varphi(G(t)-Z)}} \leq \frac{\norm{\varphi(W-Z)}}{R} < 1.
	\]
	Hence for all $W\in\BB^{\norm{\cdot}}_\A(Z;r)$ the geometric series
	\[
	\begin{aligned}
	\frac{1}{\varphi(G(t)-W)} = \frac{1}{\varphi(G(t)-Z)}\frac{1}{1-\frac{\varphi(W-Z)}{\varphi(G(t)-Z)}} = 
	\frac{1}{\varphi(G(t)-Z)} \sum_{\ell=0}^\infty \left(\frac{\varphi(W-Z)}{\varphi(G(t)-Z)}\right)^\ell
	\end{aligned}
	\]
	converges uniformly in $t\in\Omega$. 
	Therefore, since $\mu$ is finite, we can interchange summation and integration to obtain $\forall W\in\BB^{\norm{\cdot}}_\A(Z;r)$:
	\[
	\begin{aligned}
	f(W) &= \int_\Omega \frac{\d\mu(t)}{\varphi(G(t)-W)} = \int_\Omega \left(\sum_{\ell=0}^\infty \frac{\varphi(W-Z)^\ell}{\varphi(G(t)-Z)^{\ell+1}}\right) \d\mu(t) =
	\sum_{\ell=0}^\infty \left(\int_\Omega \frac{\d\mu(t)}{\varphi(G(t)-Z)^{\ell+1}}\right) \varphi(W-Z)^\ell
	\end{aligned}
	\]
	as desired.
\end{proof}

\begin{remark}
	Using \cref{norminverselocaluniform} one can  relax the hypothesis of \cref{cauchytrans} as follows:
	\[
	\forall Z\in U\ \forall 1\leq\ell\leq N\ \forall t\in\Omega:\abs{g_{\tau(\ell)}(t)-z_{\tau(\ell)}} \geq
	\norm{\bar{\varphi}_\ell(G_{\m_{\tau(\ell)}}(t)-Z_{\m_{\tau(\ell)}})}_{\cB_\ell},
	\]
	but keep
	\[
	\inf\limits_{t\in\Omega} \abs{g_{\tau(\ell)}(t)-z_{\tau(\ell)}} > 0.
	\]
	Indeed, if we fix $Z\in U$ and choose $R>0$ small enough such that $\abs{g_{\tau(\ell)}(t)-z_{\tau(\ell)}} \geq \nu_\ell R$, where $\nu_\ell\in\NN$ is the smallest integer such that $\forall t\in\Omega: \bar{\varphi}_\ell(G_{\m_{\tau(\ell)}}(t)-Z_{\m_{\tau(\ell)}})^{\nu_\ell}=0$, then by \cref{norminverselocaluniform} we obtain
	\[
	\norm{\frac{1}{\bar{\varphi}_\ell(G_{\tau(\ell)}(t)-Z_{\tau(\ell)})}}_{\cB_\ell} \leq \frac{\nu_\ell}{\abs{g_{\tau(\ell)}(t)-z_{\tau(\ell)}}} \leq \frac{1}{R},
	\]
	from where one proceeds analogously as in the proof above. \qed
\end{remark}

\begin{proposition}[Analyticity \& Cauchy's Integral Formula over $\A$ for Derivatives]\label{analyt}
	Let $U\subseteq\A$ be open and path-connected and let $f\in\O_\varphi(U)$.
	Then:
	\begin{enumerate}
		\item For every open polydisk $\Delta\coloneqq\Delta_r(Z_0)\subseteq U$ with center $Z_0\in U$ and every $Z_0$-admissible loop $\gamma\in\Cc^1_\pw(\SS^1,\Delta)$ in it there exists an open neighbourhood $V\ni Z_0$ in $\Delta$ such that $\forall Z\in V$:
		\begin{equation}\label{cifder}
		f^{(k)}(Z) \Ind_\varphi(\gamma,Z) = \frac{k!}{2\pi i} \oint_\gamma \frac{f(W)}{\varphi(W-Z)^{k+1}}\d W,\ k\in\NN_0.
		\end{equation}
		
		\item $f$ is $\varphi$-analytic in $U$.
	\end{enumerate}
\end{proposition}

\begin{proof}
	To (1): Let $\Delta\coloneqq\Delta_r(Z_0)\subseteq U$ be an arbitrary open polydisk in $U$ with center $Z_0\in U$. 
	Fix a $Z_0$-admissible loop $\gamma\in\Cc^1_\pw(\SS^1,\Delta)\neq\emptyset$ and put $\fI_\gamma(Z)\coloneqq\Ind_\varphi(\gamma,Z)$. 
	Define $V\coloneqq\Delta\cap\Adm_\varphi(\gamma)\ni Z_0$, which is open. 
	Then by \cref{cif} we obtain $\forall Z\in V$:
	\[
	f(Z)\fI_\gamma(Z) = \frac{1}{2\pi i} \oint_\gamma \frac{f(W)}{\varphi(W-Z)}\d W.
	\]
	Differentiating both sides gives us
	\[
	f'(Z)\fI_\gamma(Z) + f(Z)\fI_\gamma'(Z) = f'(Z)\fI_\gamma(Z) = \frac{1}{2\pi i}\oint_\gamma\frac{f(W)}{\varphi(W-Z)^2}\d W,
	\]
	since $\fI_\gamma:\Adm_\varphi(\gamma)\to\ZZ^{\oplus N}\subseteq\cB$ is a locally constant map.	
	Now the formula for the higher derivatives $f^{(k)}$, $k\in\NN$, follows in the same way by induction.
	
	To (2):
	Without loss of generality we can assume $Z_0=0$.
	Following the argument in the the proof of \cref{cif} we know constructively that there exists a sufficiently small $\vareps>0$ such that the loop $\gamma \coloneqq \gamma_\vareps$, defined by
	\[
	\gamma_\vareps(t) \coloneqq \bigoplus_{k=1}^M \vareps e^{2\pi i t},\ t\in [0,1],
	\]
	is contained in $\Delta$ and $Z_0$-admissible with $\fI_{\gamma}(Z_0)=1_\A$ (in fact, only the indices $k\in\im \tau$ matter).
	Denote by $C\ni Z_0$ the connected open component of $Z_0$, on which $\fI_\gamma$ stays $=1_\A$.
	Then we have $\forall Z\in V\cap C$:
	\[
	f(Z) = \frac{1}{2\pi i\fI_\gamma(Z)} \oint_\gamma \frac{f(W)}{\varphi(W-Z)}\d W = \frac{1}{2\pi i} \oint_\gamma \frac{f(W)}{\varphi(W-Z)}\d W.
	\]
	One immediately checks that $\Omega\coloneqq I\defeq [0,1]$, $G\coloneqq\gamma$, and $\d\mu(t)\coloneqq f(\gamma(t))\varphi(\gamma'(t))\d t$ satisfy the hypothesis of \cref{cauchytrans}.
	In particular, the ``$\inf$-condition'' presents no problem because $\Omega$ is compact and $\gamma$ is continuous.
	Now the claim follows since $\fI_\gamma|_{V\cap C}=\const=1_\A$.
\end{proof}

\begin{remark}
	It is seductive to conjecture at this point that if $\fA$ is a connected non-unital $\CC$-algebra, then all $\fA$-holomorphic functions at $0$ are given by $\fA\{Z\} = \fA[Z]$, and hence that $\fA[Z]$ contains in fact all $\fA$-holomorphic functions.
	However, this is not true.
	Consider $\fA \coloneqq \fM$ being the maximal ideal of $\CC[X]/(X^3)$, i.e. $\fA = \CC e_1 \oplus \CC e_2$ with $e_1^2 = e_2$, $e_1 e_2 = 0$, $e_2^2 = 0$. 
	Let $f\in \CC\{z_1\} \setminus \CC[z_1]$ and define $F(z_1 e_1 + z_2 e_2) \coloneqq e_1 + f(z_1) e_2$. 
	Then $F$ is not polynomial, but it is $\fA$-holomorphic:
	\[
	\begin{aligned}
	F(Z+H) &= e_1 + f(z_1 + h_1) e_2 = e_1 + (f(z_1) + f'(z_1) h_1 + r(h_1)) e_2 =  e_1 + f(z_1)e_2 + f'(z_1)h_1 e_2 + r(h_1) e_2 = \\
	&= F(Z) + (f'(z_1) e_1 + w e_2)(h_1 e_1 + h_2 e_2) + r(h_1) e_2,
	\end{aligned}
	\]
	where $w\in\CC$ is arbitrary and $\abs{r(h_1)} = o(\abs{h_1})$.
	But $\norm{r(h_1) e_2}_1 = o(\norm{H}_1)$ because:
	\[
	\frac{\norm{r(h_1) e_2}_1}{\norm{h_1 e_1 + h_2 e_2}_1} = \frac{\abs{r(h_1)}}{\abs{h_1} + \abs{h_2}} \leq \frac{\abs{r(h_1)}}{\abs{h_1}} \to 0
	\]
	as $H \to 0$.
	Notice once again the non-uniqueness of the derivative $F'$ for non-unital algebras. \qed
\end{remark}

\begin{corollary}[Cauchy's Inequality over $\A$]\label{cauchyineq}
	Let $U\subseteq\A$ be open and path-connected, $Z_0\in U$, $f\in\O_\varphi(U)$, and $\gamma\in\Cc^1_\pw(\SS^1,U)$ an admissible loop for $Z_0$ with $\Ind_\A(\gamma,Z_0)=1$. Then $\forall k\in\NN_0:$
	\begin{equation}
	\norm{f^{(k)}(Z_0)}_\B \leq \frac{k!}{2\pi}\norm{f}_{\B,\gamma} L_\B(\gamma)
	\sup\limits_{t\in I}\norm{\frac{1}{\varphi(\gamma(t))-\varphi(Z_0)}}_\B^{k+1}.
	\end{equation}
	In particular, we have the following two special cases:
	\begin{enumerate}
		\item If $\gamma(t)\coloneqq Z_0 + r e^{2\pi i t}$, $t\in I$, then
		\begin{equation}\label{meanvalueineq0}
		\norm{f^{(k)}(Z_0)}_\B \leq \frac{k!}{r^k} \norm{f}_{\B,\gamma}.
		\end{equation}
		
		\item If $(\A,\m)\xrightarrow{\varphi}(\B,\n)$ and $\norm{\cdot}_\B$ is unital, then for $X_0\in\m$ and $\nu\in\NN$ with $\varphi(X_0)^{\nu-1}\neq 0$ and $\varphi(X_0)^\nu=0$ in $\n$ and $\gamma(t)\coloneqq(\lambda_0+re^{2\pi i t})$, $t\in I$, with $\lambda_0\in\CC$ we have $\forall k\in\NN_0$:		
		\begin{equation}\label{meanvalueineq}
		\norm{f^{(k)}(\lambda_0+X_0)}_\B \leq k! \norm{f}_{\B,\gamma} \bigg(\sum_{j=0}^{\nu-1}\bigg(\frac{\norm{\varphi(X_0)}_\B}{r}\bigg)^j\bigg)^{k+1}.
		\end{equation}
	\end{enumerate}
\end{corollary}

\begin{proof}
	We have
	\[
	\begin{aligned}
	\norm{f^{(k)}(Z_0)}_\B &= \frac{k!}{2\pi} \norm{\int_0^1\frac{f(\gamma(t))\varphi(\gamma'(t))}{\varphi(\gamma(t)-Z_0)^{k+1}}\d t}_\B
	\leq \frac{k!}{2\pi} \int_0^1 \norm{f(\gamma(t))}_\B \norm{\varphi(\gamma'(t))}_\B \norm{\frac{1}{\varphi(\gamma(t))-\varphi(Z_0)}}_\B^{k+1}\d t,
	\end{aligned}
	\]
	from which the desired inequality follows.
	Now, (1) and (2) follow by direct substitution as $\varphi$ preserves scalars, where in (2) we also apply \cref{norminverselocalformula}.
\end{proof}

\begin{corollary}[Liouville over $\A$]\label{liouville}
	Let $U\subseteq\A$ be open and path-connected and let $f\in\O_\varphi(U)$. 
	Suppose that there exists a subset $V\subseteq U$ such that:
	\begin{enumerate}[(i)]
		\item $f|_V$ is bounded and
		\item for every $Z_0\in U$, $V$ contains a collection of $Z_0$-admissible simple loops $\{\gamma_n\}_{n\in\NN}$ such that
		\[
		\bigg\{\sup\limits_{t\in I} \norm{\frac{1}{\varphi(\gamma_n(t))-\varphi(Z_0)}}_\B: n\in\NN \bigg\}
		\]
		is unbounded, but
		\[
		L_\B(\gamma_n) \sup\limits_{t\in I} \norm{\frac{1}{\varphi(\gamma_n(t))-\varphi(Z_0)}}_\B \xrightarrow{n\to\infty} 0.
		\]
	\end{enumerate}
	Then $f=\const$. In particular, if $(\A,\m)$ is local, $U\supseteq\CC$, and $f\big|_\CC$ is bounded, then $f=\const$.
\end{corollary}

\begin{proof}
	By Cauchy's inequality over $\A$ we have for $k=1$
	\[
	\norm{f'(Z_0)}_\B \leq \frac{1}{2\pi}\norm{f}_{\B,\gamma_n} L_\B(\gamma_n) 
	\sup\limits_{t\in I} \norm{\frac{1}{\varphi(\gamma_n(t))-\varphi(Z_0)}}_\B\to 0
	\] 
	as $n\in\NN$ varies, because $\norm{f}_{\B,\gamma_n}$ is bounded by assumption (no pun intended). 
	Thus $f$ is locally constant and hence constant by path-connectedness of $U$.
	For the last statement it clearly suffices to take (scalar) loops in $\CC$.
\end{proof}

%
%
%
%
%
%

\begin{corollary}[Local form of $\varphi$-holomorphic functions]\label{localform}
	Let $U\subseteq\A$ be open and $\A\xrightarrow{\varphi}\B$ a morphism of $\CC$-algebras. Let $f\in\O_\varphi(U)$. Then:
	\begin{enumerate}
		\item $f$ locally factors analytically through $\varphi$: for every $Z_0\in U$ there exist an open neighbourhood of $V\ni Z_0$ and an open neighborhood $W\ni\varphi(Z_0)$ with $\varphi(V)\subseteq W$ such that $\exists_1 g_W\in\O_\B(W): f|_V=g_W\circ\varphi$, namely:
		\[
		\begin{tikzcd}[column sep=scriptsize, row sep=normal]
		\A \ar[d,"\varphi",swap] \ar[r,hookleftarrow] & \BB^\sp_\A(Z_0,\underline{R}) \ar[d,"\varphi",swap] \ar[r,"f"] & \B\\
		\B \ar[r,hookleftarrow] & \BB^\sp_\B(\varphi(Z_0),\underline{R})  \ar[ur,dashed,"\exists_1 g",swap]
		\end{tikzcd}
		\]
		
		\item If $U\ni 0$, then $f$ extends polynomially to $\nil\A$.
	\end{enumerate}
\end{corollary}

\begin{proof}
	To (1): This is immediate because locally at $Z_0$ we have $f\in\B\{\varphi(Z-Z_0)\}$.
	
	To (2): Explicitly, if $f$ is analytic at 0, then in a small open neighbourhood of 0 we have $f(Z)=\sum_{j=0}^\infty B_j Z^j$ for some $B_j\in\B$. 
	If $X\in\nil\A$, then simply $f(X)\coloneqq \sum_{j=0}^\infty B_j X^j$, which is polynomial as $X$ is nilpotent.	
\end{proof}

\begin{corollary}
	Let $(\A,\m)\xrightarrow{\varphi}(\B,\n)$ be a morphism of local $\CC$-algebras, $U\subseteq\A$ open with $U\ni 0$, and $f\in\O_\varphi(U)$.
	If $f(0)\in\n$, then $f(\m)\subseteq\n$.
\end{corollary}

\begin{proof}
	In a spectral neighbourhood of $0$ we have $f(Z)=\sum_{k=0}^\infty B_k \varphi(Z)^k$, where $B_0 = f(0)\in\n$ by hypothesis.
	Since $\varphi$ is automatically local for Artinian $\A$ and $\B$, i.e. $\varphi(\m)\subseteq\varphi(\n)$, the claim follows as $\n$ is an ideal.
\end{proof}

\begin{proposition}[Canonical form of $\varphi$-holomorphic functions]\label{Canform}
	Let $\A\xrightarrow{\varphi}\B$ be a morphism, let $U\subseteq\A$ be open and $f\in\O_\varphi(U)$.	
	\begin{enumerate}
		\item Separation of the scalar variable from the nilpotent variable: Suppose that $\A=(\A,\m)$ and $\B=(\B,\n)$ are local and let $\nu \coloneqq \rh(\varphi)$ be the height of $\varphi$.
		If $z\oplus X\in U$ with $z\in\CC$ and $X\in\m$, then
		\begin{equation}\label{canform}
			f(Z) = f(z\oplus X) = \sum_{k=0}^{\nu-1}\frac{f^{(k)}(z)}{k!}\varphi(X)^k,
		\end{equation}
		where notice that $f^{(k)}|\sigma_\cA(U)\to\B$ simply are $\CC$-holomorphic functions with values in the finite-dimensional Banach space $\B$.
		
		\item If $\A=(\A,\m)$ and $\B=(\B,\n)$ are local, then $f$ extends $\varphi$-holomorphically to $\wtilde{U}$. 
		In other words, $\O_\varphi(U)\cong\O_\varphi(\wtilde{U})$ canonically.
		In particular, if $U\supseteq\CC$, then $f$ extends to a $\varphi$-entire function.
		
		\item Conversely, if $(\A,\m)\xrightarrow{\varphi}(\B,\n)$ is a morphism of local $\CC$-algebras with $\nu\coloneqq \rh(\varphi)$, $V\subseteq\CC$ open, and $g:V\to\B$ a $\CC$-holomorphic function with values in $\B$, then $\tilde{g}:V\times\m\to\B$,
		\[
		\tilde{g}(z\oplus X)\coloneqq\sum_{k=0}^{\nu-1}\frac{g^{(k)}(z)}{k!}\varphi(X)^k
		\]
		defines a $\varphi$-holomorphic function.
		
		\item $f$ always extends $\varphi$-holomorphically to $\wtilde{U}$.
		In other words, $\O_\varphi(U)\cong\O_\varphi(\wtilde{U})$ canonically.
	\end{enumerate}
\end{proposition}

\begin{proof}
	To (1) and (2):
	\cref{localcylindercont} ensures that locally $f$ extends to some sufficiently small open cylinder $V=\wtilde{V}$ with $\sigma_\cA(V) \subseteq \sigma_\cA(U) \subseteq \CC$.
	Since $f$ is analytic in $V$, Taylor expansion at some fixed $Z_0\coloneqq z\oplus 0$, $z\in\sigma_\cA(V)$, gives \cref{canform} valid in some open spectral cylindrical neighbourhood of $z$.
	But this is a globally defined expression on $\sigma_\cA(U)\times\m$ that agrees on all spectral cylindrical open neighbourhoods inside $\wtilde{U}$.
	Notice that any open neighbourhood of $z$ contains in particular a basis of $\m$, so the choice of $\nu$ is necessary to ensure that for all $z\oplus X$ in said neighbourhood we have $\varphi(X)^\nu = \varphi(X^\nu) = 0$.
	
	To (3): 
	Write $z^1\coloneqq z$ and $X\eqqcolon z^2 a_2+\dots+z^n a_n$ for a $\CC$-basis $\{a_2,\dots,a_n\}$ of $\m$.
	By \cref{delop} we only need to check that $\forall 2\leq j\leq n$:
	\[
	0 = \d_j \tilde{g} \equiv \bigg(-a_j\pdv{}{z^1} + \pdv{}{z^j}\bigg)\tilde{g}
	\]
	Indeed we have
	\[
	\begin{aligned}
	&-a_j \pdv{}{z^1}\tilde{g} + \pdv{}{z^j} \tilde{g} = 
	-a_j \sum_{k=0}^{\nu-1}\frac{g^{(k+1)}(z^1)}{k!}\varphi(X)^k + \sum_{k=1}^{\nu-1}\frac{g^{(k)}(z^1)}{k!}\varphi(X)^{k-1} k a_j = \\
	&-a_j \sum_{k=0}^{\nu-1}\frac{g^{(k+1)}(z^1)}{k!}\varphi(X)^k + a_j \sum_{k=0}^{\nu-2}\frac{g^{(k+1)}(z^1)}{k!}\varphi(X)^k \defeq 
	-\varphi(a_j) \varphi(X)^{\nu-1} = 0,
	\end{aligned} 
	\]
	since $\nu \defeq \rh(\varphi)$.
	
	To (4): If $f\in\cO_\varphi(U)$, then $f=\bar{f}\circ\Pi_\tau$, where $\bar{f}$ is of the form $\bar{f}=\oplus_{\ell=1}^N \bar{f}_\ell$ for some $\bar{\varphi}_\ell$-holomorphic functions $\bar{f}_\ell$.
	Now applying (2) to each $\bar{f}_\ell$, $1\leq \ell\leq N$, and taking $\Pi_\tau^{-1}$ yields the statement.
\end{proof}

\begin{corollary}\ 
	\begin{enumerate}[topsep=-\parskip]
		\item We have a bijection between $m$-tuples of $\CC$-holomorphic functions $f^i:V\to\CC$, $V\subseteq\CC$ open, together with a morphism $\A\xrightarrow{\varphi}\B$ of local $\CC$-algebras, where $m=\dim_\CC\B$, on the one hand and $\varphi$-holomorphic functions on $V\times\m$ on the other hand
		\[
		\{(f:V\to\B,\varphi)\}\longleftrightarrow\O_\varphi(V\times\m),
		\]
		where $f\coloneqq f^1 1_\B + f^2 b_2 + \dots f^m b_m$ for a choice of $\CC$-basis $\{1,b_2,\dots,b_m\}$ of $\B$.
		
		
		\item Isolated Zeros in Scalar/Spectral Planes: let $(\cA,\fm) \xrightarrow{\varphi} (\cB,\fn)$ be a morphism of local $\CC$-algebras, $U\subseteq\cA$ open and connected, and $0\neq f\in\cO_\varphi(U)$.
		Then $\forall Z_0\in\cA: \cZ(f)\cap(\CC+Z_0)$ is a countable discrete subset of $U$.
		
		\item Identity Theorem in Scalar/Spectral Planes: let $(\cA,\fm) \xrightarrow{\varphi} (\cB,\fn)$ be a morphism of local $\CC$-algebras, $U\subseteq\cA$ open and connected, and $f,g\in\cO_\varphi(U)$.
		If the incidence set $\cZ(f-g) \cap (\CC+Z_0)$ has an accumulation point in $U$, then $f=g$.
	\end{enumerate}
\end{corollary}

\begin{proof}
	To (1): clear from \cref{Canform}.
	
	
	To (2) \& (3): We only prove (2).
	It follows essentially from \cref{canform} and (1).
	We can assume $U=\wtilde{U}$.
	Since $\sigma_\cA$ is open and continuous as a projection, $\sigma_\cA(U)$ is also open and connected.
	We can assume without loss of generality that $Z_0=0$, otherwise consider $\tilde{f}(Z) \coloneqq f(Z+Z_0)$.
	Suppose that $\cZ(f)\cap\CC$ has an accumulation point in $\sigma_\cA(U)$.
	Then by the Identity Principle of Complex Analysis of a Single Variable $f|_{\sigma_\cA(U)} = 0$.
	But by \cref{canform} we have
	\[
	f(z\oplus X) = \sum_{k=0}^{\nu-1}\frac{f^{(k)}(z)}{k!}\varphi(X)^k,
	\]
	where $\nu = \rh(\varphi)$.
	Thus $f=0$.
\end{proof}

\begin{remarks}\ 
	\begin{enumerate}[topsep=-\parskip]
		\item Thus the local theory of $\varphi$-holomorphic functions is in a sense about doing Complex Analysis of One Variable with commutative nilpotents.
		
		\item Notice that (3) and (4) do not follow by simply applying the canonical projection $\sigma_\cA$ to the power series of $f$ since this only yields $f_1=0$ rather than $f|_\CC = 0$.
		The reader is free to convince herself that the latter does not necessarily follow from the former by means of the generalized Cauchy-Riemann-Sheffers equations for a local $\CC$-algebra $(\cA,\fm)$.
	\end{enumerate}
\end{remarks}

\begin{definition}
	Let $\A=(\A,\m)$ be a local $\CC$-algebra, $\norm{\cdot}$ an arbitrary vector norm on $\A$, and $\vareps>0$.
	Then
	\[
	\m_\vareps\coloneqq\{X\in\m: \norm{X}<\vareps\}.
	\]
\end{definition}

\begin{lemma}[Extension of bi-$\A$-holomorphisms]
	Let $\A=(\A,\m)$ be a local $\CC$-algebra, $U,V\subseteq\A$ open subsets, and $f:U\xrightarrow{\cong} V$ an $\A$-biholomorphism. 
	Then $f$ extends to an $\A$-biholomorphism $\tilde{f}:\wtilde{U}\xrightarrow{\cong}\wtilde{V}$:
	\[
	\begin{tikzcd}[column sep=scriptsize, row sep=scriptsize]
	U \arrow[hook]{d} \arrow{r}{\cong}[swap]{f} & V \arrow[hook]{d}\\
	\wtilde{U} \arrow{r}{\cong}[swap]{\tilde{f}} & \wtilde{V}
	\end{tikzcd}
	\]
\end{lemma}

\begin{proof}
	Let $g\coloneqq f^{-1}:V\to U$, let $\tilde{f}$ be the extension of $f$ to $\wtilde{U}$ and $\tilde{g}$ be the extension of $g$ to $\wtilde{V}$.
	We need to show that $\tilde{f}(\wtilde{U})\subseteq\wtilde{V}$: then likewise $\tilde{g}(\wtilde{V})\subseteq\wtilde{U}$, and since $g\circ f = \id$ on the open $U$, it follows by the identity principle of SCVs that $g\circ f = \id$ on $\wtilde{U}$.	
	Now, if $U=\bigcup_{i\in I} U_i$ is an arbitrary (open) cover, then $\wtilde{U} = \bigcup_{i\in I} \wtilde{U}_i$.	
	Thus it suffices to show the claim for bounded cylindrical neighbourhoods in $U$.
	Without loss of generality we can consider $f(Z)=\sum_{k=1}^\infty A_k Z^k\in\A\{Z\}$ and $g(Z) = \sum_{k=1}^\infty B_k Z^k\in\A\{Z\}$ with $A_1,B_1\in\A^\times$ inducing automorphisms on $\DD_r(0)\times\m_\vareps$ for $r\leq\min\{R_f,R_g\}$ and a sufficiently small $\vareps>0$.
	We need to show that then $f(\DD_r(0)\times\m)\subseteq\DD_r\times\m$.
	But if $f_1$ denotes the unital component of $f$, then $\forall z\oplus X\in\DD_r(0)\times\m: \rho_\A(f(z\oplus X)) = \abs{f_1(z)}<r$ since $f(\DD_r(0)\times\m_\vareps)\subseteq \DD_r(0)\times\m_\vareps$.
\end{proof}

\begin{lemma}[Analyticity and Nilpotents]\label{nilanal} 
	Let $(\cA,\fm) \xrightarrow{\varphi} (\cB,\fn)$ be a morphism of local $\CC$-algebras, $U\subseteq\A$ open and path-connected, and $f\in\O_\varphi(U)$.
	\begin{enumerate}
		\item For $Z_0\in U$ and $X\in\m=\nil\A$ the limit of the generalized derivative
		\[
		f'_{(X)}(Z_0) \coloneqq \lim_{\substack{H\to X \\ H\in\A^\times}} \frac{f(Z_0+H)-f(Z_0)}{\varphi(H)}
		\]
		exists and gives rise to a function $\check{f}:\m\times U\to\B$, $(X,Z)\mapsto f'_{(X)}(Z)$, polynomial in $X$ and $\varphi$-analytic in $Z$.
		
		\item The function $g:U\times U\to\B$,
		\begin{equation}
		\begin{aligned}
		g(Z,W) &\coloneqq
		\begin{cases}
		\frac{f(W)-f(Z)}{\varphi(W-Z)}, &\text{if } W-Z\in\A^\times\cap U\\
		f'_{(W-Z)}(Z), &\text{if } W-Z\in\m\cap U
		\end{cases}
		\defeq
		\begin{cases}
		\frac{f(W)-f(Z)}{\varphi(W-Z)}, &\text{if } W-Z\in\A^\times\cap U\\
		\check{f}(W-Z,Z), &\text{if } W-Z\in\m\cap U
		\end{cases}
		\end{aligned}
		\end{equation}
		is $\varphi$-holomorphic both in $Z$ and $W$.
	\end{enumerate}
\end{lemma}

\begin{proof}
	To (1): First notice that the definition makes sense because $\A^\times$ is dense in $\A$.
	Without loss of generality we can assume $U=\wtilde{U}$ cylindrical. 
	Fix $Z_0\in U$ and let $R\coloneqq R_f(Z_0)$ be the radius of convergence of $f$ at $Z_0$.
	By $\varphi$-analyticity we have $R>0$ and $\forall W\in\BB^\sp_\A(Z_0,R)$:
	\[
	f(W) = \sum_{k=0}^\infty \frac{f^{(k)}(Z_0)}{k!} \varphi(W-Z_0)^k \defeq \sum_{k=0}^\infty\frac{f^{(k)}(Z_0)}{k!}\varphi(H)^k
	\]
	for all $H\in\BB^\sp_\A(0,R)$, where $H\coloneqq W-Z_0$.
	Hence $\forall H\in\BB^\sp_\A(0,R)$:
	\[
	\begin{aligned}
	\frac{f(W)-f(Z_0)}{\varphi(W-Z_0)} &= \frac{f(Z_0+H)-f(Z_0)}{\varphi(H)} = \sum_{k=1}^\infty\frac{f^{(k)}(Z_0)}{k!}\varphi(H)^{k-1} \eqqcolon \check{f}_\loc(H,Z_0) \defeq \sum_{k=0}^\infty\frac{f^{(k+1)}(Z_0)}{(k+1)!}\varphi(W-Z_0)^k \defeq \\
	&\defeq \check{f}_\loc(W-Z_0,Z_0).
	\end{aligned}
	\]
	Since the power series of a $\varphi$-analytic function around a point converge on the biggest open disk cylinder that fits in, there exist sufficiently small open subcylinders $V_1\subseteq U$ for $W$, $V_0\subseteq\BB^\sp_\A(0,R)$, $V_0\ni 0$, for $H$, and $V_2\subseteq U$, $V_2\ni Z_0$, for $Z$ with Minkowski difference $V_0\supseteq V_1-V_2\supseteq\m$, such that $\check{f}_\loc:V_0\times V_2\to\B$ is well-defined and $\varphi$-holomorphic in $H$ and $Z$, and respectively $\check{f}_\loc(W-Z,Z)$ is well-defined on $V_1\times V_2$ and $\varphi$-holomorphic in both $Z$ and $W$.
	Notice that $\check{f}_\loc$ implicitly depends on $Z_0$, that is, on the small neighbourhood of definition $V_2$ of $Z$ (hence the notation).
	On the other hand, if $\nu\in\NN$ is any integer with the property that $\forall X\in\m: X^\nu = 0$, then we also have $\forall Z\in U\ \forall H\in\BB^\sp_\A(0,R(Z))\cap\A^\times$:
	\[
	\begin{aligned}
	\lim_{H\to X}\frac{f(Z+H)-f(Z)}{\varphi(H)} &= \lim_{H\to X} \sum_{k=1}^\infty \frac{f^{(k)}(Z)}{k!} \varphi(H)^{k-1} = 
	\sum_{k=0}^\nu \frac{f^{(k+1)}(Z)}{(k+1)!} \varphi(X)^k \eqqcolon \check{f}(X,Z),
	\end{aligned}
	\]
	since $\rho(X)=0<R$, which in turn is globally defined in $Z$ over $\m\times U$.
	Thus any $\check{f}_\loc$, when restricted to $\m=\nil\A$ in the first argument, extends globally in the second argument to $U$.
	In other words, $\check{f}:\m\times U\to\B$ is ambiently always a restriction of some $\check{f}_\loc$ $\varphi$-holomorphic in the \textit{first} (both) argument(s).
	
	To (2): First we show that for any $Z_0\in U$ the function $g(Z_0,W)$ is $\varphi$-holomorphic in $W$.
	It suffices to prove this in a neighbourhood of $Z_0+\m$.
	This follows directly from the proof of (1): $\check{f}_\loc(W-Z,Z)$ exists $\varphi$-holomorphically on an open product $V_1\times V_2\ni (W,Z)$ with $V_2\ni Z_0$ and $V_1\supseteq Z_0+\m$ and coincides with $\frac{f(W)-f(Z_0)}{\varphi(W-Z_0)}$ outside of $\{W-Z_0\in\m\}$. 
	Next, fix $W_0\in U$ instead. 
	The task at hand is similar, but slightly different in the way of choosing $V_1$ and $V_2$. 
	We need to show that 
	\[
	\begin{aligned}
	g(Z,W_0) &\defeq
	\begin{cases}
	\frac{f(W_0)-f(Z)}{\varphi(W_0-Z)}, &\text{if } W_0-Z\in\A^\times\cap U\\
	\check{f}(W_0-Z,Z), &\text{if } W_0-Z\in\m\cap U
	\end{cases}
	\defeq
	\begin{cases}
	\frac{f(W_0)-f(Z)}{\varphi(W_0-Z)}, &\text{if } W_0-Z\in\A^\times\cap U\\
	\sum_{k=0}^\infty\frac{f^{(k+1)}(Z)}{(k+1)!} \varphi(W_0-Z)^k, &\text{if } W_0-Z\in\m\cap U
	\end{cases}
	\end{aligned}
	\]
	is $\varphi$-holomorphic in $Z\in U$.
	Again, it suffices to show this for $Z$ in a small (cylindrical) neighbourhood of $W_0+\m$.
	For any $W_0\in U$ there exist small (cylindrical) neighbhourhoods $V_1\subseteq U$, $V_1\ni W_0$, and $V_2\subseteq U$, $V_2\supseteq V_1$, such that $\forall W\in V_1\ \forall Z\in V_2$:
	\[
	\check{f}_\loc(W-Z,Z)\defeq\sum_{k=0}^\infty\frac{f^{(k+1)}(Z)}{(k+1)!}\varphi(W-Z)^k
	\]
	is well-defined and $\varphi$-holomorphic in both $W$ and $Z$ and by the proof of (1) coincides with $\frac{f(W)-f(Z)}{\varphi(W-Z)}$ for $(W,Z)\in (V_1\times V_2) \setminus \{W-Z\in\m\}$. 
	Therefore $g(Z,W_0)$ is $\varphi$-holomorphic for $Z\in V_2\supseteq V_1\supseteq W_0+\m$.
\end{proof}


\begin{proposition}[Morera over $\A$]\label{morera}
	Let $U\subseteq\A$ open and path-connected and let $f\in\Cc^0(U,\B)$.
	\begin{enumerate}
		\item If for all simple rectifiable loops $\gamma$ in $U$
		\[
		\oint_\gamma f(Z)\d Z = 0,
		\]
		then $f\in\O_\varphi(U)$.
		
		\item If for all $\triangle\subseteq U$
		\[
		\oint_{\pd\triangle} f(Z)\d Z = 0,
		\]
		then $f\in\O_\varphi(U)$.
		
		\item If for all $\square\subseteq U$
		\[
		\oint_{\pd\square} f(Z)\d Z = 0,
		\]
		then $f\in\O_\varphi(U)$.
	\end{enumerate}
\end{proposition}

\begin{proof}
	To (1): by \cref{premorera} (1), $f$ has a primitive $F$, which is analytic by \cref{analyt}, hence so is $f$.
	
	To (2): Being finite-dimensional, $\A$ is a locally convex TVS. Therefore, $U$ can be covered by (star-)convex open sets, over which $f$ has local primitives by \cref{premorera}.
	
	To (3): by \cref{premorera} (3).
\end{proof}

In the same spirit we obtain a slight variation of \cref{closed} (2):

\begin{lemma}
	Let $U\subseteq\A$ be open and let $f\in\Cc^1(U,\B)$ such that $\omega\coloneqq f(Z)\d Z$ is $\d$-closed.
	Then $f\in\cO_\varphi(U)$.
\end{lemma}

\begin{proof}
	$\omega$ $\d$-closed $\Rightarrow$ $\omega$ locally exact, i.e. $f$ locally $\varphi$-integrable $\Rightarrow$ $f$ locally $\varphi$-analytic, i.e. $f\in\cO_\varphi(U)$.
\end{proof}

\begin{proposition}[Weierstraß Convergence Theorems over $\A$]
	Let $\cA \xrightarrow{\varphi} \cB$ be a general morphism of $\CC$-algebras, let $U\subseteq\A$ be open and path-connected, and let $(f_n)_{n\in\NN}\subset\O_\varphi(U)$.
	\begin{enumerate}
		\item If $f_n\to f$ locally uniformly\footnote{since $\A$ is locally compact, locally uniform convergence not only implies compact convergence, but is also equivalent to it;}, then $f\in\O_\varphi(U)$ and $f'_n\to f'$ locally uniformly.
		
		\item If $\sum_{k=1}^n f_k\to f$ locally normally\footnote{for the same reason, locally normal convergence not only implies compactly normal convergence, but is equivalent to it;}, then $f\in\O_\varphi(U)$ and $\sum_{k=1}^n f'_k\to f'$ also locally normally.
	\end{enumerate}
\end{proposition}

\begin{proof}
	To (1): follows completely analogously to the case of $\A$ being the complex plane itself.
	
	To (2): Fix $Z_0\in U$ and $\gamma\in\Cc^1_\pw(\SS^1,U)$ a $Z_0$-admissible loop with $\Ind_\varphi(\gamma,Z_0)=1$.
	Without loss of generality we can assume that $\norm{\cdot}_\B$ is unital.
	By translation with $-Z_0$ we can also assume without loss of generality that $Z_0=0$, which allows us to take $\gamma$ to be a scalar curve.
	Note, however, that for an arbitrary $Z_0$ the $\Ind_\varphi=1$-condition may prevent us in general from choosing $\gamma$ to be a scalar curve, since $\cA$ need not be a connected algebra. 
	Now, let for instance $\gamma(t)\coloneqq r e^{2\pi i t} 1_\A$, $t\in [0,1]$, for a sufficiently small $r>0$, in particular $L_\B(\gamma) = 2\pi r$.
	Fix $0<\vareps< r$.
	As before, there exists a neighbourhood $V\ni Z_0=0$ of $\gamma$-admissible points such that $\forall Z\in V: \Ind_\varphi(\gamma,Z)=1$ and $\norm{\varphi(Z)}<\vareps$ by continuity of $\varphi$.
	By \cref{norminverse} we have $\forall Z\in V$:
	\[
	\norm{\frac{1}{\gamma(t)-\varphi(Z)}}_\cB \leq \frac{1}{r-\norm{\varphi(Z)}_\cB}<\frac{1}{r-\vareps}.
	\]
	Thus by \hyperref[cauchyineq]{ Cauchy's inequality over $\A$ (\ref*{cauchyineq})} we obtain $\forall Z\in V$:
	\[
	\norm{f'_n(Z)}_\cB \leq \frac{L_\B(\gamma)}{2\pi}\norm{f_n}_{\cB,\gamma} \sup\limits_{t\in I} \norm{\frac{1}{\gamma(t)-Z_0}}_\cB \leq \frac{r}{r-\vareps} \norm{f_n}_{\cB,\gamma},
	\]
	which proves the claim.
\end{proof}

\subsection{The Homological Cauchy's Integral Formula over $\cA$}\label{subs:homol}


\begin{definition}
	Let $\cA \cong \bigoplus_{k=1}^M (\cA_k,\fm_k)$ be a decomposition of $\cA$ into local Artin $\CC$-algebras, let $U\subseteq\cA$ be open and path-connected, and let $I\subset \{1,\dots,M\}$ be an index subset.
	\begin{enumerate}
		\item $I$-restricted spectral 1-homology: $B_1^\sp(U,\ZZ)_I \coloneqq \bigcap_{k\in I} (\sigma_k)_\#^{-1}(B_1(\sigma_k(U),\ZZ))$, $H_1^\sp(U,\ZZ)_I \coloneqq Z_1(U,\ZZ) / B_1^\sp(U,\ZZ)_I$.
		
		\item Spectral 1-homology associated to $\varphi$: if $\varphi:\cA \to \cB$ is a morphism of $\CC$-algebras, then $B_1^\varphi(U,\ZZ) \coloneqq B_1^\sp(U,\ZZ)_{\im\tau_\varphi}$ and $H_1^\varphi(U,\ZZ) \coloneqq H_1^\sp(U,\ZZ)_{\im\tau_\varphi}$.
	\end{enumerate}
\end{definition}

\begin{theorem}[Homological Cauchy Integral Formula over $\A$]\label{homolCIT}
	Let $U\subseteq\A$ be open and path-connected, let $f\in\O_\varphi(U)$, and $\Gamma\in Z_1(U,\ZZ)$.
	If $[\Gamma] = 0$ in $H^\varphi_1(U,\ZZ)$, then $\forall Z\in U\cap\Adm(\Gamma)$:
	\begin{equation}
	f(Z)\Ind_\varphi(\Gamma,Z)=\frac{1}{2\pi i}\int_\Gamma\frac{f(W)}{\varphi(W-Z)}\d W.
	\end{equation}
	In particular
	\begin{equation}
	\int_\Gamma f(W)\d W = 0.
	\end{equation}
\end{theorem}

\begin{proof}
	We first prove the claim for a local morphism $(\A,\m)\xrightarrow{\varphi} (\B,\m)$, which then easily translates to a general $\A\xrightarrow{\varphi}\B$.
	We adapt the well-known proof of Dixon to our setting.
	Extend $f:U\to\B$ to $f:\wtilde{U}\to\B$ and define $g:\wtilde{U}\times\wtilde{U}\to\B$,
	\[
	g(Z,W)\coloneqq
	\begin{cases}
	\frac{f(W)-f(Z)}{\varphi(W-Z)}, &\text{if } W-Z\in\wtilde{U}\cap\A^\times\\
	\check{f}(W-Z,Z), &\text{if } W-Z\in\wtilde{U}\cap\m.
	\end{cases} 
	\]
	We note that $\wtilde{U}\supseteq\fF(\Gamma)\defeq\Gamma^\sp\times\m$ and
	\[
	\wtilde{U}\cap\m\defeq
	\begin{cases}
	\m, &\text{if } \sigma_\cA(U)\ni 0\\
	\emptyset, &\text{otherwise}.
	\end{cases}
	\]
	By \cref{nilanal} $g$ is $\varphi$-holomorphic in $Z$ and $W$.
	Now put
	\[
	h(Z)\coloneqq \frac{1}{2\pi i} \int_\Gamma g(Z,W)\d W.
	\]
	Then $h\in\O_\varphi(\wtilde{U})$ as a parameter-integral over the compact $\Gamma$.
	We have $\forall Z\in\wtilde{U}\cap\Adm(\Gamma)$:
	\[
	h(Z) = \frac{1}{2\pi i}\int_{\Gamma}\frac{f(W)-f(Z)}{\varphi(W-Z)}\d W = \frac{1}{2\pi i} \int_{\Gamma} \frac{f(W)}{\varphi(W-Z)}\d W-f(Z) \ind_\CC(\Gamma^\sp,z),
	\]
	where recall that $z \defeq \sigma_\cA(Z)$ and $\ind_\CC(\Gamma^\sp,z) = \Ind_\varphi(\Gamma,Z)$.
	We are going to show that $h=0$.
	By hypothesis, we have $\forall z\in\sigma_\cA(U)^c: \ind_\CC(\Gamma^\sp,z) = 0$.
	Therefore $h\in\O_{\varphi}(\wtilde{U})$ extends over $\wtilde{U}^c = \sigma_\cA(U)^c\times\m \subseteq \Adm(\Gamma)$ to a $\varphi$-entire function via
	\[
	h(Z)\coloneqq\frac{1}{2\pi i} \int_{\Gamma}\frac{f(W)}{\varphi(W-Z)}\d W,\ Z\in \wtilde{U}^c.
	\]
	Now consider $h|_\CC$. 
	Since $\Gamma\simeq\Gamma^\sp$ in $\wtilde{U}$, we obtain for $z,w\in\CC$:
	\[
	h(z) = \frac{1}{2\pi i} \int_{\Gamma^\sp}\frac{f(w)}{w-z}\d w\to 0\ \text{as } \abs{z}\to 0
	\]
	by the usual estimate, where notice that $h(z)$ is simply a ($\CC$-)entire function taking values in $\B$. 
	Thus by \hyperref[liouville]{Liouville (\cref*{liouville})} we have $h=\const=0$.
	In other words, $\forall Z\in\wtilde{U}\cap\Adm(\Gamma)$:
	\[
	0 = h(Z) = \frac{1}{2\pi i}\int_\Gamma\frac{f(W)}{\varphi(W-Z)}\d W - f(Z)\Ind_\varphi(\Gamma,Z) 
	\]
	as desired.
	Finally, using \cref{morphind}, \cref{indphi}, and the definition of $B_1^\varphi(U,\ZZ)$, the result easily generalizes to an arbitrary $\cA\xrightarrow{\varphi}\cB$ in an analogous manner as in the proof of \cref{cif}.
\end{proof}

\begin{remarks}\ 
	\begin{enumerate}[topsep=-\parskip]
		\item Passing to $\wtilde{U}$ in the proof allows to contain $\fF(\Gamma)$ entirely in the initial domain of definition of $h$, which makes it easier to exhibit the holomorphic continuation of $h$ to $\wtilde{U}^c$.
		If we didn't extend $f$ to $\wtilde{U}$, we would have had to prove that
		\[
		h(Z)\coloneqq
		\begin{cases}
		\frac{1}{2\pi i}\int_{\Gamma}\frac{f(W)}{\varphi(W-Z)}\d W, &\text{if } Z\in \wtilde{U}^c\cap\Adm(\Gamma)\\
		0, &\text{if } Z\in \wtilde{U}^c\cap\fF(\Gamma)
		\end{cases}
		\]
		gives the analytic continuation of $h$, for example by using the somewhat cumbersome $\check{f}_\loc$-yoga.
		
		\item A similar result has been shown by Giovanni Battista Rizza (1952) in \cite{Rizza}.
		As of writing this, we have not been able to obtain his paper, however the Zentralblatt review indicates that his proof uses much stronger topological assumptions.
		Namely, in our language his main statement reads as follows: let $\A=(\A,\m)$ be local, $U\subseteq\A$ open and path-connected, $f\in\O_\A(U)$, and $\Gamma\in Z_1(U,\ZZ)$ with $[\Gamma]=0$ in $H_1(U,\ZZ)$. 
		Then $\exists N\in\NN\ \forall Z\in U\cap\Adm(\Gamma)$:
		\[
		f(Z)2\pi i N = \int_\Gamma \frac{f(W)}{W-Z}\d W.
		\]
		In comparison, we only require $[\Gamma^\sp] = 0$ in $H_1(\sigma_\cA(U),\ZZ)$.
		
		\item In order to illustrate the difference between the condition $[\Gamma]=0$ in $H_1(U,\ZZ)$ and the condition $\forall 1\leq k\leq N: [\Gamma^\sp_k] = 0$ in $H_1(\sigma_k(U),\ZZ)$, we give the following example\footnote{see \cite{ProjHomol}}: consider the set $U\coloneqq(\CC^2\setminus\{z_1^3+z_2^2=0\})\cap B_r(0)$, where $B_r(0)\subseteq\CC^2$ is an open ball of radius $r>0$ around the origin.
		$U$ is known to be diffeomorphic to $(\SS^3\setminus K)\times (0,r)$, where $K$ denotes the trefoil knot.
		Then $H_1(U,\ZZ) = \ZZ$, but $H_1(p_1(U),\ZZ) = H_1(p_2(U),\ZZ) = 0$, where $p_{1,2}:\CC^2\to\CC$ are the canonical projections, since $p_{1,2}(U)$ are open disks in $\CC$ of radius $r$.
	\end{enumerate}
\end{remarks}

It is now evident that the right choice of index-based 1-homology in the theory of $\varphi$-holomorphic functions is given by ``spectral'' $\varphi$-1-Homology $H^\varphi_1(U,\ZZ)$:

\begin{corollary}
	Let $U\subseteq\A$ be open and path-connected. We have a well-defined $\ZZ$-$\B$-bilinear pairing
	\begin{equation}
	\beta: H^\varphi_1(U,\ZZ) \times \O_{\varphi}(U) \to \B,\ ([\Gamma],f) \mapsto \int_\Gamma f(Z) \d Z
	\end{equation}
	between $\varphi$-1-homology and $\varphi$-holomorphic functions. \qed
\end{corollary}

\begin{proposition}[Homological Cauchy Differentiation Formula over $\A$]
	\label{homolcauchydiff}
	Let $U\subseteq\A$ be open and path-connected and let $f\in\O_\varphi(U)$.
	Then for every open polydisk $\Delta\coloneqq\Delta_r(Z_0)\subseteq U$ with center $Z_0\in U$ and every $Z_0$-admissible 1-cycle $\Gamma\in B^\sp_1(\Delta,\ZZ)$ there exists an open neighbourhood $V\ni Z_0$, $V\subseteq\Delta$, such that $\forall Z\in V$:
	\begin{equation}\label{homolcifder}
	f^{(k)}(Z) \Ind_\varphi(\Gamma,Z) = \frac{k!}{2\pi i} \int_\Gamma \frac{f(W)}{\varphi(W-Z)^{k+1}}\d W,\ k\in\NN_0.
	\end{equation}
\end{proposition}

\begin{proof}
	Same as the proof of \cref{analyt}.
\end{proof}

\begin{corollary}[Cauchy's Inequality over $\A$ for 1-Cycles]
	Let $U\subseteq\A$ be open and path-connected, $Z_0\in U$, $f\in\O_\varphi(U)$, and let $\Gamma\coloneqq\sum_{j=1}^m n_j \gamma_j\in B^\sp_1(U,\ZZ)$, $\gamma_i\neq\gamma_j\ \forall i\neq j$, be a $Z_0$-admissible 1-cycle with $\Ind_\varphi(\Gamma,Z_0)=1$. 
	Then $\forall k\in\NN_0:$
	\begin{equation}
	\norm{f^{(k)}(Z_0)}_\cB \leq \frac{k!}{2\pi}\norm{f}_{\cB,\Gamma} L_\B(\Gamma)\max_{1\leq j\leq m} \norm{\frac{1}{\varphi(\gamma_j(t))-\varphi(Z_0)}}_{\cB,I}^{k+1},
	\end{equation}
	where $L_\B(\Gamma)\coloneqq\sum_{j=1}^m \abs{n_j} L_\B(\gamma_j)$ is the total length of the 1-cycle $\Gamma$.
\end{corollary}

\begin{proof}
	By \cref{homolcauchydiff} we have
	\[
	\begin{aligned}
	\norm{f^{(k)}(Z_0)}_\cB &\leq \frac{k!}{2\pi}\sum_{j=1}^m\abs{n_j}\int_0^1 \norm{f(\gamma_j(t))}_\cB \norm{\varphi(\gamma_j'(t))}_\cB \norm{\frac{1}{\varphi(\gamma_j(t)-Z_0)}}_\cB^{k+1}\d t\\
	&\leq \frac{k!}{2\pi} \norm{f}_{\cB,\Gamma} \bigg(\sum_{j=1}^m \abs{n_j} L_\B(\gamma_j)\bigg)\max_{1\leq j\leq m}\norm{\frac{1}{\varphi(\gamma_j(t))-\varphi(Z_0)}}_{\cB,I}^{k+1},
	\end{aligned}
	\]
	since $\forall t\in I: \norm{f(\gamma_j(t))}_\cB\leq\norm{f}_{\cB,\Gamma}$.
\end{proof}

	
	\section*{Acknowledgements}
	
	I would like to thank Prof. Emil Horozov for giving me the opportunity and freedom to work and present on this topic, for his moral support and the numerous conversations concerning different fields of mathematics.
	I want to thank Prof. Azniv Kasparian for her patience, for the countless discussions on algebra, topology, and geometry, as well as for proof-reading various parts and suggesting several improvements of the present manuscript.
	Finally, I wish to thank Prof. Stefan Ivanov for his moral support and the various conversations on differential-geometric issues directly or indirectly related to this topic.
	All mistakes are mine and mine only.
	
	
	\begin{refsegment}
		\mklocalfilter{nocitelocal}
		\nocite{*}
		\printbibliography[keyword={complex analysis},filter=nocitelocal,title={References on Standard Complex Analysis Stuff}]	
	\end{refsegment}

	\begin{refsegment}
		\mklocalfilter{nocitelocal}
		\nocite{*}
		\printbibliography[keyword={A-holomorphy},filter=nocitelocal,title={References on $\cA$-differentiability}]	
	\end{refsegment}

	\vspace{-\baselineskip} 

	\printbibliography[notkeyword={complex analysis},notkeyword={A-holomorphy},segment=0,title={Other References}]
	
\end{document}